\documentclass[11pt, reqno]{amsart}
\usepackage[left=2.3cm, top=2.3cm,bottom=2.3cm,right=2.3cm]{geometry}
\pdfoutput=1
\usepackage{mathtools,amssymb,amsthm,mathrsfs,calc,graphicx,stmaryrd,xcolor,dsfont,tikz,pgfplots,bbm,float,array,epsfig,hyperref,color}
\usepackage{tikz-3dplot}
\usepackage{caption}
\usepackage{subcaption}
\usepackage[numbers,square]{natbib}
\usepackage[british]{babel}
\usepackage{amsfonts}              % for blackboard bold, etc
\usetikzlibrary{calc}

\definecolor{navy}{rgb}{0.0,0.0,0.5} % define 'navy' as a dark blue
\usepackage[scr=rsfso,cal=zapfc,frak=euler,bb=ams]{mathalfa}
\numberwithin{equation}{section}
\numberwithin{figure}{section}

\newtheorem {theorem}{Theorem}[section]
\newtheorem {proposition}[theorem]{Proposition}
\newtheorem {lemma}[theorem]{Lemma}
\newtheorem {corollary}[theorem]{Corollary}
\newtheorem {conjecture}[theorem]{Conjecture}

{\theoremstyle{definition}
\newtheorem{definition}[theorem]{Definition}
\newtheorem*{convention}{Convention}

\newtheorem {example}[theorem]{Example}
\newtheorem {remark}[theorem]{Remark}

}

{
\theoremstyle{remark}
}

\newcommand{\Cones}{\mathrm{CCC}} %% space of closed convex cones
\newcommand{\PolyCones}{\mathrm{PolyCones}} %% space of polyhedral convex cones
\newcommand{\Simpl}{\mathrm{Simpl}} %% space of simplices

\newcommand{\Beta}{\text{\rm Beta}}

\newcommand{\BetaCone}{\operatorname{BetaCone}}
\newcommand{\BetaPoly}{\operatorname{BetaPoly}}

\newcommand{\Vol}{\operatorname{Vol}}
\newcommand{\Content}{\operatorname{Cont}}
\newcommand{\intern}{\operatorname{Int}}
\newcommand{\extern}{\operatorname{Ext}}

\newcommand{\skel}{\operatorname{skel}}

\DeclareMathOperator*{\argmin}{arg\,min}
\renewcommand{\Re}{\operatorname{Re}}  %Realteil
\newcommand{\ii}{{\rm{i}}}

\def\ba{\begin{array}}
\def\ea{\end{array}}
\def\bea{\begin{eqnarray} \label}
\def\eea{\end{eqnarray}}
\def\be{\begin{equation} \label}
\def\ee{\end{equation}}
\def\bit{\begin{itemize}}
\def\eit{\end{itemize}}
\def\ben{\begin{enumerate}}
\def\een{\end{enumerate}}

\def\bI{\mathbb{I}}
\def\bJ{\mathbb{J}}

\def\BB{\mathbb{B}}

\def\E{\mathbb{E}}

\def\N{\mathbb{N}}
\def\P{\mathbb{P}}

\def\R{\mathbb{R}}
\def\Z{\mathbb{Z}}
\def\C{\mathbb{C}}
\def\RRd1{\mathbb{R}^{d+1}}
\def\SS{\mathbb{S}}

\def\bS{\mathbb{S}}

\def\a{\alpha}
\def\b{\beta}
\def\g{\gamma}

\def\k{\kappa}
\def\l{\lambda}

\newcommand{\linsp}{\mathop{\mathrm{linsp}}\nolimits}  %lineality space

\def\cC{\mathcal{C}}
\def\cD{\mathcal{D}}

\def\cF{\mathscr{F}}

\def\sC{\mathscr{C}}
\def\sP{\mathscr{P}}
\def\sD{\mathscr{D}}
\def\sG{\mathscr{G}}
\def\sW{\mathscr{W}}

\def\dint{\textup{d}}

\newcommand{\eee}{{\rm e}}

\newcommand{\ind}{\mathbbm{1}}
\newcommand{\eps}{\varepsilon}
\newcommand{\pos}{\mathop{\mathrm{pos}}\nolimits}

\newcommand{\aff}{\mathop{\mathrm{aff}}\nolimits}
\newcommand{\lin}{\mathop{\mathrm{lin}}\nolimits}
\newcommand{\conv}{\mathop{\mathrm{conv}}\nolimits}

\newcommand{\dd}{{\rm d}}
\newcommand{\bsl}{\backslash}
\DeclareMathOperator{\relint}{relint}

\makeatletter
\@namedef{subjclassname@2020}{\textup{2020} Mathematics Subject Classification}
\makeatother

\begin{document}

\title[Beta Polytopes and Beta Cones]{Beta Polytopes and Beta Cones: \\ An Exactly Solvable Model in Geometric Probability}

\author{Zakhar Kabluchko and David Albert Steigenberger}

\date{}


  %%%ZK: I checked this.

\begin{abstract}

Let $X_1,\ldots, X_n$ be independent random points in the unit ball of $\R^d$ such that $X_i$ follows a beta distribution with the density proportional to $(1-\|x\|^2)^{\beta_i}\mathbbm{1}_{\{\|x\| <1\}}$.  Here, $\beta_1,\ldots, \beta_n> -1$ are parameters.
We study random polytopes of the form $[X_1,\ldots,X_n]$, called beta polytopes. We determine explicitly expected values of several functionals of these polytopes including the number of $k$-dimensional faces, the volume, the intrinsic volumes, the total $k$-volume of the $k$-skeleton, various angle sums, and the $S$-functional which generalizes and unifies many of the above examples. We identify and study the central object needed to analyze beta polytopes: beta cones.
For these, we determine explicitly expected values of several functionals including the solid angle, conic intrinsic volumes and the number of $k$-dimensional faces.  We identify expected conic intrinsic volumes of beta cones as a crucial quantity needed to express all the functionals mentioned above. We obtain a formula for these expected conic intrinsic volumes  in terms of a  function $\Theta$ for which we provide an explicit integral representation. The proofs combine methods from integral and stochastic geometry with the study of the analytic properties of  the function $\Theta$.
\end{abstract}

\maketitle

\tableofcontents

\section{Introduction}

\subsection{Summary of main results}
In the present paper we study problems of geometric probability related to the $d$-dimensional \emph{beta distribution}. This rotationally invariant probability distribution on the $d$-dimensional unit ball $\BB^d:= \{x\in \R^d: \|x\| \leq 1\}$ is defined via its Lebesgue density
\begin{align}\label{eq:def_f_beta}
f_{d,\b}(x) = c_{d,\b}(1-\|x \|)^\b \ind_{\{\|x\|  < 1\}}
\qquad
\text{ with }
\qquad
c_{d,\b} = \frac{\Gamma\left(\frac{d}{2}+\b+1\right)}{\Gamma\left(\b+1\right)\pi^{d/2}}.
\end{align}
Here, $\beta >-1$ is a parameter and $\|x\|$ is the Euclidean norm of $x\in\R^d$. If a random point $X$ in $\BB^d$ follows this distribution, we write $X\sim f_{d,\beta}$.
For instance, the beta distribution with $\beta=0$ is the uniform distribution on the ball $\BB^d$. The weak limit of the beta distribution as $\beta \downarrow -1$ is the uniform distribution on the unit sphere $\bS^{d-1}:= \{x\in \R^d: \|x\| = 1\}$.  For this reason, we say that a random point $X$ follows a beta distribution with parameter $\beta = -1$ if $X$ is uniformly distributed on $\bS^{d-1}$ and write $X\sim f_{d,-1}$ even though $f_{d,-1}$ does not exist as a probability density.
The beta distribution has been introduced into geometric probability by \citet{miles} and~\citet{ruben_miles}. Their results are reviewed in~\cite{kabluchko_steigenberger_thaele}; see also~\cite{bonnet_etal,BonnetEtAlThresholds,bonnet_kabluchko_turchi,gusakova_kabluchko_sylvester_beta,kabluchko_angle_sums_dim_3_4,kabluchko_angles_of_random_simplices_and_face_numbers,kabluchko_recursive_scheme,kabluchko_on_expected_face_numbers_of_beta_polytopes,beta_polytopes,kabluchko_thaele_zaporozhets_beta_polytopes_poi_polyhedra} for further papers on stochastic geometry where the beta distributions play a major role.

\subsubsection{Beta polytopes}
The main object of investigation in the present paper are random beta polytopes defined in the following way. Let $X_1,\dots,X_n$ be independent random  points in $\BB^d$, where each $X_i$ follows a beta-distribution $f_{d,\beta_i}$ with its own parameter $\beta_i\geq -1$. The convex hull of these points, denoted by
$$
\sP:= \sP_{n,d}^{\beta_1,\ldots, \beta_n} := [X_1,\dots,X_n] := \{\lambda_1X_1+\ldots + \lambda_n X_n: \lambda_1\geq 0,\ldots, \lambda_n\geq 0, \lambda_1+\ldots + \lambda_n = 1\},
$$
is called a \emph{beta polytope}.
We shall also express the fact that $\sP$ is  a beta polytope in the following way,
$$
\sP \sim \BetaPoly(\R^d; \beta_1,\ldots,  \beta_n).
$$
In the course of this work, we shall give explicit formulas for expected values of the following functionals of beta polytopes:
\begin{itemize}
\item $f_{k}(\sP)$, the number of $k$-dimensional faces of $\sP$, for all $k\in \{0,\ldots, d-1\}$;
\item $\Vol_d(\sP)$, the volume of $\sP$;
\item $V_k(\sP)$, the $k$-th intrinsic volume of $\sP$;
\item the total $k$-volume of the $k$-dimensional skeleton of $\sP$ and its $L^p$-generalization;
\item the sum of internal and the sum of external angles at all $k$-dimensional faces of $\sP$;
\item the sum of conic intrinsic volumes of tangent cones of $\sP$ at its $k$-dimensional faces;
\item the beta content, that is the probability that an additional independent beta-distributed point $X\sim f_{d,\beta}$ falls into $\sP$;
\item the $S$-functional, to be introduced in Section~\ref{subsec:S_functional}, which  unifies and generalizes  many of the above mentioned examples.
\end{itemize}
We shall also compute the probability that a beta polytope $\sP_{d+2,d}^{\beta_1,\ldots, \beta_{d+2}}$ generated by $n=d+2$ points is a simplex, thus solving an analogue of the Sylvester problem for beta-distributed points with possibly different parameters. All our formulas are \emph{exact} rather than asymptotic (that is, they are valid for fixed $n$ and $d$).

\subsubsection{Beta cones}
In order to prove the results listed above, we shall identify and study the central object needed to analyze beta polytopes: beta cones. These are random polyhedral cones defined as follows.
Let $Z, Z_1,\ldots, Z_n$ be independent random vectors in $\R^d$ with $Z\sim f_{d, \b}$ and $Z_i\sim f_{d,\b_i}$ for all $i\in \{1,\ldots, n\}$, where $\beta_1,\ldots, \beta_n\geq -1$ are parameters. A \emph{beta cone} is defined as the positive hull of the vectors $Z_1-Z,\ldots, Z_n-Z$,
\begin{align}\label{eq:beta_cones_intro}
\sC
\coloneqq \sC_{n,d}^{\beta; \beta_1,\ldots, \beta_n}
&\coloneqq \pos (Z_1-Z,\dots,Z_n-Z) \notag
\\
&\coloneqq \{\lambda_1(Z_1-Z)+\ldots + \lambda_n (Z_n-Z): \lambda_1\geq 0,\ldots, \lambda_n\geq 0\}.
\end{align}
For visualization, see Figure~\ref{fig:construction_of_cones}. Similarly to the notation used for beta polytopes, we write
	\begin{align*}
	\sC \sim	\BetaCone (\R^d; \b; \b_1,\dots,\b_n).
	\end{align*}
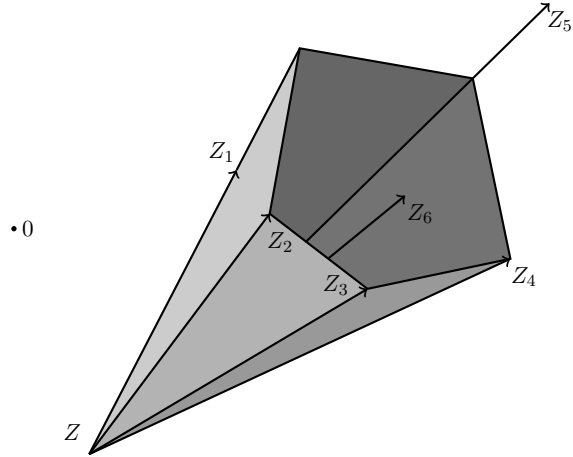
\begin{figure}[h!]\label{fig:beta_cone_realization}
	\centering
  		\begin{tikzpicture}[scale=2]

	%--------------------------------------------------------------------------
	% 1) Define the main vectors and points
	%--------------------------------------------------------------------------
	\coordinate (Z)  at (0,0);
	\coordinate (Z1) at (1.4,2.7);
	\coordinate (Z1real) at ($0.7*(Z1)$);
	\coordinate (Z2) at ($(1.2,1.6)$);
	\coordinate (Z2real) at ($(1.2,1.36)+(-0.0,0.0)$);
	\coordinate (Z3) at ($(2.55,2.5)+(-0.0,-0.0)$);
	\coordinate (Z3real) at ($1.2*(Z3)+(-0.0,-0.0)$);
	\coordinate (Z3thick) at ($0.564*(Z3)+(-0.0,-0.0)$);
	\coordinate (Z4) at (2.8,1.3);
	\coordinate (Z5) at (1.85,1.1);
	\coordinate (0)  at (-0.5,1.5);
	\coordinate (Z6) at ($(2.0,1.2)+(0.1,0.52)$);
	\coordinate (Z6thick) at ($0.756*(Z6)$);
	% 3) Draw axes, etc.
	%--------------------------------------------------------------------------
	%\draw[->, thick] (-4,0) -- (3.2,0);   % x-axis
	%\draw[->, thick] (0,-3.2) -- (0,4);   % y-axis
	\filldraw[black] (0) circle (0.4pt); % origin
	
	%--------------------------------------------------------------------------
	%--------------------------------------------------------------------------
	% 5) Polygon edges & lines from origin
	%--------------------------------------------------------------------------
	\draw[thick, black] (Z1) -- (Z2) -- (Z5) -- (Z4) -- (Z3) -- cycle;
	%\fill[black, opacity=0.6] (Z1) -- (Z2) -- (Z5) -- (Z4) -- (Z3) -- cycle;
	\fill[black, opacity=0.2] (Z1) -- (Z2) -- (Z) -- cycle;
	\fill[black, opacity=0.6] (Z1) -- (Z2) -- (Z3thick) -- (Z3) -- cycle;
	\fill[black, opacity=0.55] (Z5) -- (Z4) -- (Z3) -- (Z3thick) -- cycle;
	%\fill[black, opacity=0.1] (Z1) -- (Z3) -- (Z) -- cycle;
	\fill[black, opacity=0.3] (Z2) -- (Z5) -- (Z) -- cycle;
	\fill[black, opacity=0.4] (Z4) -- (Z5) -- (Z) -- cycle;
	%\fill[black, opacity=0.1] (Z3) -- (Z4) -- (Z) -- cycle;

	\draw[->,thick] (Z) -- (Z2);
	%\draw[->, thick] (Z) -- (Z2real);
	\draw[->, thick] (Z3thick) -- (Z3real);
	\draw[->,thick] (Z) -- (Z1real);
	\draw[thick] (Z1real) -- (Z1);
	%\draw[dotted,thick] (Z) -- (Z3thick);
	\draw[->,thick] (Z) -- (Z4);
	\draw[->,thick] (Z) -- (Z5);
	%\draw[dotted,thick] (Z) -- (Z6);
	\draw[->,thick] (Z6thick) -- (Z6);
	%--------------------------------------------------------------------------
	% 8) Labels
	%--------------------------------------------------------------------------
	%\draw[help lines,step=0.1] (-5,-5) grid (5,5);
	%help lines to get coordinates
	\node[above, scale=0.8] at ($(Z1real)+(-0.1,0)$) {$Z_1$};
	\node[below right, scale=0.8]  at ($(Z2)+(-0.07,-0.07)$) {$Z_2$};
	\node[below right, scale=0.8] at ($(Z3real)+(-0.07,0.0)$) {$Z_5$};
	\node[below right, scale=0.8] at ($(Z4)+(-0.05,0)$) {$Z_4$};
	\node[left, scale=0.8] at ($(Z5)+(-0.07,0.025)$) {$Z_3$};
	\node[below, scale=0.8] at ($(Z6)+(0.1,0)$) {$Z_6$};
	\node[left, scale=0.8] at (0,0.15) {$Z$};
	\node[right, scale=0.8] at (0) {$0$};
  		\end{tikzpicture}
  	\caption{A possible realization of $\sC \sim \BetaCone(\R^3; \b;\b_1,\dots,\b_6)$. Actually, the figure shows $\sC+Z$, a beta cone shifted by the vector $Z$. The apex of the original beta cone is at $0$.}
  	\label{fig:construction_of_cones}
\end{figure} \\
As we shall show, the tangent cones of beta polytopes at their faces are beta cones (after factoring out the lineality space) - this allows to reduce questions about beta polytopes to questions about beta cones. We shall compute the probability of the event that $\lbrace \sC  \neq \R^d \rbrace$ -- equivalently, this is the probability that $Z$ falls outside of the convex hull $[Z_1,\ldots, Z_n]$. Moreover, we shall compute the expected values of the following functionals of beta cones:
\begin{itemize}
\item the solid angle of $\sC$ and its dual cone $\sC^\circ$;
\item more generally, the conic intrinsic volumes $\upsilon_0(\sC), \ldots, \upsilon_d(\sC)$ of $\sC$;
\item $f_k(\sC)$, the number of $k$-dimensional faces of $\sC$, for all $k\in \{0,\ldots, d-1\}$;
\item angles of tangent and normal cones at all faces of $\sC$;
\item more generally, conic intrinsic volumes of the tangent cones of $\sC$ at its $k$-dimensional faces.
\end{itemize}

\subsection{A collection of results}
The expected values of the functionals listed above will be determined in terms of a special function $\Theta$ which we are now going to introduce. First of all, as a special case of the normalizing constant in~\eqref{eq:def_f_beta}, define
$$
c_{\beta} := c_{1,\beta} = \frac{\Gamma(\beta + \frac 32)}{\sqrt \pi \; \Gamma(\beta+1)}.
$$
Let  $Y= \{y_1,\dots, y_\ell\}$ and $Z=\{z_1,\dots, z_k\}$ be multisets of non-negative numbers which might be empty.  A multiset is a set which might have repeated elements -- this means that we do not require the $y_i$'s and the $z_j$'s to be pairwise different.   For $\gamma \geq -1/2$ we define the function
\begin{multline*}
\Theta(\gamma;Y;Z)
:=
\Biggl( \int_{-1}^{+1} c_{\a} (1-t^2)^{\alpha} \prod_{j=1}^\ell \left(\frac{1}{2} + \ii \int_{0}^{t} c_{y_j -\frac{1}{2}} (1-s^2)^{-y_j -1} \dd s \right) \dd t \Biggr)  \\
		\times \Biggl( \int_{-1}^{+1} c_{\a-\frac{1}{2}} (1-t^2)^{\alpha -\frac{1}{2}}
		\prod_{j=1}^k \left(\frac 12 + \int_{0}^t c_{z_j - \frac{1}{2}} (1-s^2)^{z_j - \frac{1}{2}}\,\dd s\right) \,\dd t \Biggr),
\end{multline*}
where $\ii = \sqrt{-1}$ and $\alpha = \gamma + y_1+\dots + y_\ell$. The factors in the definition of $\Theta$ can be expressed in terms of  special functions $a$ and $b$ which will be introduced and studied in Section~\ref{sec:beta_cones_explicit_internal_angles}.

In the next theorem we collect some of our results on beta cones. 
For an index set $I= \{i_1,\ldots, i_k\} \subseteq \{1,\ldots, n\}$ we let $I^c:= \{1,\ldots, n\}\bsl I$ denote its complement. If $\gamma_1,\ldots, \gamma_n$ are real numbers, then  $\{\gamma_i: i\in I\}$  denotes the multiset  $\{\lambda_{i_1},\ldots,\lambda_{i_k}\}$.

\begin{theorem}[Selected results on beta cones]\label{theo:beta_cones_for_intro}
Let $\sC = \sC_{n,d}^{\beta; \beta_1,\ldots, \beta_n}\sim \BetaCone (\R^d; \b; \b_1,\dots,\b_n)$ be a beta cone in $\R^d$ with $n\geq d\geq 2$ and  $\beta_1,\ldots, \beta_n\geq -1$. Define $\g_i = \b_i +d/2$ for $i=1,\ldots, n$.  Then,
\begin{align}
\P [\sC=\R^d]
&=
2\sum_{\substack{I \subseteq \{1,\dots,n\} \\ \# I \in \lbrace d+1,d+3,\ldots \rbrace}}\Theta ( \gamma;\left\{ \g_i : i\in I\right\};\left\{ \g_i : i\in I^c\right\}),
\\
\P [\sC\neq \R^d]
&=
 2\sum_{\substack{I \subseteq \{1,\dots,n\} \\ \# I \in \lbrace d-1,d-3,\ldots \rbrace}} \Theta ( \gamma ;\left\{ \g_i : i\in I\right\};\left\{ \g_i : i\in I^c\right\}).
\end{align}
Let $\alpha(\sC):= \Vol_d (\sC\cap \BB^d)/\Vol_d(\BB^d)$ denote the solid angle of $\sC$. Then,
\begin{align}
\E \left[\alpha(\sC)\right]
&= 1- \sum_{\substack{I \subseteq \{1,\dots,n\} \\ \# I\leq d-1}}\Theta ( \gamma ;\left\{ \g_i : i\in I\right\};\left\{ \g_i : i\in I^c\right\}),\label{eq:expect_upsilon_d_beta_cone_1_intro}\\
\E \left[\alpha(\sC)\ind_{\{\sC\neq \R^d\}}\right]
&= \sum_{\substack{I \subseteq \{1,\dots,n\} \\ \# I\leq d-1}} (-1)^{d-1-\# I}\Theta ( \gamma ;\left\{ \g_i : i\in I\right\};\left\{ \g_i : i\in I^c\right\}). \label{eq:upislon_R_d_ind_BetaCone_intro}
\end{align}
For every $k\in \{0,\ldots, d-1\}$, the expected number of $k$-dimensional faces of $\sC$  is given by
\begin{equation}
\E f_k(\sC) = 2 \sum_{\substack{K\subseteq \{1,\ldots, n\}\\ \# K = k}}
\sum_{\substack{I \subseteq K^c \\ \# I \in \lbrace d-k-1,d-k-3,\ldots \rbrace}}
		\Theta \left(\g  + \sum_{i \in K } \g_i  ;\left\{ \g_i : i\in I\right\}; \left\{\g_{i}: i\in K^c\bsl I\right\}  \right).
\end{equation}
\end{theorem}
Proofs of these and further results on beta cones will be given in Section~\ref{sec:beta_cones_results}.  In the next theorem we collect some results on beta polytopes.

\begin{theorem}[Selected results on beta polytopes]
Consider a beta polytope
$
\sP = \sP_{n,d}^{\beta_1,\ldots,\beta_n} \sim \BetaPoly(\R^d; \beta_1,\ldots,  \beta_n)
$
with $d\in \N$, $n \geq d+1$ and $\beta_1,\dots, \beta_n\geq -1$ (with a strict inequality if $d=1$).  Define $\g_i = \b_i +d/2$ for $i=1,\ldots, n$.
The expected volume of $\sP$ is given by
\begin{align*}
\E \Vol_d(\sP)
= 2\kappa_d \cdot  \sum_{\substack{I \subseteq \{1,\ldots, n\} \\ \# I  = d+1}} \Theta \left(\frac{d}{2};\{\g_i: i \in I\};\{\g_i: i \in I^c\}\right),
\end{align*}
where $\kappa_d := \Vol_d(\BB^d) = \pi^{d/2}/\Gamma(\frac{d}{2}+1)$ is the volume of the $d$-dimensional unit ball.
More generally, for every $\ell\in \{0,\ldots, d\}$, the expected $\ell$-th intrinsic volume of $\sP$ is given by
\begin{align*}
		\E V_\ell (\sP) = 2  \binom{d}{\ell} \frac{\k_d}{\k_{d-\ell}} \sum_{\substack{I \subseteq \{1,\ldots,n\} \\ \# I = \ell+1}}  \Theta \left(\frac{\ell}{2};  \lbrace \gamma_i : i \in I \rbrace;\left\lbrace \g_{i}: i \in  I^c \right\rbrace\right).
	\end{align*}
For every $\ell\in \{0,\ldots, d-1\}$, the expected number of $\ell$-dimensional faces of $\sP$ is given by
$$
\E f_{\ell}(\sP)
=2\sum_{\substack{K \subseteq \{1,\dots,n\} \\ \# K =\ell+1 }} \sum_{\substack{ J \supseteq K \\ \# J = \{d,d-2,\dots\}}} \Theta \left(\sum_{i\in K}\g_i;\{\g_i: i \in J \bsl K\};\{\g_i: i \in J^c\}\right).
$$
If $n=d+2$, then
$$
\P[\,\sP \text{ is a simplex }] = 2 \cdot \sum_{j=1}^{d+2} \Theta \left(\g_j ; \{ \gamma_1,\ldots, \gamma_{j-1}, \gamma_{j+1}, \ldots, \gamma_{d+2}\}; \varnothing\right).
$$
\end{theorem}
Proofs of these and further results on beta polytopes will be presented in Section~\ref{sec:beta_polytopes_results}.  The proofs combine techniques from stochastic and integral geometry with the analysis of special functions $a$ and $b$ to be carried out  in Section~\ref{sec:beta_cones_explicit_internal_angles}.

\begin{remark}[On beta type distributions]
It is possible to develop a theory parallel to the one presented here by replacing the beta distribution with the \emph{beta prime distribution} whose probability density is proportional to $(1+\|x\|^2)^{-\beta}$, with parameter $\beta > d/2$, on the whole of $\R^d$. Further, the isotropic normal distribution with density proportional to $\eee^{-\|x\|^2/(2\sigma^2)}$ is a limit of appropriately scaled beta and beta prime distributions as $\beta\to\infty$. For this reason, all results of the present paper have analogues for Gaussian polytopes generated by random Gaussian points $X_1,\ldots, X_n$ in $\R^d$ with possibly different variances $\sigma_1^2,\ldots, \sigma_n^2$.  We refrain from presenting these extensions here.
\end{remark}

\begin{remark}[On variances and moments]
In this paper we focus on computing the \emph{expectations} of geometric functionals. Explicit formulas for variances (or higher moments) are not known except in some special cases. If $n\leq d+1$, that is when a beta polytope $\sP_{n,d}^{\beta_1,\dots, \beta_n}$ is a.s.\ a simplex, \citet{miles} and \citet{ruben_miles} were able to calculate all integer moments of its volume; see~\cite{kabluchko_steigenberger_thaele} for a discussion of these results  and  various extensions. In the regime when $d$ is fixed and $n\to\infty$, several authors~\cite{calka_quilan_first_layers,calka_yukich_variance_asympt,calka_schreiber_yukich_brownian_limits_local_limits,reitzner_clt_for_random_polys} proved \emph{asymptotic} formulas for the variance and central limit theorems for the number of $k$-faces, the volume and the intrinsic volumes of $\sP_{n,d}^{0,\dots, 0}$ and more general random polytopes.
\end{remark}

\subsection{Motivation to introduce different parameters and beta cones}
To conclude the introduction, we would like to motivate the study of beta cones and beta polytopes with different paramaters $\beta_1,\ldots, \beta_n$.
Up until recently, the work on random beta polytopes has been mostly in the setting where all parameters are equal, that is, $\b_1=\dots=\b_n \geq -1$, such  as~\cite{bonnet_etal,BonnetEtAlThresholds,bonnet_kabluchko_turchi,beta_simplices,gusakova_kabluchko_sylvester_beta,kabluchko_angle_sums_dim_3_4,kabluchko_angles_of_random_simplices_and_face_numbers,kabluchko_recursive_scheme,kabluchko_on_expected_face_numbers_of_beta_polytopes,beta_polytopes,kabluchko_thaele_zaporozhets_beta_polytopes_poi_polyhedra}. Notable exceptions are the papers of Ruben and Miles~\cite{miles,ruben_parallelotopes,ruben_miles}, where the authors introduced and studied random \emph{beta simplices}, that is, convex hulls of at most $d+1$  independent beta-distributed points in $\R^d$, allowing for possibly different parameters. An extensive study of beta simplices was conducted in~\cite{kabluchko_steigenberger_thaele}.
Recently, \citet{moseeva2024mixedrandombetapolytopes} derived formulas for the expected volume of beta polytopes with different parameters adapting the method of~\cite{beta_polytopes}.

Let us consider two natural problems that crucially rely on allowing for different parameters. Both problems lead in a natural way to beta cones, which motivates their study.

\begin{example}[Expected volume as absorbtion probability]\label{ex:volume_of_polytope_through_uniform_point}
To calculate the expected volume of a beta polytope $ \sP_{n,d}^{\b_1,\dots,\b_n} = [X_1,\dots,X_n]\subseteq \R^d$, we start with a version of Efron's identity~\cite{Efron-identity,buchta}. Take an additional point $Y$ which is uniformly distributed in the unit ball $\BB^d$, that is, $Y \sim f_{d,0}$, and independent of $X_1,\ldots, X_n$. The expected volume of $\sP_{n,d}^{\b_1,\dots,\b_n}$, divided by the volume of the unit ball, equals the probability that the additional point $Y$ falls inside of $\sP_{n,d}^{\b_1,\dots,\b_n}$, that is
	\begin{align*}
		\frac{\E \Vol_d \sP_{n,d}^{\b_1,\dots,\b_n}}{\Vol_d(\mathbb{B}^d)} = \P \left[Y \in \sP_{n,d}^{\b_1,\dots,\b_n} \right].
	\end{align*}
The right-hand side can be understood as the probability that the point $Y$ is not a vertex of the polytope $\sP_{n+1,d}^{\b_1,\dots,\b_n,0}=[X_1,\ldots, X_n, Y]$. Note that even if $\beta_1= \ldots = \beta_n$, the parameters of this polytope are in general not equal -- this motivates the introduction of  beta polytopes with different parameters.  Furthermore, the event that $Y$ is not a vertex takes place if and only if $\pos (X_1-Y,\ldots, X_{n}-Y)=\R^d$.  
So,
$$
\E \Vol_d \sP_{n,d}^{\b_1,\dots,\b_n} = \Vol_d(\mathbb{B}^d) \cdot \P[\, \pos (X_1-Y,\ldots, X_{n}-Y) = \R^d\,].
$$
This expresses the expected volume of a beta polytope in terms of an ``absorbtion probability'' of $\pos (X_1-Y,\ldots, X_{n}-Y)$, which is a beta cone by~\eqref{eq:beta_cones_intro}.  Thus, in order to calculate
$\E \Vol_d \sP_{n,d}^{\b_1,\dots,\b_n}$ it is useful to introduce beta cones.
\end{example}
\begin{figure}[h!]
	\centering
	\begin{tikzpicture}[scale=1.4]
		% 0) Coordinates
		\coordinate (O)  at (0,0);
		\coordinate (X1) at (-1.0,-1.2);
		\coordinate (X2) at (-1.2,-0.5);
		\coordinate (X3) at (-1.1,0.7);
		\coordinate (X4) at (-0.3,1.3);
		\coordinate (X5) at (0.7,1.1);
		\coordinate (X6) at (1.5,0.4);
		\coordinate (X7) at (1.3,-0.7);
		\coordinate (X8) at (0.7,-1.3);
		\coordinate (X9) at (0.1,-1.5);
		\coordinate (X10) at (0.6,-0.5);
		\coordinate (X11) at (1.1,-0.2);
		\coordinate (X12) at (-0.6,-0.5);
		\coordinate (Y) at (-0.6,0.5);
		% 1) Draw the outer circle
		\draw[thick] (0,0) circle (2);

		% 2) Place a black dot
		\fill (0.0, 0.0) circle (0.5pt);
		\fill (X10) circle (0.3pt);
		\fill (X11) circle (0.3pt);
		\fill (X12) circle (0.3pt);

		% 3) Place a small green circle (dot)
		\fill[green!50!black] (Y) circle (0.6pt);

		% Coordinates

		% 4) Draw the red polygon
		\draw[red!80!black, thick]
			(X1) -- (X2) -- (X3) -- (X4) -- (X5) -- (X6) -- (X7) -- (X8) -- (X9) -- cycle;
		\fill[red!40, opacity=0.4] (X1) -- (X2) -- (X3) -- (X4) -- (X5) -- (X6) -- (X7) -- (X8) -- (X9) -- cycle
		%node [pos=0, above, scale=0.9, opacity=1, green!50!black] {$A^\perp$}
		;

		% 5) Nodes
		\node[left, scale=0.7] at  (X1) {$X_{11}$};
		\node[left, scale=0.7] at  (X2) {$X_3$};
		\node[left, scale=0.7] at  (X3) {$X_5$};
		\node[above, scale=0.7] at  (X4) {$X_4$};
		\node[above, scale=0.7] at  (X5) {$X_{10}$};
		\node[right, scale=0.7] at  (X6) {$X_{12}$};
		\node[right, scale=0.7] at  (X7) {$X_7$};
		\node[right, scale=0.7] at  (X8) {$X_6$};
		\node[below, scale=0.7] at  (X9) {$X_9$};
		\node[below, scale=0.7] at  (X10) {$X_{2}$};
		\node[below, scale=0.7] at  (X11) {$X_{1}$};
		\node[below, scale=0.7] at  (X12) {$X_{8}$};
		\node[right, scale=0.7] at  (Y) {$Y$};
		\node[right, scale=0.7] at  (0,0) {$0$};

	\end{tikzpicture}
	\caption{The expected volume of $\sP_{12,2}^{\b_1,\dots,\b_{12}}$ divided by the area of the unit disk is the probability that the additional point $Y$ takes its value inside of $\sP_{12,2}^{\b_1,\dots,\b_{12}}$.}
\end{figure}
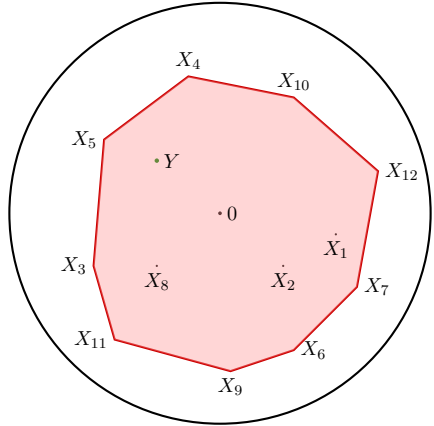

\begin{example}[Expected volume as expected angle]\label{ex:expect_vol_as_angle}
Another way to calculate the expected volume of $\sP_{n,d}^{\b_1,\dots,\b_n} = [X_1,\dots,X_n]$
is by introducing a higher-dimensional polytope
$$
[Y_1,\dots,Y_{n}]  \sim \BetaPoly(\R^{d+1}; \beta_1-1/2,\ldots,  \beta_n-1/2).
$$
The original polytope $[X_1,\ldots, X_n]$ can then be understood as an orthogonal projection of $[Y_1,\dots,Y_{n}]$ onto a fixed hyperplane $L \subseteq \R^{d+1}$ -- this follows from the projection property of the beta distribution stated in Lemma~\ref{lem:projection} below. More precisely, if $\Pi_L:\R^{d+1} \to L$ denotes the orthogonal projection to $L$, then $[X_1,\ldots, X_n]$ has the same distribution as $[\Pi_L Y_1,\ldots, \Pi_L Y_n]$.  Further,  for an additional point $Y_{n+1}\sim f_{d+1,-1/2}$ independent of $Y_1,\ldots, Y_n$ and its projection $X_{n+1}:= \Pi_L (Y_{n+1}) \sim f_{d,0}$,  we have
\begin{align*}
\frac{\E \Vol_d \sP_{n,d}^{\b_1,\dots,\b_n}}{\Vol_d(\mathbb{B}^d)}
=
\P \left[X_{n+1} \in \sP_{n,d}^{\b_1,\dots,\b_n}\right] &= \P \left[\Pi_L(Y_{n+1}) \in [\Pi_L(Y_{1}),\dots,\Pi_L(Y_{n})], \; Y_{n+1} \in [Y_1,\dots,Y_n]\right] \\ &+ \P \left[\Pi_L(Y_{n+1}) \in [\Pi_L(Y_{1}),\dots,\Pi_L(Y_{n})], \; Y_{n+1} \notin [Y_1,\dots,Y_n]\right].
	\end{align*}
The first summand can be rewritten as $\P[Y_{n+1} \in [Y_1,\dots,Y_n]]$, which is the probability that $\pos (Y_1-Y_{n+1},\dots,Y_n- Y_{n+1})=\R^{d+1}$. For the second summand, we do the following. First, we choose $L=U^\perp$, where $U$ is a random unit vector uniformly distributed on $\SS^d\subseteq \R^{d+1}$ and independent of $Y_1,\dots,Y_{n+1}$. Then, applying the formula of total probability, we write the second summand as
\begin{align*}
\E \left[\; \P \left[ \Pi_{U^\perp}(Y_{n+1}) \in [\Pi_{U^\perp}(Y_{1}),\dots,\Pi_{U^\perp}(Y_{n})] \;|\; Y_1,\dots,Y_{n+1} \right] \cdot \ind_{\{Y_{n+1} \notin [Y_1,\dots,Y_n]\}}\;\right].
\end{align*}
	In the conditional probability, the points $Y_1,\dots,Y_{n+1}$ are fixed and only the hyperplane $L= U^\perp$ is random. So, let us fix the points $Y_1,\ldots, Y_{n+1}\in \R^{d+1}$ such that  $Y_{n+1} \notin [Y_1,\dots,Y_n]$. Then, we can understand the conditional probability as \emph{two} times the solid angle of the cone $\sC := \pos (Y_1-Y_{n+1},\dots,Y_n-Y_{n+1})$. Indeed, $Y_{n+1}$ gets projected inside of $[\Pi_{U^\perp}(Y_{1}),\dots,\Pi_{U^\perp}(Y_{n})]$ if and only if  $U\in \sC$ or $U\in -\sC$; see Figures~\ref{fig:angle_2_dim} and~\ref{fig:angle_3_dim} for a clarification of this fact. This gives
	\begin{align*}
\E \left[\; \P \left[ \Pi_{U^\perp}(Y_{n+1}) \in [\Pi_{U^\perp}(Y_{1}),\dots,\Pi_{U^\perp}(Y_{n})] \;|\; Y_1,\dots,Y_{n+1} \right] \cdot \ind_{\{Y_{n+1} \notin [Y_1,\dots,Y_n]\}}\;\right] = 2 \cdot \E \left[\a(\sC)\ind_{\{\sC \neq \R^{d+1}\}}\right].
	\end{align*}
The basic idea of interpreting angles as probabilities can be traced back to~\citet{Perles_Shepard_1967} and~\citet{Shephard_1967}; see the paper of~\citet{Feldman01102009} for a particularly nice explanation.  Recently, in~\cite{gusakova_kabluchko_sylvester_beta}, it was applied to Sylvester's problem for the beta distributions. In total, we get that
\begin{align*}
\frac{\E \Vol_d \sP_{n,d}^{\b_1,\dots,\b_n}}{\Vol_d(\mathbb{B}^d)} = 	\P \left[X_{n+1} \in \sP_{n,d}^{\b_1,\dots,\b_n}\right] = \P \left[\sC = \R^{d+1} \right] + 2 \cdot  \E \left[\a(\sC) \ind_{\{\sC \neq \R^{d+1}\}} \right].
\end{align*}
In particular, as stated in~\eqref{eq:beta_cones_intro}, $\sC = \pos(Y_1-Y_{n+1},\dots,Y_n-Y_{n+1})$ is a beta cone. This again reduces a natural problem concerning beta polytopes to beta cones. Both probabilities on the right-hand side will be made explicit in Theorem~\ref{theo:beta_cones_conic_intrinsic_vol_as_A_B}. For generalizations of these results  we refer to Sections~\ref{sec:beta_cones_results} and~\ref{sec:beta_polytopes_results}.
\end{example}

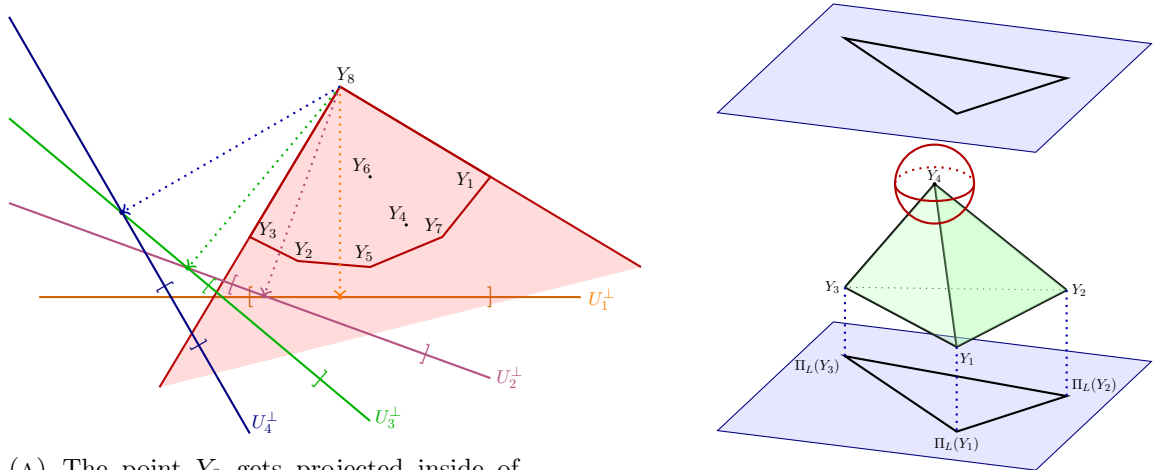
\begin{figure}[h!]
    \centering
	\begin{subfigure}[b]{0.40\textwidth}
        \centering
		\begin{tikzpicture}[scale=0.8, every node/.style={scale=0.8}]
			% Draw circle (radius 4, centered at (0,0))
			%\draw[thick] (0,0) circle (4cm);
			%\filldraw[black] (O) circle (1.0pt);
			% Original horizontal (orange) tangent line (tangent at (0,-4), θ=270°)
			\draw[thick, orange!80!black] (-4.5,-4) -- (4.5,-4);
			
			% Simplex vertices
			\coordinate (X1) at (-1.0,-3.0);
			\coordinate (X2) at (3.0,-2.0);
			\coordinate (X3) at (0.5,-0.5);
			\coordinate (X4) at (-0.2,-3.4);
			\coordinate (X5) at (1.0,-3.5);
			\coordinate (X6) at (2.2,-3.0);
			\coordinate (X7) at (1.0,-2.0);
			\coordinate (X8) at (1.6,-2.8);
			%\filldraw[black] (X4) circle (1.5pt);
			%\filldraw[black] (X5) circle (1.0pt);
			%\filldraw[black] (X6) circle (1.0pt);
			% Original projections onto the orange tangent (by forcing y=-4)
			\coordinate (X1proj) at (-1.0,-4.0);
			\coordinate (X2proj) at (3.0,-4.0);
			\coordinate (X3proj) at (0.5,-4.0);
			\coordinate (X1-X3) at ($(X1)-(X3)$);
			\coordinate (X2-X3) at ($(X2)-(X3)$);
			\coordinate (scaledX1-X3) at ($2*(X1-X3)+(X3)$);
			\coordinate (scaledX2-X3) at ($2*(X2-X3)+(X3)$);
			%\filldraw[orange] (X1proj) circle (1.5pt);
			\node[orange!80!black, scale=0.9, rotate=0] at (X1proj) {$\textbf{[}$};
			%\filldraw[orange] (X2proj) circle (1.5pt);
			\node[orange!80!black, scale=0.9, rotate=0] at (X2proj) {$\textbf{]}$};
			\filldraw[orange] (X3proj) circle (1.0pt);
			Cone
			\draw[red!70!black,thick]        (X3) -- (scaledX1-X3);
			\draw[red!70!black,thick]        (X3) -- (scaledX2-X3);
			\fill[red!70, opacity=0.2] (X3) -- (scaledX1-X3) -- (scaledX2-X3) -- cycle;
			\filldraw[black] (X7) circle (0.5pt);
			\filldraw[black] (X8) circle (0.5pt);
			% Draw simplex and its projection on the orange tangent
			\draw[red!70!black, thick] (X1) -- (X4) -- (X5) -- (X6) -- (X2) -- (X3) -- cycle;
			%\draw[thick, dotted, orange] (X1) -- (X1proj);
			%\draw[thick, dotted, orange] (X2) -- (X2proj);
			\draw[->, thick, dotted,  orange] (X3) -- (X3proj);
			
			% Labels for vertices and projection line
			\node[above, scale=0.8]  at ($(X1)+(0.3,-0.13)$) {$Y_3$};
			\node[left, scale=0.8]  at ($(X2)+(-0.15,-0.05)$) {$Y_1$};
			\node[above, scale=0.8] at ($(X3)+(0.1,-0.05)$) {$Y_8$};
 			\node[above, scale=0.8] at ($(X4)+(0.1,-0.05)$) {$Y_2$};
			\node[above, scale=0.8] at ($(X5)+(-0.1,-0.00)$) {$Y_5$};
			\node[above, scale=0.8] at ($(X6)+(-0.14,-0.05)$) {$Y_7$};
			\node[above, scale=0.8] at ($(X7)+(-0.14,-0.05)$) {$Y_6$};
			\node[above, scale=0.8] at ($(X8)+(-0.14,-0.05)$) {$Y_4$};
			\node[below, orange, scale=0.8] at ($(X2proj)+(1.85,0.25)$) {$U^\perp_1$};
			%\node[below] at (X1proj) {$\Pi_{U^\perp}(Y_1)$};
			%\node[below] at (X2proj) {$\Pi_{U^\perp}(Y_2)$};
			
			% A reference dotted circle (for context)
			%\draw[thick, dotted, green!60!black] (0.5,-0.5) circle (2.9155cm);
			
			% --- New tangent lines with tangent points moving clockwise ---
			% For θ = 250°:
			% n = (cos250, sin250) = (-0.3420, -0.9397)
			% Tangent line: 0.3420*x+0.9397*y=-4  →  y=(-4-0.3420*x)/0.9397.
			\draw[thick, magenta!70!black] plot[domain=-5:3] (\x, {(-4 - 0.3420*\x)/0.9397});
			
			% For θ = 230°:
			% n = (cos230, sin230) = (-0.6428, -0.7660)
			% Tangent line: 0.6428*x+0.7660*y=-4  →  y=(-4-0.6428*\x)/0.7660.
			\draw[thick, green!70!black] plot[domain=-5:1] (\x, {(-4 - 0.6428*\x)/0.7660});
			
			% For θ = 210°:
			% n = (cos210, sin210) = (-0.8660, -0.5)
			% Tangent line: 0.8660*x+0.5*y=-4  →  y=(-4-0.8660*\x)/0.5.
			\draw[thick, blue!50!black] plot[domain=-5:-1] (\x, {(-4 - 0.8660*\x)/0.5});
			
			% --- Projections of the simplex vertices onto the new tangents ---
			% Using: X_proj = X + (4 - (X·n))* n.
			
			% For θ = 250° (n = (-0.3420, -0.9397)):
			% X1 = (-1,-3): dot = 3.1611, correction = 0.8389.
			\coordinate (X1proj250) at (-1.287, -3.787);
			% X2 = (3,-2): dot = 0.8534, correction = 3.1466.
			\coordinate (X2proj250) at (1.924, -4.956);
			% X3 = (0.5,-0.5): dot = 0.2989, correction = 3.7011.
			\coordinate (X3proj250) at (-0.764, -3.977);
			%\draw[thick, dotted, magenta] (X1) -- (X1proj250);
			%\draw[thick, dotted, magenta] (X2) -- (X2proj250);
			\draw[->, thick, dotted,magenta!70!black] (X3) -- (X3proj250);
			%\filldraw[magenta!70!black] (X1proj250) circle (1.5pt);
			\node[magenta!70!black, scale=0.9, rotate=340] at (X1proj250) {$\textbf{[}$};
			%\filldraw[magenta!70!black] (X2proj250) circle (1.5pt);
			\node[magenta!70!black, scale=0.9, rotate=340] at (X2proj250) {$\textbf{]}$};
			\filldraw[magenta!70!black] (X3proj250) circle (1.0pt);
			\node[below, magenta!70!black, scale=0.8] at ($(X2proj250)+(1.4,-0.1)$) {$U^\perp_2$};
			
			% For θ = 230° (n = (-0.6428, -0.7660)):
			% X1 = (-1,-3): dot = 2.9408, correction = 1.0592.
			\coordinate (X1proj230) at (-1.680, -3.811);
			% X2 = (3,-2): dot = -0.3964, correction = 4.3964.
			\coordinate (X2proj230) at (0.174, -5.366);
			% X3 = (0.5,-0.5): dot = 0.0616, correction = 3.9384.
			\coordinate (X3proj230) at (-2.029, -3.515);
			%\draw[thick, dotted, green] (X1) -- (X1proj230);
			%\draw[thick, dotted, green] (X2) -- (X2proj230);
			\draw[->, thick,  dotted,green!70!black] (X3) -- (X3proj230);
			%\filldraw[green!70!black] (X1proj230) circle (1.5pt);
			\node[green!70!black, scale=0.8, rotate=320] at (X1proj230) {$\textbf{[}$};
			%\filldraw[green!70!black] (X2proj230) circle (1.5pt);
			\node[green!70!black, scale=0.9, rotate=320] at (X2proj230) {$\textbf{]}$};
			\filldraw[green!70!black] (X3proj230) circle (1.0pt);
			\node[below, green!70!black, scale=0.8] at ($(X2proj230)+(1.13,-0.35)$) {$U^\perp_3$};
			
			% For θ = 210° (n = (-0.8660, -0.5)):
			% X1 = (-1,-3): dot = 2.3660, correction = 1.6340.
			\coordinate (X1proj210) at (-2.415, -3.817);
			% X2 = (3,-2): dot = -1.5980, correction = 5.5980.
			\coordinate (X2proj210) at (-1.843, -4.799);
			% X3 = (0.5,-0.5): dot = -0.1830, correction = 4.1830.
			\coordinate (X3proj210) at (-3.123, -2.592);
			%\draw[thick, dotted, cyan] (X1) -- (X1proj210);
			%\draw[thick, dotted, cyan] (X2) -- (X2proj210);
			\node[blue!50!black, scale=0.8, rotate=300] at (X1proj210) {$\textbf{[}$};
			\node[blue!50!black, scale=0.8, rotate=300] at (X2proj210) {$\textbf{]}$};
			\draw[->, thick, dotted,blue!70!black] (X3) -- (X3proj210);
			%\filldraw[blue!50!black] (X1proj210) circle (1.5pt);
			%\filldraw[blue!50!black] (X2proj210) circle (1.5pt);
			\filldraw[blue!50!black] (X3proj210) circle (1.0pt);
			\node[below, blue!50!black, scale=0.8] at ($(X2proj210)+(1.1,-1.0)$) {$U^\perp_4$};
			
		\end{tikzpicture}
		
        \caption{
		The point $Y_8$ gets projected inside of $[\Pi_{U^\perp}(Y_1),\dots,\Pi_{U^\perp}(Y_7)]$, an interval indicated by the square brackets, if and only if the vector $U$  falls inside of $\pm \pos(Y_1-Y_8,\dots,Y_7-Y_8)$. The  presence of $\pm$ explains why we take two times the solid angle. 
		}
		\label{fig:angle_2_dim}
    \end{subfigure}
	\hfill
    %--- First subfigure ---%
    \begin{subfigure}[b]{0.55\textwidth}
			\centering
			\tdplotsetmaincoords{70}{110}
		
		\begin{tikzpicture}[tdplot_main_coords, scale=1.50, every node/.style={scale=0.5}]
		
		% -----------------------------------------------------
		% 1) Define coordinates for the larger lower plane
		% -----------------------------------------------------
		\coordinate (Plane1_1) at (-1.5, -1.5, 0);
		\coordinate (Plane1_2) at ( 1.5, -1.5, 0);
		\coordinate (Plane1_3) at ( 1.5,  1.5, 0);
		\coordinate (Plane1_4) at (-1.5,  1.5, 0);
		
		% -----------------------------------------------------
		% 2) Define coordinates for the larger upper plane (shifted up by z=3)
		% -----------------------------------------------------
		\coordinate (Plane2_1) at (-1.5, -1.5, 3);
		\coordinate (Plane2_2) at ( 1.5, -1.5, 3);
		\coordinate (Plane2_3) at ( 1.5,  1.5, 3);
		\coordinate (Plane2_4) at (-1.5,  1.5, 3);
		
		% -----------------------------------------------------
		% 3) Draw the planes with outline and fill
		% -----------------------------------------------------
		\filldraw[fill=blue!10, draw=navy] (Plane1_1) -- (Plane1_2) -- (Plane1_3) -- (Plane1_4) -- cycle;
		\filldraw[fill=blue!10, draw=navy] (Plane2_1) -- (Plane2_2) -- (Plane2_3) -- (Plane2_4) -- cycle;
		
		% -----------------------------------------------------
		% 4) Define the pyramid vertices (Y_1..Y_4)
		% -----------------------------------------------------
		\coordinate (Y_4) at ( 0.0,  0.0,  2.0); % Top vertex (also sphere center)
		\coordinate (Y_1) at ( 0.8,  0.5,  0.8);
		\coordinate (Y_2) at (-0.4,  1.1,  1.0);
		\coordinate (Y_3) at (-0.7, -1.1,  0.65);
		
		% -----------------------------------------------------
		% 5) Define the projections of Y_1, Y_2, Y_3 onto the lower and upper planes
		% -----------------------------------------------------
		% Lower plane: use same (x,y) but set z=0.
		\coordinate (piLY1) at (0.8,  0.5, 0);
		\coordinate (piLY2) at (-0.4, 1.1, 0);
		\coordinate (piLY3) at (-0.7,-1.1, 0);
		% Upper plane: set z=3.
		\coordinate (Y1up) at (0.8,  0.5, 3);
		\coordinate (Y2up) at (-0.4, 1.1, 3);
		\coordinate (Y3up) at (-0.7,-1.1, 3);

		% -----------------------------------------------------
		% 6) Draw the pyramid edges
		% -----------------------------------------------------
		% Base edges: draw Y_1--Y_2 and Y_1--Y_3 solid; the edge Y_2--Y_3 dotted.
		\draw[thick] (Y_1) -- (Y_2);
		\draw[thick] (Y_1) -- (Y_3);
		\draw[dotted] (Y_2) -- (Y_3);
		% Edges connecting the apex to the base.
		\draw[thick] (Y_4) -- (Y_1);
		\draw[thick] (Y_4) -- (Y_2);
		\draw[thick] (Y_4) -- (Y_3);
		
		%Faces of Simplex
		\fill[green!30, opacity=0.3] (Y_1) -- (Y_3) -- (Y_4) -- cycle;
		\fill[green!30, opacity=0.5] (Y_1) -- (Y_2) -- (Y_4) -- cycle;
		
		% -----------------------------------------------------
		% 7) Draw the projected triangles on both planes
		% -----------------------------------------------------
		\draw[thick] (piLY1) -- (piLY2) -- (piLY3) -- cycle;
		\draw[thick] (Y1up)  -- (Y2up)  -- (Y3up)  -- cycle;
		
		% -----------------------------------------------------
		% 8) Add nodes (labels) to the vertices and lower-plane projections
		% -----------------------------------------------------
		\node[above] at ($(Y_4)+(0.0,0.0,-0.032)$) {$Y_4$};
		\node[below]  at ($(Y_1)+(0.0,0.1,0.0)$) {$Y_1$};
		\node[right] at (Y_2) {$Y_2$};
		\node[left]  at (Y_3) {$Y_3$};
		
		\node[below]  at ($(piLY1)+(0.0,0.0,0.0)$) {$\Pi_L(Y_1)$};
		\node[right] at ($(piLY2)+(0.0,0.0,0.1)$) {$\Pi_L(Y_2)$};
		\node[left] at ($(piLY3)+(0.0,0.0,-0.1)$) {$\Pi_L(Y_3)$};
		
		% Optionally, mark the vertices with small filled circles:
		\fill (Y_4) circle (0.5pt);
		%\fill (Y_1) circle (0.5pt);
		%\fill (Y_2) circle (0.5pt);
		%\fill (Y_3) circle (0.5pt);
		%\fill (piLY1) circle (0.5pt);
		%\fill (piLY2) circle (0.5pt);
		%\fill (piLY3) circle (0.5pt);
		\draw[thick,dotted,blue!70!black] (Y_3) -- (piLY3);
		\draw[thick,dotted,blue!70!black] (Y_2) -- (piLY2);
		\draw[thick,dotted,blue!70!black] (Y_1) -- (piLY1);
		% -----------------------------------------------------
		% 9) Draw a bigger sphere around Y_4 with a split equator (ellipse)
		% -----------------------------------------------------
		\begin{scope}[tdplot_screen_coords]
			\path (Y_4) coordinate (SphereCenter);
			\draw[red!70!black,thick] (SphereCenter) circle (0.35);
			\begin{scope}[shift={(SphereCenter)}, xscale=0.35, yscale=0.15]
			  \draw[red!70!black,thick] (180:1) arc (180:360:1);
			  \draw[red!70!black, dotted,thick] (0:1) arc (0:180:1);
			\end{scope}
		  \end{scope}
		
		\end{tikzpicture}
		\caption{The probability that $\Pi_{U^\perp}(Y_4)$ falls inside of $[\Pi_{U^\perp}(Y_1),\Pi_{U^\perp}(Y_2),\Pi_{U^\perp}(Y_3)]$ can be expressed as two times the solid angle of the cone $\pos(Y_1-Y_4,Y_2-Y_4,Y_3-Y_4)$. The red sphere is centered at $Y_4$ and has been added to clarify this.}
\label{fig:angle_3_dim}
    \end{subfigure}
    \caption{Illustration of Example~\ref{ex:expect_vol_as_angle} for $d=1$ and $d=2$.}
    %--- Second subfigure ---%
\end{figure}
\section{Notation and facts from stochastic and polyhedral geometry}\label{sec:facts}

\subsection{General notation}
For $d\in \N$ let $\R^d$ denote the $d$-dimensional Euclidean space endowed with the standard scalar product $\langle\,\cdot\,,\,\cdot\,\rangle$ and the Euclidean norm $\|\,\cdot\,\|$. The $d$-dimensional volume is denoted by $\Vol_d(\cdot)$.  The (closed) unit ball and the unit sphere in $\R^d$ are denoted by $\BB^d:=\{x\in\R^d:\|x\|\leq 1\}$  and $\SS^{d-1}=\{x\in\R^d:\|x\|=1\}$, respectively.  The volume of the unit ball $\BB_d$ is denoted by
\begin{align}\label{volume_of_the_unit_ball}
\kappa_d = \Vol_d(\mathbb{B}^d) = \frac{\pi^{d/2}}{\Gamma\left(\frac{d}{2}+1\right)}.
\end{align}
We write $[n]:= \{1,\ldots, n\}$. The number of elements of a finite set $K$ is denoted by $\#K$.
We let $(\Omega,\cF,\P)$ be our underlying probability space.
Expectation (i.e.\ integration) with respect to $\P$ is denoted by $\E$.

\subsection{Hulls, polytopes and cones}
In this section we recall some standard facts and notation from convex geometry. For proofs we refer, e.g.,\ to the books by~\citet{barvinok_book}, \citet{broendsted_book}.
Let $v_1,\ldots, v_m\in \R^d$ be vectors. An expression of the form $\lambda_1 v_1 + \ldots + \lambda_m v_m$
with real coefficients $\lambda_1,\ldots,\lambda_m$ is called a \emph{linear combination} of $v_1,\ldots, v_m$.
A linear combination $\lambda_1v_1 + \ldots + \lambda_m v_m$ is called
\begin{itemize}
\item a \emph{positive combination} if $\lambda_1\geq 0, \ldots, \lambda_m \geq 0$;
\item a \emph{convex combination} if $\lambda_1\geq 0,\ldots, \lambda_m \geq 0$ and $\lambda_1 + \ldots + \lambda_m = 1$;
\item an \emph{affine combination} if $\lambda_1+\ldots + \lambda_m = 1$.
\end{itemize}

The \textit{linear hull} $\lin A$ (respectively, the \textit{positive hull} $\pos A$, the \textit{convex hull} $\conv A$ and the \textit{affine hull} $\aff A$) of a set $A\subseteq\R^d$ is the set of all linear (respectively, positive, convex, affine) combinations of elements from $A$. The convex hull of  finitely  many points $v_1,\ldots,v_m$ is also denoted by $[v_1,\ldots,v_m]$, and their positive hull is denoted by $\pos (v_1,\ldots, v_m)$.

A set $A\subseteq \R^d$ is called a \emph{linear subspace} (respectively, a \emph{convex cone}, a \emph{convex set} or an \emph{affine subspace}) if it coincides with its own linear (respectively, positive, convex, affine) hull.
We denote by $G(d,k)$, respectively $A(d,k)$, the set of $k$-dimensional linear, respectively affine, subspaces of $\R^d$, where $k\in \{0,\ldots,d\}$.  For an affine subspace $A\in A(d,k)$ let $\Pi_A:\R^d \to A$ be the orthogonal projection onto $A$. Also, let $A^\perp:= \{x\in \R^d: \langle x, y\rangle = 0 \text{ for all } y\in A-a_0\}$ (where $a_0\in A$ is arbitrary) denote the orthogonal complement of $A$. Note that $A^\perp$ is always a \emph{linear} subspace.   If $v\in \R^d$ is a non-zero vector, then $v^\perp$ denotes its orthogonal complement.

For a set $A\subseteq \R^d$, its relative interior, denoted by $\relint A$, is the interior of $A$ taken with respect to $\aff A$ as the ambient space.

A \emph{polyhedron} is an intersection of finitely many closed half-spaces. A \textit{polytope} is a convex hull of finitely many points. Equivalently, by Minkowski's theorem, a polytope can be defined as a bounded polyhedron.  The dimension $\dim P$ of a polyhedron $P$ is the dimension of its affine hull $\aff P$.
A \textit{polyhedral cone} $C$ is the positive hull of finitely many vectors in $\R^d$. Equivalently, a polyhedral cone is an intersection $H_1\cap \ldots\cap H_p$ of finitely many half-spaces of the form $H_i = \{x\in \R^d: \langle x, u_i\rangle \leq 0\}$, where $u_1,\ldots, u_p \in \R^{d}\bsl \{0\}$.
The \textit{polar} (or \textit{dual}) of a polyhedral cone $C$ is defined by
$$
C^\circ := \{v\in \R^d\colon \langle v, z\rangle\leq 0 \text{ for all } z\in C\}.
$$
The double dual theorem states that $C^{\circ\circ} = C$.

Unless stated otherwise, beta cones are denoted with $\sC$, beta polytopes with $\sP$, arbitrary deterministic %
polyhedral cones with $C$ and deterministic polytopes with $P$.

\subsection{Faces, tangent cones and normal cones}
Let $P\subseteq \R^d$ be a polyhedron.  A \emph{face} of $P$ is defined as an intersection of $P$ with a supporting hyperplane. By convention, $P$ itself is also considered as a face. The dimension of a face $F$ is the dimension of $\aff F$.   Vertices are $0$-dimensional faces, edges $1$-dimensional faces and, fbacets are faces of dimension $\dim P -1$.  The \emph{lineality space} of a polyhedral cone $C$ is defined as the inclusion-maximal linear subspace contained in $C$, or as $\linsp = C\cap C^\circ$.  For $k\in \{0,\ldots, d\}$, let  $\cF_k(P)$ be the set of all $k$-dimensional faces of $P$. Also, let $f_k(P) := \# \cF_k(P)$ be the number of $k$-dimensional faces of $P$.  The \emph{$f$-vector} of $P$ is the vector $(f_i(P))_{i=0}^{\dim P}$.

Let $F$ be a face of $P$ and let $v$ be any point in $\relint F$. The  \textit{tangent cone} of $P$ at $F$  is defined as the set of vectors which, when placed at $v$, point inside $P$. More precisely, we define
$$
T(F,P) := \{x\in\R^d: v+\eps x\in P\text{ for some } \eps>0\} = \pos (P-v).
$$
The \textit{normal cone} of $P$ at $F$ is the polar to the tangent cone, that is
$$
N(F,P) := T^\circ (F,P).
$$

It is known that there is a bijective correspondence between $k$-dimensional faces of $C$ and $(d-k)$-dimensional faces of $C^\circ$ given by
$$
\cF_k(C) \ni F \qquad  \leftrightarrow \qquad N(F,C) = (\lin F)^\perp \cap C^\circ \in \cF_{d-k}(C^\circ).
$$
For example, the lineality space of $C$ corresponds to the linear hull of $C^\circ$.

\subsection{Angles and conic intrinsic volumes}\label{subsec:angles_conic_intrinsic_defs}
The \textit{solid angle} of a polyhedral cone $C\subseteq \R^d$ is defined as the proportion of the unit ball occupied by the cone with $\lin C$ viewed as the ambient space. More precisely, the solid angle of $C$ is
$$
\alpha(C) = \frac{\lambda_{\lin C}(\BB^d \cap C)}{\lambda_{\lin C} (\BB^d \cap \lin C)} = \P[N_{\lin C}\in C],
$$
where $\lambda_{\lin C}$ is the Lebesgue measure on the linear subspace $\lin C$  and $N_{\lin C}$ is a random vector having a standard normal distribution on the linear hull of $C$. Another way to define the solid angle is by taking a random vector $U_{\lin C}$ being either uniformly distributed on the unit ball in $\lin C$ or on the unit sphere in $\lin C$ and setting $\a(C) \coloneqq \P [U_{\lin C} \in C]$. More generally, any rotationally invariant distribution on $\lin C$ could be used.

Let $C\subseteq \R^d$ be a polyhedral cone.  With each face $G$ of $C$ three natural angles are associated: the \emph{internal angle at $G$} defined by $\beta(G,C) = \alpha (T(G,C))$, \emph{the external angle at $G$} defined by $\gamma(G,C) = \alpha (N(G,C))$ and the angle of $G$ itself denoted by $\alpha(G)$. The \emph{conic intrinsic volumes} $\upsilon_0(C), \ldots, \upsilon_d(C)$  of $C$ are defined by
\begin{equation}\label{eq:upsilon_def}
\upsilon_k(C)
=
\sum_{G\in\cF_k(C)} \alpha(G)  \gamma(G,C),
\qquad
k\in \{0,\ldots,d\}.
\end{equation}
For properties of the conic intrinsic volumes we refer to~\cite{amelunxen_lotz,amelunxen_lotz_mccoy_tropp_living_on_the_edge,glasauer_phd} as well as to the books by~\citet[Section~2.3]{schneider_book_convex_cones_probab_geom} and~\citet[Section~6.5]{schneider_weil_book}. The conic intrinsic volumes satisfy the relations
\begin{align*}
\sum_{k=0}^d \upsilon_k(C) = 1,
\qquad \text{as well as} \qquad
&\sum_{k=0}^d (-1)^k \upsilon_k(C) = 0,
\end{align*}
where for the second relation (which is a version of the Gauss-Bonnet theorem) we assume that $C$ is not a linear subspace.

\begin{lemma}\label{lem:cone_intrinsic_volume_lineality_space}
Let $C\subseteq \R^d$ be a polyhedral cone whose  lineality space has dimension $p$ and whose linear hull has dimension $q$. Note that $0\leq p \leq q \leq d$.  Then, $\upsilon_{p}(C)= \alpha (C^\circ)$,  $\upsilon_{q}(C) = \alpha(C)$ and $\upsilon_{k}(C) = 0$ if $k\notin \{p,\ldots, q\}$.
\end{lemma}
\begin{proof}
Follows from~\eqref{eq:upsilon_def} upon observing that the only $p$-dimensional face of $C$ is its lineality space, while the only $q$-dimensional face of $C$ is $C$ itself.
\end{proof}
\begin{corollary}
If the lineality space of $C$ is $\{0\}$, then $\alpha(C^\circ) = \upsilon_0(C)$. If the linear hull of $C$ is $\R^d$, then $\alpha(C) = \upsilon_d(C)$.
\end{corollary}

\subsection{Properties of beta distributions}\label{subsec:beta_properties}
The family of beta distribution possesses two main properties which will be crucial in the following: invariance under projections and the canonical decomposition (or invariance under slicing). To state these properties, it will be convenient to fix some way to identify every affine subspace of $\R^d$ with the Euclidean space of the same dimension.

\subsubsection{Identification of affine subspaces}\label{subsubsec:identification_affine_subspaces}
Recall that $A(d,k)$ is the set of $k$-dimensional affine subspaces of $\R^d$ and that, for an affine subspace $E \in A(d, k)$, we denote by $\Pi_E: \R^d\to E$ the orthogonal projection onto $E$. Let $0_E=\Pi_E(0)=\argmin_{x\in E}\|x\|$ denote the projection of the origin on $E$. For every affine subspace $E\in A(d,k)$ we fix an isometry $I_E:E\to \R^k$ such that $I_E(0_E)=0$.  The exact choice of the isometries $I_E$ is not important. We only require that $(x,E) \mapsto I_E(\Pi_E(x))$ defines a Borel measurable map from $\R^d \times A(d,k)$ to $\R^k$, where we supply $\R^d$, $\R^k$ and $A(d,k)$ with their standard Borel $\sigma$-algebras; see~\cite[Chapter 13.2]{schneider_weil_book} for the case of $A(d,k)$. We say that a random vector $Z$ taking values in $E$ is beta-distributed on $E\in A(d,k)$ if $I_E(Z)$ is beta-distributed on $\R^k$. By the rotation invariance of the beta distribution, this definition does not depend on the choice of $I_E$.

\subsubsection{Projection property}
The next lemma, taken from  \cite[Lemma 4.4]{beta_polytopes} (alternatively, see~\cite[Section~2.2]{kabluchko_steigenberger_thaele}), states that an orthogonal projection  of a beta distribution is again a beta distribution with a larger parameter: decreasing the dimension by $\ell$ amounts to increasing the parameter by $\ell/2$.

\begin{lemma}[Projection property]\label{lem:projection}
Denote by $\Pi_L: \R^d\rightarrow L$ the orthogonal projection onto a $k$-dimensional linear subspace $L\subseteq \R^d$, where $k\in \{1,\ldots,d\}$. If the random point $X$ has distribution $f_{d,\beta}$ for some $\beta\geq -1$, then $I_L(\Pi_L (X))$ has distribution $f_{k,\b+\frac{d-k}{2}}$ in $\R^k$.
\end{lemma}

\subsubsection{Canonical decomposition of Ruben and Miles}
Let $X_1,\ldots,X_k$, where $k \in \{1,\ldots, d\}$,  be independent random points in $\R^d$ such that $X_i$ has the beta distribution $f_{d,\beta_i}$ with parameter $\beta_i\geq -1$.  %so that $[X_1,\ldots,X_k]$ is a simplex.
The next theorem, which will be crucial for what follows, is the so-called \emph{canonical decomposition} due to Ruben and Miles; see Theorem in Section 3 of \citet{ruben_miles} and~\citet[Section~12]{miles} for earlier versions.
It provides a description of the positions of the points $X_1,\ldots,X_k$ inside their own affine hull $A:= \aff(X_1,\ldots,X_k)$, as well as the position of $A$ inside of $\R^d$.
\begin{theorem}[Canonical decomposition]\label{theo:ruben_miles}
For $d \in \N$, $k \in \{1,\dots ,d\}$ and arbitrary $\beta_1,\ldots,\beta_k\geq -1$, let $X_1,\ldots,X_k$ be random points in $\R^d$, where for each $i \in \{1,\dots,k\}$, the point $X_i$ is beta-distributed with parameter $\b_i$.
Let $A:=\aff(X_1,\ldots,X_k)$ be the a.s.\ $(k-1)$-dimensional affine subspace spanned by $X_1,\ldots,X_k$, regarded as a random element taking values in the affine Grassmannian $A(d,k-1)$. Recall that $0_A$ denotes the orthogonal projection of the origin onto $A$ and let $h(A) := \|0_A\|$ denote the distance from the origin to $A$. Observe that $A\cap \BB^d$ is a.s.\ a $(k-1)$-dimensional ball of radius $\sqrt{1-h^2(A)}$ and consider the points
		$$
		Z_i := \frac{I_A(X_i)}{\sqrt{1-h^2(A)}} \in \BB^{k-1}, \quad i=1,\ldots,k,
		$$
		where $I_A:A\to \R^{k-1}$ is an isometry as in  Section~\ref{subsubsec:identification_affine_subspaces}.
		Then, the joint distribution of $(Z_1,\ldots,Z_k)$ and $A$ can be described as follows.
		\begin{itemize}
			\item[(a)] The random tuple $(Z_1,\ldots,Z_k)$ is stochastically independent of $A$.
			%\item[(b)] $[Z_1,\ldots,Z_k] \sim \BetaSimp (\R^{k-1}; \beta_1,\ldots,\beta_{k}; \Delta^{d-k+1+\rho})$.
			\item[(b)] The joint density of the vector $(Z_1,\dots,Z_k)$ with respect to the direct product of distributions $f_{k-1,\beta_1}, \ldots, f_{k-1,\beta_k}$ on $(\BB^{k-1})^k$ is				
			\begin{align}
					c_{k-1;\b_1,\dots,\b_k}^{(d-k+1)} \cdot   \Delta^{d-k+1} (z_1,\ldots,z_k),
					\qquad
					z_1,\ldots, z_k\in \BB^{k-1},
			\end{align}
where $\Delta(z_1,\ldots,z_k)$ is the $k$-dimensional volume of the simplex $[z_1,\ldots,z_k]$ 	and $c_{k-1;\b_1,\dots,\b_k}^{(d-k+1)}$ is a normalization constant whose value will not be needed in the sequel, but can be found in Remark 4.7 of~\cite{kabluchko_steigenberger_thaele}.
			\item[(c)] The distribution of $A$ on the affine Grassmannian $A(d,k-1)$ can be characterized as follows:
\begin{itemize}
\item[(c1)] The random linear space $A-0_A$ is uniformly distributed on the linear Grassmannian $G(d,k-1)$.
\item[(c2)] The density of the random vector $I_{A^\bot}(0_A)$, which takes values in $\BB^{d-k+1}$, is  $f_{d-k+1, \gamma}$ with
$$
\gamma = \frac {(k-1)(d+1)}2 + (\beta_1+\ldots+\beta_k)
= \sum_{i=1}^{k} \left(\b_i+\frac{d}{2} \right) - \frac{d-k+1}{2}.
$$
\item[(c3)] The random vector $I_{A^\bot}(0_A)$ is stochastically independent of $A-0_A$.
\end{itemize}
		\end{itemize}
	\end{theorem}

For a proof, see also Theorem 3.3 in \cite{kabluchko_steigenberger_thaele}. In the case $k=1$, the only non-empty statement is  Part~(c): it  states that $X_1$ has density $f_{d,\beta_1}$ (which is trivial).
The fact that $\dim A = k-1$ a.s.\ follows from  Lemma~\ref{lem:beta_poi_general_affine}, below.

\subsection{General position and its consequences}
We say that vectors $v_1,\ldots, v_n\in \R^d$ are in \emph{general linear position} if for every index set $K\subseteq\{1,\ldots, n\}$ we have $\dim \lin (v_i: i\in K) = \min (\#K, d)$.
The next lemma is standard.
\begin{lemma}[Cones spanned by vectors in general linear position]\label{lem:cones_generated_by_vectors_in_gen_lin_pos}
Let $v_1,\ldots, v_n\in \R^d$ be vectors in general linear position.
Then, for the cone $C:= \pos (v_1,\ldots, v_n)$ the following hold:
\begin{itemize}
\item[(i)] $\dim C = \min (n,d)$.
\item[(ii)] For every $k\in \{0,\ldots, \min (n,d)-1\}$,  every $k$-dimensional face $G$ of $C$ has the form $\pos (v_i: i\in K)$ for some subset $K\subseteq \{1,\ldots, n\}$ of size $k$. For $k=0$, we use the convention $\pos (\varnothing) = \{0\}$.
\item[(ii)] Either $\linsp C= \{0\}$ or $C = \R^d$.
\end{itemize}
\end{lemma}
\begin{proof}
For Part~(i) observe that $\dim C = \dim \lin C =  \dim \lin (v_1,\ldots, v_n)$ and use the definition of the general linear position.
For Part~(ii) let $G= C\cap H$ be a $k$-face of $C$, where $H$ is a supporting hyperplane of $C$. Let $K= \{i\in \{1,\ldots, n\}: v_i \in H\}$.  Let $u\in \R^d$ be a unit normal vector of $H$. Then, $\langle u, v_i\rangle = 0$ for $i\in K$ and, after replacing $u$ by $-u$ if necessary,   $\langle u, v_{i} \rangle >0$  for $i\notin K$. Clearly, $\pos (v_i: i\in K) \subseteq C\cap H = G$. To prove the converse inclusion, let $w \in (C\cap H) \bsl \pos (v_i: i\in K)$. Since $w\in C$, we can represent it as $w = \lambda_1 v_1 + \ldots + \lambda_n v_n$ with $\lambda_i \geq 0$ for all $i=1,\ldots, n$. Since $v\notin\pos (v_i: i\in K)$, we have $\lambda_{i_0}>0$ for some $i_0\notin K$.   Taking the scalar product with $u$, we get $0 = \lambda_1 \langle u, v_1\rangle + \ldots +\lambda_n \langle u, v_n \rangle$. This is a contradiction since  $u$ is orthogonal to $v_{i}$ with $i\in K$ and $\langle u, v_{i} \rangle >0$  for all $i\notin K$. So, $\pos (v_i: i\in K) = C\cap H = G$.
By the general linear position assumption, we have $k = \dim G = \dim \lin (v_i: i\in K) = \min (\#K, d)$. Since $k\neq d$, we conclude that $k=\# K$, which completes the proof of~(ii). To prove~(iii), assume by contraposition that $C\neq \R^d$ and $L:= \linsp C$ satisfies $\dim L \geq 1$. Since $L$ is a face of $C$, Part~(ii) entails that  $L=\pos (v_i: i \in K)$ for some  $K\subseteq \{1,\dots, n\}$ with $\#K = \dim L\geq 1$. Observe that  $-\sum_{i\in K} v_i \in L$ (since $L$ is a linear subspace) but $-\sum_{i\in K} v_i \notin \pos (v_i: i\in K)$ (since $v_i, i\in K$, are linearly independent). This is a contradiction.
\end{proof}

We say that points $x_1,\ldots, x_n\in \R^d$ are in \emph{general affine position} if for every non-empty subset $K\subseteq\{1,\ldots, n\}$ we have $\dim \aff (x_i: i\in K) = \min (\#K-1, d)$. If $x_1,\ldots, x_n$ are in general affine position, then all proper faces of the polytope $[x_1,\ldots, x_n]$ are simplices - such polytopes are called simplicial. Beta polytopes are simplicial with probability $1$, as shown in the next lemma.

\begin{lemma}[General position a.s.]\label{lem:beta_poi_general_affine}
Let $X_1\sim f_{d,\beta_1},\ldots, X_n\sim f_{d,\beta_n}$ be independent random points in $\R^d$, where $d\geq 2$ and $\beta_1,\ldots, \beta_n\geq -1$ or $d=1$ and $\beta_1,\ldots, \beta_n> -1$.  Then, $X_1,\ldots, X_n$  are in general affine position a.s.  Also, $X_1,\ldots, X_n$  are in general linear position a.s.
\end{lemma}
\begin{proof}
We prove the first claim only, the proof of the second one being analogous.
Take some $k\in \{1,\ldots, d\}$, some pairwise distinct indices $i_1,\ldots, i_{k}\in \{1,\ldots, n\}$ and condition on $X_{i_1},\ldots, X_{i_{k}}$. Let $A := \aff(X_{i_1},\ldots, X_{i_{k}})$ be the affine subspace spanned by these points. The conditional probability that some point $X_j$, $j\notin \{i_1,\ldots, i_{k}\}$, belongs to $A$ is zero since the beta distributions assign mass $0$ to affine hyperplanes (the only exception being the case when $d=1$ and $\beta_j = -1$ for some $j$, but we excluded it). By the formula of total probability, we conclude that $\P[X_j \in \aff (X_{i_1},\ldots, X_{i_{k}})=0]$.
\end{proof}

The next proposition will be used to transform the question whether some points span a face of a beta polytope into the question whether some beta cone is equal to the whole space.

\begin{proposition}[Face events as absorbtion events: polytopal version]\label{prop:face_events_as_absorbtion_events_polytopes}
Let $x_1,\ldots, x_n\in \R^d$ be points in general affine position.
Consider the polytope $P= [x_1,\ldots, x_n]$ and let $A:= \aff(x_1,\ldots, x_k)$, where $k\in \{0,\ldots, \min (n,d)\}$. Then,
$$
F:=[x_1,\ldots, x_k] \text{ is a face of } P \quad \Longleftrightarrow \quad   \pos (\Pi_{A^\perp}x_{k+1}-y,\ldots, \Pi_{A^\perp}x_{n}-y) \neq A^\perp,
$$
where $y= \Pi_{A^\perp}x_{1} = \ldots = \Pi_{A^\perp}x_{k}$.
Without the general affine position assumption, only the implication ``$\Longrightarrow$'' holds.
\end{proposition}
\begin{proof}
``$\Longrightarrow$'': If $F$ is a face of $P$, then there is an affine hyperplane $H = \{x\in\R^d: \langle x, u \rangle = c\}$ with a unit normal vector $u$ and $c>0$ such that $P\subseteq H^- = \{x\in \R^d: \langle x, u \rangle \leq c\}$ and $P\cap H = F$. Since $H\supseteq F$, we also have $H\supseteq \aff F = A$ and $u\in A^\perp$. Also, for $i\in \{k+1,\ldots, n\}$ we have $\langle u, x_i\rangle \leq  c = \langle u, y\rangle$ and hence $\langle u, \Pi_{A^\perp} x_i- y\rangle \leq 0$. It follows that $\pos (\Pi_{A^\perp}x_{k+1}-y,\ldots, \Pi_{A^\perp}x_{n}-y) \neq A^\perp$ since the dual of this cone contains $u$.

``$\Longleftarrow$'': If $\pos (\Pi_{A^\perp}x_{k+1},\ldots, \Pi_{A^\perp}x_{n})\neq A^\perp$, then there is a unit vector $u\in A^\perp$ such that $\langle u, \Pi_{A^\perp}x_{i}\rangle \leq 0$ for all $i\in \{k+1,\ldots, n\}$. It follows that $\langle u, x_i\rangle \leq 0$ for $i\in \{k+1,\ldots, n\}$ and since $\langle u, x_1\rangle = \ldots = \langle u, x_k\rangle = : c$, we conclude that $H:= \{x\in \R^d: \langle x, u\rangle = c\}$ is a supporting hyperplane of $P$. Clearly, $H\supseteq A$, hence $H\cap P\supseteq [x_1,\ldots, x_k]$ (but equality need not hold, in general). Since $x_1,\ldots, x_n$ are in general affine position, the polytope $P$ is simplicial. Hence, $H\cap P$, which is a face of $P$, is a simplex of the form $\conv(x_i: i \in I)$ for some set $I\subseteq \{1,\ldots, n\}$. By the general affine position assumption, $H\cap P$ does not contain points $x_j$ with $j\notin I$. It follows that $x_1,\ldots, x_k$ are vertices of $H\cap P$. Hence, $[x_1,\ldots, x_k]$ is a face of the simplex $H\cap P$ and therefore a face of $P$.
\end{proof}
\begin{example}
For $k=1$ we get: $x_1$ is a vertex of $[x_1,\ldots, x_n]$ if and only if $\pos (x_{2}-x_1,\ldots, x_{n}-x_1) \neq \R^d$. (Note that $A=\{x_1\}$ and $A^\perp=\R^d$.) We used this equivalence in Example~\ref{ex:volume_of_polytope_through_uniform_point}.
\end{example}

The next proposition is a conical version of Proposition~\ref{prop:face_events_as_absorbtion_events_polytopes}. Its proof is similar and we omit it.

\begin{proposition}[Face events as absorbtion events: conical version]\label{prop:face_events_as_absorbtion_events_cones}
Let $v_1,\ldots, v_n\in \R^d$ be vectors in general linear position and let $C= \pos (v_1,\ldots, v_n)$. Let $L:= \lin(v_1,\ldots, v_k)$, where $k\in \{0,\ldots, \min (n,d)-1\}$. Then,
$$
\pos (v_1,\ldots, v_k) \text{ is a face of } C \quad \Longleftrightarrow \quad   \pos (\Pi_{L^\perp}v_{k+1},\ldots, \Pi_{L^\perp}v_{n}) \neq L^\perp.
$$
Without the general linear position assumption, only the implication ``$\Longrightarrow$'' holds.
\end{proposition}

\section{Beta cones and their angles}\label{section:beta_cones}

\subsection{Definition of beta cones}
Let us introduce the main object of interest of the present section.
\begin{definition}[Beta cones]\label{def:beta_cone}
Take some $n\in \N$, $d\in \N$ and real numbers $\b, \b_1,\ldots, \b_n \geq -1$. Let $Z, Z_1,\ldots, Z_n$ be independent random vectors in $\R^d$ with $Z\sim f_{d, \b}$ and $Z_i\sim f_{d,\b_i}$ for all $i=1,\ldots, n$. A \emph{beta cone} is defined as the random polyhedral cone
	\begin{align*}
\sC := \sC_{n,d}^{\beta; \beta_1,\ldots, \beta_n} \coloneqq \pos (Z_1-Z,\dots,Z_n-Z) \subseteq \R^d,
	\end{align*}
see Figure~\ref{fig:construction_of_cones} for a visualization. We then write $\sC \sim \BetaCone (\R^d; \b; \b_1,\dots,\b_n)$.
\end{definition}
Our ultimate goal is to compute explicitly the expected internal and external angles of beta cones and, more generally, their expected conic intrinsic volumes. As we shall show, the tangent cones of beta polytopes become beta cones after factoring out the lineality space. This will allow us to express the expected $f$-vectors and many other quantities related to beta polytopes through the expected angles of beta cones, for which we shall obtain explicit formulas. Note that if $n>d$, a beta cone can be equal $\R^d$ with positive probability (which we shall determine in Theorem~\ref{theo:beta_cones_conic_intrinsic_vol_as_A_B} below).

\begin{lemma}\label{lem:beta_cones_generic_properties}
Consider a beta cone $\sC\sim  \BetaCone(\R^d;\b;\b_1,\dots,\b_n)$. If $d=1$, suppose additionally that $\beta, \beta_1,\ldots, \beta_n \neq -1$.
Then, the following hold with probability $1$:
\begin{itemize}
\item[(i)] $\dim \sC = \min (n,d)$.
\item[(ii)] For every $k\in \{0,\ldots, \min (n,d)-1\}$ every $k$-dimensional face of $\sC$ has the form $\pos (Z_i-Z: i\in K)$ for some subset $K\subseteq \{1,\ldots, n\}$ of size $k$. For $k=0$, we use the convention $\pos (\varnothing) = \{0\}$.
\item[(iii)] Either $\linsp \sC = \{0\}$  or  $\sC = \R^d$.
\end{itemize}
\end{lemma}
\begin{proof}
The vectors $Z, Z_1,\ldots, Z_n$ from the definition of $\sC$ are a.s.\ in general affine position by Lemma~\ref{lem:beta_poi_general_affine}. It follows that the vectors $Z_1-Z, \ldots, Z_n-Z$ are a.s.\ in general linear position. It remains to apply Lemma~\ref{lem:cones_generated_by_vectors_in_gen_lin_pos}.
\end{proof}

\subsection{Tangent cones of beta cones}
The next theorem identifies the tangent cones of beta cones. In order to state it, we need to introduce some notation.
Let
$$\sC=\pos(Z_1-Z,\ldots, Z_n-Z) \sim \BetaCone(\R^d; \b;\b_1,\dots,\b_n)
$$
be a beta cone as  in Definition \ref{def:beta_cone}, with $n\in \N, d\in \N$ and $\b, \b_1,\ldots, \b_n \geq -1$.  If $d=1$, suppose additionally that $\beta, \beta_1,\ldots, \beta_n \neq -1$. For $k \in \{0,1,\dots,\min(d,n)-1\}$ we consider the polyhedral cone
$G \coloneqq \pos (Z_1-Z,\dots,Z_k-Z)$. If $k=0$, we agree to put $G_{\varnothing} = \{0\}$.  If $G$ is a face of $\sC$, we let  $T(G,\sC)$ be the tangent cone of $\sC$ at $G$ and $N(G,\sC) = (T(G,\sC))^\circ$ the normal cone.  By convention, if $G$ fails to be a face of $\sC$ (which is only possible if $\sC$ is full-dimensional, that is $n\geq d$), then $T(G,\sC):= \R^d$ and $N(G,\sC):= \{0\}$.

Further, let $\Cones(\R^k)$ be the set of closed convex cones in $\R^k$. The distance between two cones $C_1,C_2\in \Cones(\R^k)$ is defined as the Hausdorff distance between $C_1\cap \BB^k$ and $C_2\cap \BB^k$; see~\cite[Section~3.2]{amelunxen_phd}. Endowed with this distance, $\Cones(\R^k)$ becomes a compact metric space. A \emph{random closed convex cone} is a random variable $\cC: \Omega \to \Cones (\R^k)$ taking values in this space.
We say that two random closed convex cones $\cC_1$ and $\cC_2$ are isometric if they can be defined on the same probability space $(\Omega, \cF, \P)$ such that for a.e.\ realization $\omega\in \Omega$ the realizations of these cones, $\cC_1(\omega)$ and $\cC_2(\omega)$, are isometric.

\begin{theorem}[Tangent cones of beta cones]\label{theo:representation_angles_and_construction_of_Cone}
In the setting just introduced, the following are true.
\begin{itemize}
\item[(a)] $T(G,\sC)$ is isometric to $\R^k \oplus \BetaCone \left(\R^{d-k}; \b+ \beta_1 + \ldots + \beta_k +\frac{k(d+1)}{2};\b_{k+1}+\frac{k}{2},\dots,\b_n+\frac{k}{2}\right)$;
\item[(b)] $N(G,\sC)$ is isometric to $\BetaCone^\circ \left(\R^{d-k}; \b+\beta_1 + \ldots + \beta_k+ \frac{k(d+1)}{2};\b_{k+1}+\frac{k}{2},\dots,\b_n+\frac{k}{2}\right)$;
\item[(c)]  the isometry type of $G$   is stochastically independent of $T(G,\sC)$ and $N(G,\sC)$.
\end{itemize}
\end{theorem}

Here, $C_1\oplus C_2:= \{(x_1,x_2)\in \R^{d_1+d_2}:x_1\in C_1, x_2\in C_2\}$ denotes the direct orthogonal sum of two polyhedral cones $C_1\subseteq \R^{d_1}$ and $C_2\subseteq \R^{d_2}$. Some additional comments are necessary.

\begin{remark}\label{rem:faces_of_a_beta_cone}
With probability $1$, all $k$-dimensional faces of $\sC$ have the form $\pos(Z_{i_1}-Z, \ldots, Z_{i_k} - Z)$ for some subset $K=\{i_1,\ldots, i_k\} \subseteq \{1,\ldots, n\}$ of size $k$; see Lemma~\ref{lem:beta_cones_generic_properties}.  There is no restriction of generality in considering the case $K= \{1,\ldots, k\}$. In this sense, $G$ is a typical possible face of $\sC$. Note that, by definition,  $G$ is itself a beta cone.  If $n\leq d$, then $G$ is a face of $\sC$ a.s., but we have to keep in mind that in the case when $n>d$, $G$ may or may not be a face with certain positive probabilities. It follows from Claim~(a) that the probability that $\sC$ is not a face coincides with the probability that $\BetaCone (\R^{d-k}; \b+ \beta_1 + \ldots + \beta_k +\frac{k(d+1)}{2};\b_{k+1}+\frac{k}{2},\dots,\b_n+\frac{k}{2}) = \R^{d-k}$.
\end{remark}

\begin{remark}\label{rem:rem:hausdorff_dist_cones}
Part~(c) should be understood as follows.  Consider any Borel-measurable functional $\varphi: \Cones(\R^k)\to\R^m$ defined on the space $\Cones(\R^k)$ of closed convex cones in $\R^k$ and suppose that $\varphi$ is rotationally invariant meaning that $\varphi(OC_1) = \varphi(C_1)$ whenever $C_1\in \Cones(\R^k)$ is a closed convex cone and $O:\R^k \to\R^k$ is an orthogonal transformation. (Examples of rotationally invariant  functionals are solid angles, conic intrinsic volumes, and the functional $C\mapsto \ind_{\{C\neq \R^k\}}$). Then, Part~(c) claims that  $\varphi(G)$ is stochastically independent of $T(G,\sC)$ and $N(G,\sC)$. It is also possible to introduce the space of isometry types of cones as the set of orbits of the natural action of the orthogonal group $O(k)$ on $\Cones(\R^k)$, endowed with the quotient $\sigma$-algebra.
The next two corollaries follow from Part~(c).
\end{remark}

\begin{corollary}
$\alpha(G)$ is stochastically independent of $\beta(G,\sC)$ and $\gamma(G,\sC)$.
\end{corollary}

\begin{corollary}
The distribution of the isometry type of $G$ conditionally on the event that $G$ is a face of $\sC$ coincides with the unconditional distribution of the isometry type of $G$.
\end{corollary}
For the proof of Theorem~\ref{theo:representation_angles_and_construction_of_Cone}, which is visualized in Figure~\ref{fig:construction_of_cones_proof_projection},  we first need to introduce some notation.  Consider the random affine space $A \coloneqq \aff (Z,Z_1,\ldots, Z_k)$. By Lemma~\ref{lem:beta_poi_general_affine}  we have $\dim A = k$ almost surely, hence the orthogonal complement $A^\perp$ has dimension $(d-k)$ almost surely. Note that $A^\perp$ passes through the origin, while $A$ does not need to pass through the origin as it is an \textit{affine} subspace. We project the points $Z_{k+1},\dots,Z_n$ onto $A^\perp$ using the orthogonal projection $\Pi_{A^\perp}:\R^d \to A^\perp$. That is, we define
\begin{align*}
&Y_{k+1} \coloneqq \Pi_{A^\perp} (Z_{k+1}) \in A^\perp,\quad \dots, \quad Y_n \coloneqq \Pi_{A^\perp} (Z_n) \in A^\perp,
\\
&Y:=\Pi_{A^\perp} (Z) = \Pi_{A^\perp} (Z_1) = \ldots = \Pi_{A^\perp} (Z_k);
\end{align*}
see Figure~\ref{fig:construction_of_cones_proof_projection}.
In the definition of $Y$, we used that for every $a\in A$ we have $\Pi_{A^\perp}(a) = \Pi_{A^\perp}(Z)$.  It is convenient to identify $A^\perp$ with $\R^{d-k}$ using the isometry $I_{A^\perp}$; see Section~\ref{subsubsec:identification_affine_subspaces}. We remind that it  satisfies $I_{A^\perp} (0) = 0$.
This leads to the points
\begin{align*}
Y_{k+1}' := I_{A^\perp}(Y_{k+1})\in \R^{d-k}, \quad \dots, \quad Y_{n} := I_{A^\perp}(Y_n)\in \R^{d-k}, \quad  Y' := I_{A^\perp} (Y)\in \R^{d-k}.
\end{align*}
Their joint distribution is characterized in the following proposition.
\begin{proposition}\label{prop:beta_cones_joint_distr_projections}
With the notation just introduced, the following distributional statements hold:
	\begin{itemize}
		\item[(i)] $Y', Y_{k+1}',\dots,Y_{n}'$ are stochastically independent points in $\R^{d-k}$;
		\item[(ii)] $Y_{i}'$ has density $f_{d-k, \b_{i} + \frac{k}{2}}$ for each $i \in \{k+1,\dots,n\}$;
		\item[(iii)] $Y'$ has density $f_{d-k,\b+ \beta_1 + \ldots + \beta_k + \frac{k(d+1)}{2}}$ .
	\end{itemize}
\end{proposition}
\begin{proof}
Recall the notation $0_{A}:=\Pi_{A} (0) = \argmin_{x \in A} \| x\|$ and note that $Y= 0_A$. The point $Y'=I_{A^\perp}(0_A)$ has density $f_{d-k,\beta + \beta_1 + \ldots + \beta_k + \frac{k(d+1)}{2}}$ due to part (c2) of Theorem~\ref{theo:ruben_miles}. Note that $A$ and hence also $A^\perp$ is stochastically independent of $Z_{k+1},\ldots, Z_n$. Let us  condition on some fixed realization of $A^\perp$. The projected and then embedded points $Y_{k+1}', \ldots, Y_n'$ are conditionally stochastically independent. Further, Lemma~\ref{lem:projection} gives that, for each $i \in \{k+1,\ldots, n\}$, the point $Y_{i}'=I_{A^\perp}(\Pi_{A^\perp}(Z_{i}))$ has conditional density $f_{d-k,\b_{i} + \frac{k}{2}}$ as a projection of a beta-distributed point. Since the joint distribution of $Y_{k+1}',\ldots, Y_n'$ is the same for all realizations of $A^\perp$, it is also the unconditional joint distribution of these points (by the formula of total probability). It follows that the points $Y_{k+1}',\ldots, Y_n'$ are stochastically independent and that for $i \in \lbrace k+1,\dots,n\rbrace$ each point $Y_{i}'$ has density $f_{d-k, \b_{i} + \frac{k}{2}}$.  It remains to prove that $Y'$ is stochastically independent of $Y_{k+1}',\ldots, Y_n'$. To this end, observe that $Y'=I_{A^\perp}(0_A)$ is stochastically independent of $A^\perp$ due to part~(c3) of Theorem~\ref{theo:ruben_miles}. Clearly, $Y'$ is also independent of $Z_{k+1},\ldots, Z_n$. Since the random points $Y_{k+1}',\ldots, Y_n'$ are defined as functions of $Z_{k+1}, \ldots, Z_n$ and $A^\perp$ only, we conclude that $Y'$ is independent of $Y_{k+1}',\ldots, Y_n'$. We thus proved (i), (ii), (iii).
\end{proof}
\begin{figure}[h!]\label{fig:proof_theo:representation_angles_and_construction_of_Cone}
	\centering
	\begin{tikzpicture}[scale=2]

		%--------------------------------------------------------------------------
		% 1) Define the main vectors and points
		%--------------------------------------------------------------------------
		\coordinate (O)  at (0,0);
		\coordinate (Z1) at (1.5,2.5);
		\coordinate (Z1real) at (1.2,2.0);
		\coordinate (Z2) at (1.5,1.6);
		\coordinate (Z3) at (2.7,2.3);
		\coordinate (Z3thick) at ($0.6165*(Z3)$);
		\coordinate (Z4) at (2.9,1.2);
		\coordinate (Z4real) at ($1.12*(Z4)$);
		\coordinate (Z4thick) at ($0.752*(Z4)$);
		\coordinate (Z5) at (2.0,1.05);
		\coordinate (origin) at (-2,1.0);
		
		% Red line direction (the “normal”): n = (1.5,2.5)
		% Perpendicular direction in 2D: d = (2.5,-1.5)
		
		%--------------------------------------------------------------------------
		% 2) Draw a trapezoidal 'plane' in green, with perspective
		%--------------------------------------------------------------------------
		% We choose two edges for the trapezoid:
		%   - A "near" edge: from -1*d to +1*d (short, near the origin), slightly shifted
		%   - A "far" edge:  from [-0.4 * d + some * n] to [+0.65 * d + same * n] (larger, behind)
		
		\coordinate (Shift) at (1.5,2.5);
		\coordinate (A) at ($-1*(2.5,-1.5) + (-0.9,-0.9)$);
		\coordinate (B) at ($ 1*(2.5,-1.5) + (-0.9,-0.9)$);
		\coordinate (D) at ($(Shift) - 0.4*(2.5,-1.5) + (0.0,0.0)$);
		\coordinate (C) at ($(Shift) + 0.65*(2.5,-1.5) + (0.0,0.0)$);
		
		% Fill the trapezoid with a translucent green
		\fill[green!40, opacity=0.4] (A) -- (B) -- (C) -- (D) -- cycle
		  node [pos=0, above, scale=0.9, opacity=1, green!50!black] {$A^\perp$};
		\draw[thick, green!50!black] (A) -- (B) -- (C) -- (D) -- cycle;
		
		%--------------------------------------------------------------------------
		% 3) Draw special points, etc.
		%--------------------------------------------------------------------------
		\filldraw[red] (O) circle (0.9pt);   % The point Z
		\filldraw[black] (origin) circle (0.7pt);
		
		%--------------------------------------------------------------------------
		% 4) A long red line through O and Z1, extended
		%--------------------------------------------------------------------------
		\coordinate (ExtNeg)  at ($-0.7*(1.5,2.5)$);
		\coordinate (ExtNeg2) at ($-0.4307*(1.5,2.5)$);
		\coordinate (ExtPos)  at ($ 1.35*(1.5,2.5)$);
		
		\draw[red!80!black, thick]        (ExtNeg) -- (ExtNeg2);
		\draw[red!80!black, thick, dotted](ExtNeg2) -- (0,0);
		\draw[red!80!black, thick]        (0,0) -- (ExtPos)
		  node [pos=1, above, scale=0.9] {$A$};
		
		%--------------------------------------------------------------------------

		% ---------------------------------------------------------------------------
  % 6) Project Z_2,...,Z_5 onto the plane spanned by d = (2.5, -1.5)
  %
  % The direction d is perpendicular to (1.5,2.5) in 2D,
  % so we treat that as the “plane direction” in this 2D analogy.
  % We compute projections via:
  %   proj_d(p) = ( (p·d)/(d·d) ) * d
  %
  % We'll just plug in precomputed results here.
  % ---------------------------------------------------------------------------

  % d·d = 2.5^2 + (-1.5)^2 = 6.25 + 2.25 = 8.5

  % For Z2 = (1.5,1.5), p·d = 3.75 - 2.25 = 1.5 => scale = 1.5/8.5 ~ 0.17647
		
		\coordinate (Z2proj) at
		  ($(0.40, -0.24) + -0.267 *(1.5,2.5)$);
		\coordinate (Z3proj) at
		  ($(0.97, -0.58)+ 0.152 *(1.5,2.5)$);
		\coordinate (Z4proj) at
		  ($(1.48, -0.89) + 0*(1.5,2.5)$);
		\coordinate (Z5proj) at
		  ($(1.007, -0.604) + -0.34 *(1.5,2.5)$);
		  \coordinate (Z4realproj) at
		  ($(1.79, -1.077)+ 0.0*(1.5,2.5)$);
		%--------------------------------------------------------------------------
		
		%--------------------------------------------------------------------------
		% 9) Extended plane #2: spanned by O->Z1 and O->Z3proj The Blue Plane in the back
		%--------------------------------------------------------------------------
		\pgfmathsetmacro{\PlaneScaleTwo}{2.3}
		\coordinate (F1) at (O);
		\coordinate (F2) at ($1.0*(Z1)$);
		\coordinate (F3) at
		  ($1.0*(Z1) + \PlaneScaleTwo*(Z3proj)$);
		\coordinate (F4) at
		  ($\PlaneScaleTwo*(Z3proj)$);
		
		\fill[blue!30, opacity=0.15] (F1) -- (F2) -- (F3) -- (F4) -- cycle;
		\draw[thick, blue!50!black] (F1) -- (F2) -- (F3) -- (F4) -- cycle
		  node [pos=0, right, opacity=1, scale=0.9, blue!50!black] {$T(G,\sC)$};
		
		%--------------------------------------------------------------------------
		% 7) Draw orange projections
		%--------------------------------------------------------------------------

		\draw[orange, thick, loosely dotted] (Z3) -- (Z3proj);
		\fill[orange] (Z3proj) circle (0.6pt);
		
		\draw[orange, thick, loosely dotted] (Z4real) -- (Z4realproj);
		\fill[orange] (Z4realproj) circle (0.6pt);

		% 5) Polygon edges (Z1..Z5) & lines from origin
		%--------------------------------------------------------------------------
		\fill[black!10] (Z1) -- (Z2) -- (Z) -- cycle;
		\fill[black!20] (Z2) -- (Z5) -- (Z) -- cycle;
		\fill[black!30] (Z5) -- (Z4) -- (Z) -- cycle;
		\fill[black!60] (Z1) -- (Z2) -- (Z3thick) -- (Z3) -- cycle;
		\fill[black!55] (Z3thick) -- (Z5) -- (Z4) -- (Z3) -- cycle;
		%\fill[black, opacity=0.45] (Z4) -- (Z5) -- (Z) -- cycle;
		\draw[black] (Z1) -- (Z2) -- (Z5) -- (Z4) -- (Z3) -- cycle;

		\draw[->,thick]        (O) -- (Z1real);
		\draw[->, thick]        (Z3thick) -- (Z3);
		\draw[->,thick]        (O) -- (Z2);
		%\draw[->,dotted,thick](O) -- (Z3);
		\draw[thick](Z4thick) -- (Z4real);
		\draw[->,thick]        (O) -- (Z5);
		\draw[->, thick]        (O) -- (Z4real);

		% 8) Extended plane #1: spanned by O->Z1 and O->Z2proj
		%--------------------------------------------------------------------------
		\pgfmathsetmacro{\PlaneScaleOne}{2.3}
		\coordinate (E1) at (O);
		\coordinate (E2) at ($1.0*(Z1)$);
		\coordinate (E3) at ($1.0*(Z1) + \PlaneScaleOne*(Z2proj)$);
		\coordinate (E4) at ($\PlaneScaleOne*(Z2proj)$);
		
		\fill[blue!30, opacity=0.2] (E1) -- (E2) -- (E3) -- (E4) -- cycle;
		\draw[thick, blue!50!black] (E1) -- (E2) -- (E3) -- (E4) -- cycle;

		\draw[orange, thick, loosely dotted] (Z2) -- (Z2proj);
		\fill[orange] (Z2proj) circle (0.6pt);

		\draw[orange, thick, loosely dotted] (Z5) -- (Z5proj);
		\fill[orange] (Z5proj) circle (0.6pt);
		
		%--------------------------------------------------------------------------
		% 10) The RED region (dual cone):
		%    { w : <w, v2> <= 0 and <w, v3> <= 0 }
		%--------------------------------------------------------------------------
		% Let's define v2, v3 and pick p2, p3 perpendicular with correct sign.
		\coordinate (v2) at (-0.0005, -0.9075);
		\coordinate (v3) at ( 0.9735, -0.2175);
		
		% p2 = a perpendicular to v2
		\coordinate (p2) at ( -0.9075, -0.0005);
		
		% p3 = a perpendicular to v3
		\coordinate (p3) at ( 0.2175, 0.9735);
		
		% Scale factor for the wedge
		\def\T{2.5}
		
		\coordinate (C0) at (0,0);
		\coordinate (C1) at ($\T*(p2)$);
		\coordinate (C3) at ($\T*(p3)$);
		
		% Fill that wedge in red
		\fill[red, opacity=0.2] (C0) -- (C1) -- (C3) -- cycle;
		\draw[red!70!black, thick] (C0) -- (C1) -- (C3) -- cycle;
		
		\node[red!70!black, below left, scale=0.9] at (C1) {$N(G,\sC)$};
		
		%--------------------------------------------------------------------------
		% 8) Labels
		%--------------------------------------------------------------------------
		\node[above, scale=0.8] at ($(Z1real)+(-0.1,0)$)  {$Z_1$};
		\node[left,  scale=0.8] at  ($(Z2)+(0,0.1)$) {$Z_2$};
		\node[above, scale=0.8] at  (Z3)            {$Z_5$};
		\node[right, scale=0.8] at  ($(Z4real)+(0.0,-0.05)$) {$Z_4$};
		\node[left, scale=0.8] at  ($(Z5)+(-0.055,0.05)$) {$Z_3$};
		
		\node[right, scale=0.8] at (origin) {$0$};
		\node[left,  scale=0.8] at (0,0.15)  {$Z$};
		\node[left,  scale=0.8] at (Z2proj)  {$Y_2$};
		\node[above, scale=0.8] at ($(Z3proj)+(-0.12,0.01)$) {$Y_5$};
		\node[right, scale=0.8] at (Z4realproj)  {$Y_4$};
		\node[below, scale=0.8] at (Z5proj)  {$Y_3$};
		\end{tikzpicture}
  	\caption{Visualization of the construction used in the proof of Theorem~\ref{theo:representation_angles_and_construction_of_Cone}. Here, the gray cone is $\sC=\pos(Z_1-Z,\ldots, Z_5-Z) \sim \BetaCone(\R^3; \b;\b_1,\dots,\b_5)$ and the continuous red line is the face $G=\pos (Z_1-Z)$. Its continuation is the affine subspace $A=\aff (Z,Z_1)$, which is also the lineality space of the tangent cone $T(G,\sC)$. The green plane is the orthogonal complement $A^\perp$ of the line $A$. We project each point $Z,Z_1,\dots,Z_5$ onto $A^\perp$ which is identified with  $\R^2$. The projections $Y,Y_2,\dots,Y_5$ are beta-distributed as stated in Proposition~\ref{prop:beta_cones_joint_distr_projections}. Thus, the isometry of $T(G,\sC)$ to $\R \oplus \BetaCone(\R^2; \b+\b_1+2;\b_2+\frac{1}{2},\dots,\b_5+\frac{1}{2})$ becomes clear.
   }
  	\label{fig:construction_of_cones_proof_projection}
\end{figure}
\begin{proof}[Proof of Theorem~\ref{theo:representation_angles_and_construction_of_Cone}]
We observe that the tangent cone of $\sC$ at $G$ can be expressed as
\begin{align}
T(G,\sC)
&=
\lin (Z_1-Z, \ldots, Z_k-Z) + \pos(Z_{k+1}-Z, \ldots, Z_n-Z) \notag
\\
&=
(A-Z) \oplus \pos(Y_{k+1}-Y, \dots, Y_n-Y). \label{eq:beta_cone_tangent_cone_decomposition_direct_sum}
\end{align}
Note that this holds true also if $G$ fails to be a face of $\sC$ -- then $\pos(Y_{k+1}-Y, \dots, Y_n-Y) = A^\bot$ by Proposition~\ref{prop:face_events_as_absorbtion_events_cones}, which is consistent with the convention $T(G, \sC) = \R^d$.

\vspace*{2mm}
Now we can prove (a). Since $A-Z$ is isometric to $\R^k$  and since $I_{A^\perp}$ is an isometry between $A^\perp$ and $\R^{d-k}$, we conclude that $T(G,\sC)$ is isometric to $\R^k \oplus \pos(Y_{k+1}'-Y', \dots, Y_n'-Y')$.  Knowing this, (a) follows from the definition of a beta cone and Proposition~\ref{prop:beta_cones_joint_distr_projections}.

\vspace*{2mm}
In order to prove (b), we note that the normal cone at $G$, which is defined as the polar of the tangent cone, is the polar cone of $\pos(Y_{k+1}-Y, \dots, Y_n-Y)$ taken inside $A^\perp$ as the ambient space.

\vspace*{2mm}
Lastly, in order to prove (c), we recall that $A = \aff (Z,Z_1,\ldots, Z_k)$ and consider the random  points $Z_i':= I_A(Z_i)/ \sqrt{1-h^2(A)}\in \R^k$, for $i=1,\ldots, k$,  as well as $Z':= I_A(Z)/\sqrt{1-h^2(A)}\in \R^k$. By part~(c1) of Theorem~\ref{theo:ruben_miles}, $Z', Z_1',\ldots, Z_k'$ are stochastically independent of $A$ and $A^\perp$. Clearly, these points are also stochastically independent of $Z_{k+1},\dots,Z_{n}$. We conclude that $Z',Z_1',\ldots, Z_k'$ are stochastically independent of $A, A^\perp, Y_{k+1}, \ldots, Y_n$.

Since $I_A$ is an isometry between $A$ and $\R^k$, we realize that the cone  $G = \pos (Z_1-Z,\dots,Z_k-Z)$ is isometric to $\pos(Z_1'-Z',\ldots, Z_k'-Z')$. If $\varphi:\Cones(\R^k) \to \R^m$ is a measurable, rotationally invariant functional on the space of closed convex cones in $\R^k$, then $\varphi(G) = \varphi(\pos(Z_1'-Z',\ldots, Z_k'-Z'))$. On the other hand, $T(G,\sC)= (A-a_0) \oplus \pos(Y_{k+1}-Y, \dots, Y_n-Y)$, where $a_0$ is any point in $A$.   Hence, $\varphi(G)$ is a function of $Z', Z_1',\ldots, Z_k'$, while $T(G,\sC)$ (and hence its polar $N(G,\sC)$) are functions of $A,Y_{k+1}, \ldots, Y_n$. It follows that $\varphi(G)$ is stochastically independent of $T(G,\sC)$ and $N(G,\sC)$.
\end{proof}

Let us now give an explicit formula for the expected \emph{external} angle of beta polytopes. Computing expected \emph{internal} angles is much more difficult and will occupy us in the following.
\begin{theorem}[External angles]\label{theo:beta_cone_external_angle}
Let $\sC=\pos(Z_1-Z,\ldots, Z_n-Z) \sim \BetaCone(\R^d; \b; \b_1,\dots,\b_n)$ be a beta cone as defined in Definition \ref{def:beta_cone}. Suppose that either  $d\geq 2$ or $d=1$ and $\beta, \beta_1,\ldots, \beta_n\neq -1$.
Then,
\begin{align*}
\E \left[\alpha (\sC^\circ)\ind_{\{\sC \neq \R^d\}}\right]
=
\int_{-1}^{+1}  c_{\beta + \frac {d-1}{2}} (1-t^2)^{\beta +  \frac {d-1}{2}}
		\prod_{j =1}^{n} \left(\int_{-1}^t c_{\beta_j + \frac{d-1}{2}} (1-s^2)^{\beta_j + \frac{d-1}{2}}\,\dd s\right) \,\dd t.
	\end{align*}
\end{theorem}
\begin{proof}
We know from Lemma~\ref{lem:beta_cones_generic_properties} that $\sC$ is either $\R^d$ or its lineality space is $\{0\}$. By duality, this means that  $\sC^\circ$ is either $\{0\}$ or it has non-empty interior. Let $D\subseteq \R^d$ be a deterministic cone which either has a non-empty interior or equals $\{0\}$.  In the former case, we have $\alpha(D) = \P[\xi\in D]$, where $\xi$ denotes a vector with standard normal distribution on $\R^{d}$. In the latter case, we have $\alpha (D) = 1$ and  $\P[\xi\in D] = 0$.
So, in both cases,
$$
\alpha (D)\ind_{\{D \neq \{0\}\}} =  \P[\xi\in D].
$$
Now, let $\xi$ be independent of $Z, Z_1,\ldots, Z_n$. Applying the above formula to $D= \sC^\circ$, we get
\begin{align*}
\E \left[\alpha (\sC^\circ)\ind_{\{\sC \neq \R^d\}}\right] =
\E \left[\alpha (\sC^\circ)\ind_{\{\sC^\circ \neq \{0\}\}}\right]
= \P[\xi\in \sC^\circ] = \P [\langle Z_{1} -Z, \xi \rangle \leq 0 , \dots, \langle Z_{n} - Z, \xi \rangle \leq 0].
\end{align*}
Since the joint distribution of $Z,Z_{1},\dots,Z_n$ is rotationally invariant, we can replace $\xi$ by any unit vector $e \in \R^{d}$ and rewrite this equation as
\begin{align*}
\E \left[\alpha (\sC^\circ)\ind_{\{\sC \neq \R^d\}}\right] = \P [\langle Z_{1} -Z, e \rangle \leq 0 , \dots, \langle Z_{n} -Z, e \rangle \leq 0].
\end{align*}
We introduce the random variables $Z_{i}' \coloneqq \langle Z_{i},e \rangle$, for all $i \in \{1,\dots,n\}$, as well as $Z' \coloneqq \langle Z,e \rangle$ to get
\begin{align}\label{eq:representation_external_cone_one_dim}
\E \left[\alpha (\sC^\circ)\ind_{\{\sC \neq \R^d\}}\right]  = \P [Z_{1}' \leq Z', \dots, Z_{n}'\leq Z'].
\end{align}	
Having thus reduced the dimension by $(d-1)$, we use Lemma \ref{lem:projection}  to get
\begin{itemize}
	\item[(i)] $Z',Z_{1}',\dots,Z_n'$ are independent random variables;
	\item[(ii)] $Z_{i}'$ has density $f_{1,\b_{i}+ \frac{d-1}{2}}$, for all $i = 1,\ldots, n$;
	\item[(iii)] $Z'$ has density $f_{1, \b + \frac{d-1}{2}}$.
\end{itemize}
For a fixed $t\in \R$, we have
$$
\P[ Z_1' \leq t, \ldots, Z_n'\leq t] = \prod_{j=1}^{n} \left(\int_{-1}^t c_{\beta_j + \frac{d-1}{2}} (1-s^2)^{\beta_j + \frac{d-1}{2}}\,\dd s\right).
$$
Conditioning on the event that $Z'=t$ and then integrating over $t$ shows that the probability on the right-hand side of \eqref{eq:representation_external_cone_one_dim} is given by
\begin{align*}
\E \left[\alpha (\sC^\circ)\ind_{\{\sC \neq \R^d\}}\right] =&  \int_{-1}^{+1}  c_{\beta + \frac {d-1}{2}}
	(1-t^2)^{\beta +  \frac {d-1}{2}}
	\prod_{j =1}^{n} \left(\int_{-1}^t c_{\beta_j + \frac{d-1}{2}} (1-s^2)^{\beta_j + \frac{d-1}{2}}\,\dd s\right) \,\dd t.
\end{align*}
This completes the proof.
\end{proof}

\subsection{Notation for internal and external angles}
Let us introduce  notation for expected internal and external angles of beta cones in the case when the number of vectors spanning the cone coincides with the dimension of its ambient space.
\begin{definition}[Internal and external quantities]\label{def:int_and_ext}
Consider a beta cone $\sC=\pos(Z_1-Z,\ldots, Z_d-Z) \sim \BetaCone(\R^d; \b; \b_1,\dots,\b_d)$, where $d \geq 2$ and $\beta, \b_1,\ldots, \b_d\geq -1$  The \emph{internal quantities} are defined by
	\begin{equation}\label{eq:int_definition}
		\intern (\b; \b_1,\ldots, \b_d) :=  \E\a (\pos(Z_1-Z,\ldots, Z_d-Z)).
		\end{equation}
The \emph{external quantities} are defined by
	\begin{equation}\label{eq:ext_definition}
		\extern(\b; \b_1,\ldots, \b_d)  :=  \E\a (\pos^\circ(Z_1-Z,\ldots, Z_d-Z)).
	\end{equation}
By convention, for $d=0$, we put $\sC= \{0\}$ and $\intern(\beta; \varnothing) = \extern(\beta; \varnothing) = 1$. For $d=1$, we always put $\intern (\b; \b_1):= \extern(\b; \b_1) := 1/2$ (which is consistent with~\eqref{eq:int_definition} and~\eqref{eq:ext_definition} except when    $\beta = \beta_1 = -1$).
\end{definition}

\begin{convention}
In the course of this work, we shall introduce several functions (such as $\intern$ and $\extern$) whose arguments are multisets. By convention, $\intern(\b; \b_1,\ldots, \b_d)$ means the same as $\intern(\b; \{\b_1,\ldots, \b_d\})$.  Also, instead of listing the elements of a multiset, we sometimes use index sets: If $I=\{i_1,\ldots, i_k\}\subseteq \{1,\ldots,n\}$ is an index  set, then  $\intern(\b; \{\b_i: i\in I\})$ means  the same as $\intern(\b; \b_{i_1},\ldots, \b_{i_k})$.
\end{convention}

Taking $n=d$ in Theorem~\ref{theo:beta_cone_external_angle}  we immediately obtain an explicit formula for the \emph{external} quantities. Note that we can omit the indicator in Theorem~\ref{theo:beta_cone_external_angle} since $\sC \neq \R^d$ a.s.\ for $n=d$.
\begin{corollary}[Formula for $\extern$]\label{cor:external_beta_formula}
Let $d\in \N$ and $\b, \b_1,\ldots, \b_d \geq -1$. If $d=1$, suppose also that $\beta, \beta_1>-1$. Then,
\begin{align}\label{eq:external_beta_formula}
\extern(\beta; \beta_1,\ldots, \beta_d)
=
\int_{-1}^{+1}  c_{\beta + \frac {d-1}{2}} (1-t^2)^{\beta +  \frac {d-1}{2}}
		\prod_{j =1}^{d} \left(\int_{-1}^t c_{\beta_j + \frac{d-1}{2}} (1-s^2)^{\beta_j + \frac{d-1}{2}}\,\dd s\right) \,\dd t.
	\end{align}
\end{corollary}
The task of computing the \emph{internal} quantities will be accomplished only much later in Theorem~\ref{theo:A_integral_representation}.

Originally, we defined $\extern(\beta; \beta_1,\ldots, \beta_d)$ for $\beta, \beta_1,\ldots, \beta_d\geq -1$. However, the double integral on the right-hand side of~\eqref{eq:external_beta_formula} converges in a larger range of parameters, as the next proposition shows.

\begin{proposition}
If  $d\in \N$ and $\beta, \beta_1,\ldots, \beta_d > -(d+1)/2$, the integral on the right-hand side of~\eqref{eq:external_beta_formula} is finite.
\end{proposition}
\begin{proof}
Since $\beta_j + \frac{d-1}{2} >-1$, the integral $\int_{-1}^{+1} (1-s^2)^{\beta_j + \frac{d-1}{2}}\,\dd s$ is finite. Thus, for a suitable constant $C>0$ we have
$$
0\leq \prod_{j =1}^{d} \left(\int_{-1}^t c_{\beta_j + \frac{d-1}{2}} (1-s^2)^{\beta_j + \frac{d-1}{2}}\,\dd s\right) \leq C
\quad
\text{ for all } t\in (-1,1).
$$
Since $\beta + \frac{d-1}{2} >-1$, the integral on the right-hand side of~\eqref{eq:external_beta_formula} is finite.
\end{proof}

\begin{convention}
We use the right-hand side of~\eqref{eq:external_beta_formula} as the definition of $\extern(\beta; \beta_1,\ldots, \beta_d)$ in the range $\beta, \beta_1,\ldots, \beta_d > -(d+1)/2$. (This will be necessary, for example, in Theorem~\ref{theo:beta_cones_faces_alpha_and_gamma}.)
\end{convention}

\begin{proposition}[Limit as $\beta \to\infty$]\label{prop:intern_extern_beta_infinity}
For all $d\in \N_0$ and $\beta_1,\ldots,  \beta_d \geq -1$,
$$
\lim_{\beta \to+\infty} \intern (\b; \b_1,\ldots, \b_d) = 2^{-d},
\qquad
\lim_{\beta \to+\infty} \extern (\b; \b_1,\ldots, \b_d) = 2^{-d}.
$$
\end{proposition}
\begin{proof}
For $d \in \lbrace 0,1\rbrace$ the claim holds by Definition~\ref{def:int_and_ext}. Let $d\geq 2$.
Then, as $\beta\to+\infty$, the random point $Z\sim f_{d,\beta}$ converges weakly to $0$. (To see this, observe that $\lim_{\beta \to +\infty}f_{d,\beta}(x) = 0$ uniformly on $\eps \leq \|x\|\leq 1$ for fixed $\eps>0$.) It follows that
\begin{align*}
&\lim_{\beta \to+\infty} \intern (\b; \b_1,\ldots, \b_d) = \E\a \pos(Z_1,\ldots, Z_d),
\\
&\lim_{\beta \to+\infty} \extern (\b; \b_1,\ldots, \b_d) = \E\a \pos^\circ(Z_1,\ldots, Z_d) = \P[\langle U, Z_1\rangle\leq 0,\ldots, \langle U, Z_d\rangle \leq 0],
\end{align*}
where $U$ is uniformly distributed on $\bS^{d-1}$ and independent of $Z_1,\ldots, Z_d$.
Since the joint distribution of $(Z_1,\ldots, Z_d)$ is the same as that of $(\pm Z_1,\ldots, \pm Z_d)$ with arbitrary choice of signs, it follows that $\E\a \pos(Z_1,\ldots, Z_d) = \E\a \pos^\circ(Z_1,\ldots, Z_d) = 2^{-d}$ and the proof is complete.
\end{proof}

It is much more difficult to derive a formula for the internal quantities. Our plan is as follows. First we derive a formula for the expected conic intrinsic volumes of beta cones. Then we use properties of the conic intrinsic volumes to derive nonlinear McMullen-type relations between internal and external quantities. This will occupy us in the remainder of Section~\ref{section:beta_cones}.  In Section~\ref{sec:beta_cones_explicit_internal_angles} we shall then resolve these relations to arrive at explicit formulas for the internal quantities.

\subsection{Expected conic intrinsic volumes of beta cones in terms of \texorpdfstring{$\intern$}{Int} and \texorpdfstring{$\extern$}{Ext}}
We recall that $\upsilon_0(C),\ldots, \upsilon_d(C)$ denote the conic intrinsic volumes of a polyhedral cone $C$; see Section~\ref{subsec:angles_conic_intrinsic_defs}.
\begin{theorem}[Angles of faces and normal faces]\label{theo:beta_cones_faces_alpha_and_gamma}
Consider a beta cone $\sC=\pos(Z_1-Z,\ldots, Z_n-Z)\sim\BetaCone(\R^d; \b; \b_1,\dots,\b_n)$ with $d\in \N$ and $n\in \N$. If $d=1$, suppose additionally that  $\beta, \beta_1,\ldots, \beta_n>-1$.
For $k \in \{0,1,\dots,\min(n,d)-1\}$ let $G \coloneqq \pos (Z_1-Z,\dots,Z_k-Z)$ be a possible face of $\sC$ with the usual convention that $G= \{0\}$ if $k=0$.
\begin{itemize}
\item[(a)] The expected angle of $G$ is $\E \alpha(G) = \intern(\b + \frac{d-k}{2}; \b_1 + \frac{d-k}{2}, \dots, \b_k + \frac{d-k}{2})$.
\item[(b)] The expected external angle of $\sC$ at $G$ is given by
\begin{multline}\label{eq:beta_cones_external_angles_explicit}
\E \left[\gamma (G,\sC)\ind_{\{G \text{ is a face of } \sC\}}\right]
=
\int_{-1}^{+1}  c_{\beta +\frac{d}{2} + \sum_{i=1}^{k} \left(\beta_i + \frac{d}{2}\right) -\frac {1}{2}} (1-t^2)^{\beta +\frac{d}{2} + \sum_{i=1}^{k} \left(\beta_i + \frac{d}{2}\right) -\frac {1}{2}}
\\ \times \prod_{j =k+1}^{n} \left(\int_{-1}^t c_{\beta_j + \frac{d-1}{2}} (1-s^2)^{\beta_{j}+ \frac{d-1}{2}}\,\dd s\right) \,\dd t.
\end{multline}
Hence, we can express the above expectation  as an external quantity with $n-k$ arguments after the semicolon:
\begin{multline}\label{eq:beta_cones_external_angles_as_external_quant}
\E \left[\gamma (G,\sC)\ind_{\{G \text{ is a face of } \sC\}}\right]
\\
= \extern\left(\beta +\frac{d}{2} + \sum_{i=1}^{k} \left( \beta_i+\frac{d}{2}\right)  - \frac{n-k}{2}; \beta_{k+1} + \frac d2 - \frac{n-k}{2},\ldots, \beta_{n} + \frac d2 - \frac{n-k}{2} \right). 
\end{multline}
\item [(c)] The random variable $\alpha(G)$ is stochastically independent of the pair $(\gamma(G,\sC),\ind_{\{G\in \cF_j(\sC)\}})$.
\end{itemize}
\end{theorem}

For the proof we need a formula for expected angles of projected cones which is based on~\cite[Theorem~2.3.4]{SchneiderBook}.

\begin{proposition}\label{prop:angles_of_rotationally_invariant_projections}
Let $\sD$ be a random polyhedral cone in $\R^d$ whose distribution is rotationally invariant in the following sense: $O \sD$ has the same law as $\sD$ for every orthogonal transformation $O:\R^d \to \R^d$. Suppose that $\sD\neq \lin \sD$ with probability $1$.  Take some $k\in \{1,\ldots, d\}$ and let $\Pi_k : \R^{d} \to \R^{k}$ be the orthogonal projection onto the first $k$ coordinates defined by
\begin{align*}
\Pi_k (x_1,\dots,x_{d}) = (x_1,\dots,x_{k}).
\end{align*}
Then,
$$
\E\left[\upsilon_k(\Pi_k \sD)\right]=  \sum_{j=k}^d \E \upsilon_j(\sD).
$$
If, moreover, $\dim \sD= k$ with probability $1$, then
$
\E \alpha (\Pi_k \sD) = \E \alpha (\sD).
$
\end{proposition}
\begin{proof}
Let $\mathbf{O}$ be a random Haar-distributed matrix in the special orthogonal group $SO(d)$. Then, $\mathbf{O} \R^k$ is a random uniformly distributed $k$-dimensional linear subspace of $\R^d$. The orthogonal projection onto this subspace is $\mathbf{O} \Pi_k \mathbf{O}^{-1}$.  By~\cite[Theorem~2.3.4]{schneider_book_convex_cones_probab_geom}, for every deterministic polyhedral cone $C\subseteq \R^d$ we have
$$
\E\left[\upsilon_k(\mathbf{O} \Pi_k \mathbf{O}^{-1}  C)\right] = \sum_{j = k}^d \upsilon_j(C).
$$
Let now $\mathbf{O}$ be independent of $\sD$. Then, $\mathbf{O}^{-1} \sD$ has the same law as $\sD$ and we have
$$
\E\left[\upsilon_k(\Pi_k \sD)\right] = \E\left[\upsilon_k(\Pi_k \mathbf{O}^{-1}\sD)\right] = \E\left[\upsilon_k(\mathbf{O} \Pi_k \mathbf{O}^{-1}  \sD)\right]
=
\sum_{j = k}^d \E \upsilon_j(\sD).
$$
If $\dim \sD= k$ with probability $1$, then $\E \alpha (\Pi_k \sD) = \E \upsilon_k (\Pi_k \sD)  = \sum_{j = k}^d \E \upsilon_j(\sD) = \E \upsilon_k (\sD) = \E \alpha (\sD)$.
\end{proof}

\begin{proof}[Proof of Theorem~\ref{theo:beta_cones_faces_alpha_and_gamma}]
To prove (a), we let $\Pi_k : \R^{d} \to \R^{k}$ be the orthogonal projection as above and consider the projected points $W \coloneqq \Pi_k Z$, $W_{1} \coloneqq \Pi_k Z_{1}, \dots, W_{k} \coloneqq \Pi_k Z_{k}$. By Lemma~\ref{lem:projection}, their joint distribution satisfies:
\begin{itemize}
	\item[(1)] $W,W_{1},\dots,W_{k}$ are independent points in $\R^{k}$;
	\item[(2)] For each $i \in \{1,\dots,k\}$, the point $W_{i}$ has density $f_{k,\b_{i}  + \frac{d-k}{2}}$;
	\item[(3)] $W$ has density $f_{k,\b + \frac{d-k}{2}}$.
\end{itemize}
It follows that
\begin{align*}
\Pi_k G
&=
\Pi_k (\pos(Z_1-Z,\dots,Z_k-Z))
=
\pos(W_1-W,\dots,W_k-W)
\\
&\sim
\BetaCone\left(\R^{k};\b + \frac{d-k}{2};\b_{1}+\frac{d-k}{2},\dots,\b_k+\frac{d-k}{2}\right).
\end{align*}
Note that $\Pi_k G$ is a full-dimensional beta cone in $\R^k$.
Proposition~\ref{prop:angles_of_rotationally_invariant_projections} gives us that the expected angles of $G$ and its projection are equal, hence
\begin{align*}
\E \alpha (G)
=
\E \alpha(\Pi_k G)
=
\intern \left(\b + \frac{d-k}{2};\b_1 + \frac{d-k}{2}, \dots, \b_k + \frac{d-k}{2}\right),
\end{align*}
where the last equation holds by definition of the internal quantities.

\vspace*{2mm}
To prove Part (b), we recall from Part (b) of Theorem~\ref{theo:representation_angles_and_construction_of_Cone}  that the normal cone $N(G,\sC)$ is isometric to the polar of
$$
\BetaCone \left(\R^{d-k}; \b+\beta_1 + \ldots + \beta_k+ \frac{k(d+1)}{2};\b_{k+1}+\frac{k}{2},\dots,\b_n+\frac{k}{2}\right).
$$
By Proposition~\ref{prop:face_events_as_absorbtion_events_cones}, this beta cone is not equal to $\R^{d-k}$ if and only if  $G$ is a face of $\sC$.
Further, applying Theorem~\ref{theo:beta_cone_external_angle} to a cone $\sD\sim \BetaCone^\circ(\R^{m}; \gamma; \gamma_1,\dots,\gamma_{n-k})$ we get
\begin{align*}
\E \left[\alpha (\sD^\circ) \ind_{\{\sD \neq \R^{m}\}}\right]
=
\int_{-1}^{+1}  c_{\gamma + \frac {m-1}{2}} (1-t^2)^{\gamma +  \frac {m-1}{2}}
		\prod_{j =1}^{n-k} \left(\int_{-1}^t c_{\gamma_j + \frac{m-1}{2}} (1-s^2)^{\gamma_j + \frac{m-1}{2}}\,\dd s\right) \,\dd t.
\end{align*}
Applying this formula with $m:=d-k$, $\gamma := \b+\beta_1 + \ldots + \beta_k+ \frac{k(d+1)}{2}$ and $\gamma_j:=\b_{j}+\frac{k}{2}$ and observing that
$$
\gamma + \frac{m-1}{2} = \beta + \frac{d}{2} + \sum_{i=1}^{k} \left( \beta_i +  \frac {d}{2} \right) -\frac{1}{2},
\qquad
\gamma_j + \frac{m-1}{2} = \beta_j + \frac{d-1}{2}.
$$
gives
\begin{multline*}
\E \left[\gamma (G,\sC)\ind_{\{G \text{ is a face of } \sC\}}\right]
	=
	\int_{-1}^{+1}  c_{\beta +\frac{d}{2} + \sum_{i=1}^{k} \left(\beta_i + \frac{d}{2}\right) -\frac {1}{2}} (1-t^2)^{\beta +\frac{d}{2} + \sum_{i=1}^{k} \left(\beta_i + \frac{d}{2}\right) -\frac {1}{2}}
	\\ \times \prod_{j =k+1}^{n} \left(\int_{-1}^t c_{\beta_j + \frac{d-1}{2}} (1-s^2)^{\beta_{j}+ \frac{d-1}{2}}\,\dd s\right) \,\dd t
	\end{multline*}
and proves~\eqref{eq:beta_cones_external_angles_explicit}.  To prove~\eqref{eq:beta_cones_external_angles_as_external_quant}, recall from Corollary~\ref{cor:external_beta_formula} that
$$
\extern(\delta; \delta_{k+1},\ldots, \delta_n)
=
\int_{-1}^{+1}  c_{\delta + \frac {n-k-1}{2}} (1-t^2)^{\delta +  \frac {n-k-1}{2}}
		\prod_{j =1}^{n-k} \left(\int_{-1}^t c_{\delta_j + \frac{n-k-1}{2}} (1-s^2)^{\delta_j + \frac{n-k-1}{2}}\,\dd s\right) \,\dd t.
$$
It order for this integral to coincide with the above formula for $\E[\gamma (G,\sC)\ind_{\{G \text{ is a face of } \sC\}}]$, we set
$$
\delta := \beta +\frac{d}{2} + \sum_{i=1}^{k}\left( \beta_i +\frac{d}{2}\right) - \frac{n-k}{2},
\qquad
\delta_j:= \beta_{j} + \frac d2 - \frac{n-k}{2}.
$$
Finally, Part~(c) follows from Part (c) of Theorem~\ref{theo:representation_angles_and_construction_of_Cone}.
\end{proof}

\begin{theorem}[$\E \upsilon_k(\sC)$ through $\intern$ and $\extern$]\label{theo:beta_cones_conic_intrinsic_vol_as_ext_int}
Consider a beta cone $\sC=\pos(Z_1-Z,\ldots, Z_n-Z)\sim\BetaCone(\R^d;\b;\b_1,\dots,\b_n)$  with $d\in\N$ and $n\in \N$. If $d=1$, suppose also that $\beta, \beta_1,\ldots, \beta_n>-1$.  Then, for all $k \in \{0,\dots,\min(n,d)-1\}$,
\begin{multline}\label{eq:beta_cones_conic_intrinsic_vol_as_ext_int}
\E \left[\upsilon_k(\sC)\right]
=
\sum_{\substack{I \subseteq \{1,\dots,n\} \\ \# I=k}}
\intern\left(\b + \frac{d-k}{2}; \;\; \left\{ \b_i + \frac{d-k}{2}: i\in I\right\}\right)
\\
\times
\extern\left(\beta+\frac{d}{2} + \sum_{i\in I} \left(\beta_i+\frac{d}{2}\right)  - \frac{n-k}{2}; \;\; \left\{\beta_{i} + \frac {d}{2} - \frac{n-k}{2}: i\in I^c\right\}\right).
\end{multline}
Recall the notation  $I^c := \{1,\ldots, n\}\bsl I$.
\end{theorem}
\begin{proof}
For $k\in\{0,\ldots, \min(n,d)-1\}$ we have
$$
\E \left[\upsilon_k(\sC)\right] = \E \left[\upsilon_k(\sC) \ind_{\{\sC\neq \R^d\}}\right].
$$
Recall the formula given in~\eqref{eq:upsilon_def},
$$
\upsilon_k(\sC) =
\sum_{G\in\cF_k(\sC)} \alpha(G)  \gamma (G,\sC),
\qquad
k\in \{0,\ldots,d\}.
$$
Now observe  that on the event $\{\sC\neq \R^d\}$, each face $G \in \cF_k(\sC)$ has the form $G_I := \pos (Z_{i_1} - Z,\dots,Z_{i_k}-Z )$ for some set $I = \{i_1,\dots,i_k\} \subseteq \{1,\ldots, n\}$ of size $k$.
Multiplying the previous formula by $\ind_{\{\sC\neq \R^d\}}$ and taking expectation  gives
\begin{align*}
\E \left[\upsilon_k(\sC)\right]
=
\E \left[ \sum_{G \in \cF_k(\sC)} \a(G) \g(G,\sC)  \ind_{\{\sC\neq \R^d\}}\right]
=
\E \left[\sum_{\substack{I \subseteq \{1,\dots,n\} \\ \# I=k}} \a(G_I) \g(G_I,\sC)  \ind_{\{G_I\in \cF_k(\sC), \sC\neq \R^d\}}\right].
\end{align*}
Next we claim that
$$
\g(G_I,\sC)  \ind_{\{G_I\in \cF_k(\sC), \sC\neq \R^d\}} = \g(G_I,\sC) \ind_{\{G_I\in \cF_k(\sC)\}} .
$$
Indeed, if $G_I$ is a face of $\sC$, then $\sC\neq \R^d$ and both sides are equal to $\g(G_I,\sC)$. If $G_I$ is not a face of $\sC$, then $\gamma(G_I,\sC) = 1$ but the indicator functions make  both sides vanish.
It follows that
\begin{align*}
\E \left[\upsilon_k(\sC)\right]
=
\sum_{\substack{I \subseteq \{1,\dots,n\} \\ \# I=k}} \E \left[\a(G_I) \g(G_I,\sC) \ind_{\{G_I\in \cF_k(\sC)\}} \right]
=
\sum_{\substack{I \subseteq \{1,\dots,n\} \\ \# I=k}} \E \left[\a(G_I)\right] \E \left[\g(G_I,\sC) \ind_{\{G_I\in \cF_k(\sC)\}} \right],
\end{align*}
where we used that the random variables $\alpha(G_I)$ and $\gamma(G_I,\sC)\ind_{\{G_I\in \cF_k(\sC)\}}$ are independent by Part~(c) of Theorem~\ref{theo:beta_cones_faces_alpha_and_gamma}. Moreover, by the same theorem we have
\begin{align*}
\E \alpha(G_I)
&=
\intern\left(\b + \frac{d-k}{2}; \;\; \left\{ \b_i + \frac{d-k}{2}: i\in I\right\}\right),
\\
\E \left[\gamma(G_I,\sC)\ind_{\{G_I\in \cF_k(\sC)\}}\right]
&=
\extern\left(\beta+\frac{d}{2} + \sum_{i\in I} \left(\beta_i+\frac{d}{2}\right)  - \frac{n-k}{2}; \;\; \left\{\beta_{i} + \frac {d}{2} - \frac{n-k}{2}: i\in I^c\right\}\right).
\end{align*}
Inserting these values completes the proof of~\eqref{eq:beta_cones_conic_intrinsic_vol_as_ext_int}.
\end{proof}

\begin{theorem}[Expected internal angle through $\intern$ and $\extern$]
Let $\sC\sim\BetaCone(\R^d;\b;\b_1,\dots,\b_n)$ be a beta cone with $n\geq d$. If $d=1$, suppose additionally that  $\beta,\beta_1,\ldots,\beta_n >-1$.  Then,
\begin{multline*}
\E \left[\upsilon_d(\sC)\right]
=
1
-
\sum_{\substack{I \subseteq \{1,\dots,n\} \\ \# I\leq d-1}}
\intern\left(\b + \frac{d-k}{2}; \;\; \left\{ \b_i + \frac{d-k}{2}: i\in I\right\}\right)
\\
\times
\extern\left(\beta+\frac{d}{2} + \sum_{i\in I} \left(\beta_i+\frac{d}{2}\right)  - \frac{n-k}{2}; \;\; \left\{\beta_{i} + \frac {d}{2} - \frac{n-k}{2}: i\in I^c\right\}\right).
\end{multline*}
\end{theorem}
\begin{proof}
Use the relation
$
\E [\upsilon_d(\sC)]
=
1 - \sum_{k=0}^{d-1} \E[\upsilon_k(\sC)]$ and apply Theorem~\ref{theo:beta_cones_conic_intrinsic_vol_as_ext_int} to each summand on the right-hand side.
\end{proof}

\subsection{Nonlinear relations for \texorpdfstring{$A$}{A}- and \texorpdfstring{$B$}{B}-quantities}\label{subsec:A-and_B-quantities}
The main result of this section is a set of nonlinear relations between the expected internal and external angles of beta cones. Recall that these are denoted by $\intern$ and $\extern$. To express these relations in an especially simple form,  we introduce the following notation.

\begin{definition}[$A$ and $B$-quantities]\label{def:A_and_B}
For $p\in \N$ and $\g, \g_1,\ldots, \g_p \geq \frac{p}{2}-1$ we define the \emph{$A$- and $B$-quantities} by
\begin{align}
A\left(\g; \g_1,\ldots, \g_p\right)
&\coloneqq
\intern \left( \g - \frac{p}{2} ; \g_1 - \frac{p}{2}, \ldots,  \g_p - \frac{p}{2} \right),
\label{eq:A_through_Int}\\
B\left(\g; \g_1,\ldots, \g_p\right)
&\coloneqq
\extern \left( \gamma - \sum_{i=1}^p\gamma_i - \frac{p}{2}; \g_1 - \frac{p}{2}, \ldots, \g_p - \frac{p}{2}  \right)
\label{eq:B_through_Ext}
\end{align}
assuming in the $B$-case additionally that $\gamma - \sum_{i=1}^p\gamma_i \geq  \frac{p}{2}-1$.
Also, for  $p=0$ we  define $A(\g; \varnothing)= B(\g; \varnothing) = 1$, which is consistent with the convention $\intern(\beta; \varnothing) = \extern(\beta; \varnothing) = 1$.
\end{definition}

Conversely, we can express $\intern$ and $\extern$ through $A$ and $B$ as follows: For $\beta, \beta_1,\ldots, \beta_p\geq -1$,
\begin{align}
\intern \left(\b;\b_1, \ldots,  \b_p\right)
&=
A\left(\b + \frac p2; \b_1+ \frac p2,\ldots, \b_p + \frac p2\right), \label{eq:A_through_Int_converse}
\\
\extern\left(\b;\b_1, \ldots,  \b_p\right)
&=
B\left(\b  + \frac {p} 2 + \sum_{i=1}^p \left(\b_i + \frac p2\right); \b_1+ \frac p2,\ldots, \b_p + \frac p2\right). \label{eq:B_through_Ext_converse}
\end{align}

Thanks to Corollary~\eqref{cor:external_beta_formula} we already have an explicit formula for $\extern$ and hence for the $B$-quantities. One of our main tasks in the following is to derive a formula for the $A$-quantities. The first step is to establish relations between  $A$- and $B$-quantities. These relations are somewhat similar to the nonlinear angle-sum relations of McMullen~\cite{mcmullen} (though they do not follow from these directly).

\begin{theorem}[Nonlinear relations]\label{theo:beta_cones_relations_ext_int_second_version}
	Let $d\in \N_0$ and $\gamma, \gamma_1,\ldots, \gamma_d \geq \frac{d}{2}-1$.  Then,
	\begin{align}
	\sum_{I \subseteq \{1,\dots,d\}}
	A \left(\g ; \;\; \left\{ \g_i : i\in I\right\}\right) B \left(\g + \sum_{k=1}^d \g_k ; \;\; \left\{\g_{i} : i \in I^c\right\}\right)
	=
	1,
	\end{align}
	as well as
	\begin{align}
	\sum_{I \subseteq \{1,\dots,d\}}
	(-1)^{\#I}
	A \left(\g ; \;\; \left\{ \g_i : i\in I\right\}\right) B \left(\g + \sum_{k=1}^d \g_k  ; \;\; \left\{\g_{i} : i \in I^c\right\}\right)
	=
	\begin{cases}
	1, &\text{if } d=0,\\
	0, &\text{if } d\geq 1.
	\end{cases}
	\end{align}
\end{theorem}

\begin{proof}
For $d=0$ both relations just state that $1=1$ since, by convention, $A(\g; \varnothing)= B(\g; \varnothing) = 1$.
Let in the following $d\in \N$ and $\beta, \beta_1,\ldots, \beta_d \geq -1$. Consider a beta cone $\sC=\pos(Z_1-Z,\ldots, Z_d-Z)\sim\BetaCone(\R^d;\b;\b_1,\dots,\b_d)$, where now $n=d$. Theorem~\ref{theo:beta_cones_conic_intrinsic_vol_as_ext_int} implies that
\begin{multline}\label{eq:beta_cones_conic_intrinsic_vol_as_ext_int_case_n_eq_d}
\E \left[\upsilon_k(\sC)\right]
=
\sum_{\substack{I \subseteq \{1,\dots,d\} \\ \# I=k}}
\intern\left(\b + \frac{d-k}{2}; \;\; \left\{ \b_i + \frac{d-k}{2}: i\in I\right\}\right)
\\
\times
\extern\left(\beta +\frac{d}{2} + \sum_{i\in I} \left( \beta_i +\frac{d}{2}\right)- \frac{d-k}{2}; \;\; \left\{\beta_{i}+\frac{k}{2} : i\in I^c\right\}\right).
\end{multline}
whenever $k\in \{0,\ldots, d-1\}$. Moreover, in the special case $n=d$ we are dealing with, the same formula stays true for $k=d$ since then  the external quantity becomes $1$ by the convention $\extern(\beta; \varnothing) = 1$ and the formula turns into
$
\E \left[\upsilon_d(\sC)\right]
=
\intern\left(\b;  \beta_1,\ldots, \beta_d\right) \cdot 1,
$
which is the definition of the internal quantities.  We conclude that~\eqref{eq:beta_cones_conic_intrinsic_vol_as_ext_int_case_n_eq_d} holds for all $k\in \{0,\ldots, d\}$.

Observe that $\sC\neq \R^d$ since it is generated by $d$ vectors. Moreover, $\sC$ is a.s.\ not a linear subspace of $\R^d$ by Lemma~\ref{lem:beta_cones_generic_properties}.   Therefore, the relations
$$
\sum_{k=0}^d \E \upsilon_k(\sC) = 1,
\qquad
\sum_{k=0}^d (-1)^k \E \upsilon_k(\sC) = 0
$$
hold.  After taking into account~\eqref{eq:beta_cones_conic_intrinsic_vol_as_ext_int_case_n_eq_d}, these identities take the form
\begin{multline*}
\sum_{I \subseteq \{1,\dots,d\}}
\intern\left(\b + \frac{d-(\#I)}{2}; \;\; \left\{ \b_i + \frac{d-(\#I)}{2}: i\in I\right\}\right)
\\
\times
\extern\left(\beta+\frac{d}{2} + \sum_{i\in I}\left( \beta_i+\frac{d}{2}\right)    - \frac{\#I^c}{2}; \;\; \left\{\beta_{i} + \frac{\#I}{2}: i\in I^c\right\}\right)
=
1,
\end{multline*}
as well as
\begin{multline*}
\sum_{I \subseteq \{1,\dots,d\}}
(-1)^{\#I}
\intern\left(\b + \frac{d-(\#I)}{2}; \;\; \left\{ \b_i + \frac{d-(\#I)}{2}: i\in I\right\}\right)
\\
\times
\extern\left(\beta+\frac{d}{2} + \sum_{i\in I}\left( \beta_i+\frac{d}{2}\right)   - \frac{\#I^c}{2}; \;\; \left\{\beta_{i} + \frac{\#I}{2}: i\in I^c\right\}\right)
=
\begin{cases}
1, &\text{if } d=0,\\
0, &\text{if } d\geq 1.
\end{cases}
\end{multline*}
Setting $\gamma: = \beta + \frac d2$ and $\gamma_j:= \beta_j + \frac d2$ and using \eqref{eq:A_through_Int} as well as \eqref{eq:B_through_Ext} to represent the internal and external quantities through the $A$- and $B$-quantities  gives the claim.
\end{proof}

The next theorem states that the $A$-quantities are uniquely determined by the nonlinear relations of Theorem~\ref{theo:beta_cones_relations_ext_int_second_version}.
\begin{theorem}[Solution to nonlinear relations]\label{theo:uniqueness_of_solution}
Fix some $d\in \N_0$ and $\gamma, \gamma_1,\ldots, \gamma_d \geq \frac{d}{2}-1$. For a set $I = \{i_1,\dots, i_k\} \subseteq \{1,\ldots, d\}$ we write $S_I:= \sum_{i\in I} \gamma_i$.
Let $f_J$, where $J\subseteq\{1,\ldots, d\}$, be $2^d$ unknowns. For every $I_0\subseteq \{1,\ldots, d\}$ we consider an equation
\begin{align}\label{eq:linear_system}
	\sum_{I \subseteq I_0} (-1)^{\#I}  B \left(\g + S_{I_0}  ; \;\; \left\{\g_{i} : i \in I_0\bsl I\right\}\right) f_I = \begin{cases}
		1, &\text{if } I_0=\varnothing,\\
		0, &\text{if } \# I_0 \geq 1,
		\end{cases}
\end{align}
so that in total we have a system of $2^d$ linear equations in the unknowns $f_J$. The solution to this system is unique and is  given by $f_\varnothing =1$ and
\begin{align}\label{eq:unique_solution_linear_system}
	f_{I_0} = (-1)^{\#I_0} \sum_{p=0}^{\#I_0-1} \sum_{\varnothing \subsetneq I_p \subsetneq \dots \subsetneq I_1 \subsetneq I_0} (-1)^{p+1}
\prod_{\ell=0}^p  B\left(\g+S_{I_\ell}; \{\gamma_i: i\in I_\ell \bsl I_{\ell+1}\}\right),
\end{align}
for all non-empty $I_0\subseteq \{1,\ldots, d\} = [d]$, where we put $I_{p+1} := \varnothing$.
\end{theorem}
\begin{proof}
We show the statement using induction over $d \in \N_0$. If  $d=0$,  then~\eqref{eq:linear_system} simplifies to $f_{\varnothing} = 1$.
Since the above was trivial, let us also show the case $d=1$. In this case, \eqref{eq:linear_system} gives $f_\varnothing = 1$ and
\begin{align*}
B(\gamma+S_{\lbrace 1 \rbrace}; \g_1) f_\varnothing - B(\gamma+S_{\lbrace 1 \rbrace}; \varnothing) f_{\lbrace 1 \rbrace} = 0,
\end{align*}
	which leads to $f_{\lbrace 1 \rbrace} =B(\gamma+S_{\lbrace 1 \rbrace}; \g_1)$ and coincides with the expression for $f_{\lbrace 1 \rbrace}$ given by \eqref{eq:unique_solution_linear_system}. Let us now assume that the statement of the theorem holds with $d$ replaced by $d-1$. It suffices to prove~\eqref{eq:unique_solution_linear_system} for $I_0= [d]$.  We start with \eqref{eq:linear_system}, where we separate one term from the sum on the left-hand side to get
	\begin{align*}
		(-1)^d B(\gamma + S_{[d]}; \varnothing) f_{[d]} + \sum_{I \subsetneq [d]} (-1)^{\# I} B(\gamma + S_{[d]}; \lbrace \gamma_i : i \in  I^c\rbrace) = 0.
	\end{align*}
This gives us an expression for $f_{[d]}$,
	\begin{align}\label{proof_unique_Solution_expression_for_f_I}
		f_{[d]} =& \sum_{I \subsetneq [d]} (-1)^{\# I+d+1} B(\gamma +S_{[d]}; \lbrace \gamma_i : i \in  I^c\rbrace) f_I \notag \\
		=& (-1)^{d+1} B(\gamma + S_{[d]};\lbrace \gamma_i : i \in  [d]\rbrace)+ \sum_{\varnothing \subsetneq I \subsetneq [d]} (-1)^{\# I+d+1} B(\gamma +S_{[d]}; \lbrace \gamma_i : i \in  I^c\rbrace) f_I.
	\end{align}
	We can apply the induction assumption on the term $f_I$ which simplifies the last  sum  to
	\begin{multline*}
		\sum_{\varnothing \subsetneq I \subsetneq [d]} (-1)^{\# I+d+1}  B(\gamma +S_{[d]}; \lbrace \gamma_i : i \in  I^c\rbrace) f_I = \sum_{\varnothing \subsetneq I \subsetneq [d]}\Biggl( (-1)^{\# I+d+1} B(\gamma +S_{[d]}; \lbrace \gamma_i : i \in  I^c\rbrace)  \\
		\times (-1)^{\#I} \sum_{p=0}^{\#I-1} \sum_{\varnothing \subsetneq I_p \subsetneq \dots \subsetneq I_1 \subsetneq I} (-1)^{p+1}
		\prod_{\ell=0}^p  B\left(\g+S_{I_\ell}; \{\gamma_i: i\in I_\ell \bsl I_{\ell+1}\}\right) \Biggr)\\
		= (-1)^{d+1} \sum_{q=1}^{d-1} (-1)^q \sum_{\varnothing \subsetneq I_q' \subsetneq \dots \subsetneq I'_1 \subsetneq [d]} \prod_{j=0}^{q}  B\left(\g+S_{I_j'}; \{\gamma_i: i\in I_j' \bsl I_{j+1}'\}\right),
	\end{multline*}
where we have set $q=p+1$ in the last line. Let us explain why we can combine the sums in the last equation. In the first and second line, we first choose any non-empty, strict subset $I$ of $[d]$. Then, we have summed over all chains of length at most $p$ of this set $I$. Notice that the length of this chain might be $0$; in this case, the third sum is taken only over the previously chosen set $I$. But this is the same as taking all chains of length $q \coloneqq p+1$ of the set $[d]$. Finally, we notice that the summand for $q=0$ is actually given by the left-hand side of \eqref{proof_unique_Solution_expression_for_f_I}. Combining both parts finally gives that
\begin{align*}
	f_{[d]} = (-1)^{d} \sum_{q=0}^{d-1}  \sum_{\varnothing \subsetneq I_q' \subsetneq \dots \subsetneq I'_1 \subsetneq [d]} (-1)^{q+1}\prod_{j=0}^{q}  B\left(\g+S_{I_j'}; \{\gamma_i: i\in I_j' \bsl I_{j+1}'\}\right)
\end{align*}
which completes the induction and gives the claim.
\end{proof}

The next theorem gives a first, though non-satisfactory, formula for the $A$-quantities.  A more useful formula will be obtained in Section~\ref{sec:beta_cones_explicit_internal_angles}.
\begin{theorem}[$A$-quantities through $B$-quantities]\label{theo:A_expressed_through_B}
	For $d \in \N$ and $\g,\g_1,\dots,\g_d \geq -\frac{1}{2}$, we have
	\begin{align}\label{eq:A_through_B}
		A(\gamma; \gamma_1, \dots, \gamma_d) = (-1)^{d} \sum_{p=0}^{d-1} \sum_{\varnothing \subsetneq I_p \subsetneq \dots \subsetneq I_1 \subsetneq [d]} (-1)^{p+1}
\prod_{\ell=0}^p  B\left(\g+S_{I_\ell}; \{\gamma_i: i\in I_\ell \bsl I_{\ell+1}\}\right),
	\end{align}
	where $I_0 := \{1,\dots,d\}$ by convention.
\end{theorem}
\begin{proof}
We use Theorem~\ref{theo:uniqueness_of_solution}.
By Theorem~\ref{theo:beta_cones_relations_ext_int_second_version}, $f_{I} = A(\gamma; I)$ defines a solution to~\eqref{eq:linear_system}.
On the other hand, the solution to~\eqref{eq:linear_system} is unique and given by~\eqref{eq:unique_solution_linear_system}. Taking there $I_0 = \lbrace 1, \dots, d\rbrace$ we get~\eqref{eq:A_through_B}.
\end{proof}

\section{Explicit formula for angles of beta cones}\label{sec:beta_cones_explicit_internal_angles}
The aim of this section is to derive explicit formulas for $A$-quantities and thus also for the internal quantities. This will be done as follows. We shall introduce two new types of quantities, denoted by $a$ and $b$. Both, $a$ and $b$ will be defined in an analytic way as certain integrals. In Theorem~\ref{theo:relations_a_b} we shall show that $a$- and $b$-quantities satisfy nonlinear relations similar to those that hold for $A$- and $B$-quantities, as stated in Theorem~\ref{theo:beta_cones_relations_ext_int_second_version}. Since there is a simple way to relate $b$ to $B$ and since the nonlinear relations \textit{uniquely} define $A$ (respectively, $a$), we shall finally conclude that $A$ can be expressed through $a$; see Theorem~\ref{lem:A_through_a_B_through_b}.

\subsection{\texorpdfstring{$a$}{a}- and \texorpdfstring{$b$}{b}-quantities and their properties}\label{subsec:a_and_b}
First, recall  the notation
$$
c_{\beta}:= c_{1,\beta} =\frac{\Gamma(\beta + \frac 32)}{\sqrt \pi\; \Gamma(\beta+1)}.
$$
We shall need the following  known formulas (see, e.g.,  Equation~(3.14) in~\cite{kabluchko_angles_of_random_simplices_and_face_numbers}):
\begin{align}\label{eq:int_cos_cosh}
\int_{-\pi/2}^{+\pi/2} (\cos x)^{h-1} \dd x
=
\int_{-\infty}^{+\infty} (\cosh y)^{-h} \dd y
=
\frac{\sqrt \pi\; \Gamma \left(\frac{h}{2}\right)}{\Gamma\left(\frac{h+1}{2}\right)}
=
\frac{1}{c_{\frac{h-2}{2}}}
,
\quad
\Re h >0.
\end{align}
\begin{definition}
For $\beta\in \C$ with  $\Re \beta \geq -1$ we introduce the function
\begin{align}\label{eq:def_F_as_cos_integral}
F_\b (x) \coloneqq \int_{-\pi /2}^{x} \cos^\b (y) \dd y, \qquad x\in \sG:= \C \bsl ((-\infty, - \pi/2) \cup (\pi/2, +\infty)).
\end{align}
\end{definition}
The integral is taken along any contour connecting $-\pi/2$ to $x$ in the simply connected domain $\sG$. The intervals $(-\infty, -\pi/2)$ and $(\pi/2,+\infty)$ have been excluded since the function $x\mapsto \cos^\beta x$ has branching points at $\pi/2 + \pi m$, $m\in \Z$, if $\beta$ is non-integer.  On $\sG$, this function is a univalued analytic function.   Let us record a formula for $F_\b (\ii x)$ with $x\in \R$:
\begin{equation}\label{eq:F_a_imaginary_argument}
F_\b (\ii x)
=
\frac{1}{2c_{(\b-1)/2}} + \int_{0}^{\ii x} \cos^{\b} (z) \dd z
=
\frac{1}{2c_{(\b-1)/2}} + \ii \cdot  \int_{0}^{x} \cosh^{\b}(y) \dd y,
\end{equation}
where we used that  $\int_{-\pi /2}^{0} \cos^{\b}(z)  \dd z = \frac{1}{2c_{(\b-1)/2}}$.  For $\Re \beta \geq 0$, it follows that $F_\b (\ii x) = O(|x|\eee^{|x| \Re \beta})$  as $x\to \pm\infty$.

\begin{definition}[$a$-quantities]\label{dfn:a_quantities}
For $d\in \N_0$, real numbers $\alpha_1,\ldots, \alpha_d \geq 0$ and $\alpha > \alpha_1+\ldots +\alpha_d$  we define the quantities
\begin{align}
a(\a;\a_1,\dots, \a_d) &= \int_{-\infty}^{+\infty} \cosh^{-\a} (x) \prod_{j=1}^{d} F_{\a_j}(\ii x) \dd x, \label{eq:def_a}\\
a_1(\a;\a_1,\dots, \a_d) &= \ii\cdot  \int_{-\infty}^{+\infty} \cosh^{-\a-1} (x) \sinh(x) \prod_{j=1}^{d} F_{\a_j}(\ii x) \dd x.\label{eq:def_a_1}
\end{align}
\end{definition}

\begin{lemma}
The integrals in~\eqref{eq:def_a} and~\eqref{eq:def_a_1} converge absolutely.
\end{lemma}
\begin{proof}
As $x\to \pm \infty$, we have $\prod_{j=1}^{d} F_{\a_j}(\ii x) = O(|x|^d \eee^{(\alpha_1 + \ldots +\alpha_d) |x|})$. On the other hand,  $|\cosh^{-\alpha} x| =O(\eee^{-\alpha |x|})$ and $|\cosh^{-\a-1} (x) \sinh(x)| = O(\eee^{- \alpha |x|})$. Condition $\alpha > \alpha_1 + \ldots + \alpha_d$ implies the absolute convergence of the integrals in~\eqref{eq:def_a} and~\eqref{eq:def_a_1}.
\end{proof}
\begin{remark}
A slightly more elaborate argument shows that~\eqref{eq:def_a} and~\eqref{eq:def_a_1} define analytic functions in $\alpha, \alpha_1,\ldots, \alpha_d$ on the domain defined by $\Re \alpha_1 >  0, \ldots, \Re \alpha_d > 0$ and $\Re \alpha > \Re \alpha_1 + \ldots + \Re \alpha_d$.
\end{remark}

\begin{example}[On $d=0$ and $d=1$]\label{ex:a_def_case_d_0}
If $d=0$, the empty product is defined as $1$ and we evaluate $a(\alpha;\varnothing) = \int_{-\infty}^{+\infty} \cosh^{-\a} (x)  \dd x = c_{\frac \a2-1}^{-1}$ and $a_1(\alpha; \varnothing) = 0$, both for $\alpha >0$. For $d=1$, \eqref{eq:F_a_imaginary_argument} implies
\begin{align*}
a(\a;\a_1) &= \int_{-\infty}^{+\infty} \cosh^{-\a} (x)  F_{\a_1}(\ii x) \dd x = \frac{1}{2c_{\frac{\a_1-1}2}} \int_{-\infty}^{+\infty} \cosh^{-\a} (x) \dd x  = \frac{1}{2 c_{\frac{\a_1-1}2} c_{\frac{\alpha-2}{2}}}.
\end{align*}
As it will turn out, the ranges imposed in Definition~\ref{dfn:a_quantities} are too restrictive for $d=0$ and $d=1$. We therefore agree to use the formulas from the present example to define $a(\alpha; \varnothing)$ and $a(\alpha; \alpha_1)$ as meromorphic functions of $\alpha, \alpha_1\in \C$.
\end{example}

We can transform the integral representation of $a(\a;\a_1,\dots,\a_d)$ in the following way.
\begin{proposition}[Formulas for $a$-quantities]\label{prop:other_integral_representation_of_a}
For  $\alpha_1,\ldots, \alpha_d \geq 0$ and $\alpha >\alpha_1+\ldots + \alpha_d$ we have
\begin{align}
a(\a;\a_1,\dots, \a_d)
&=
\int_{-1}^{+1} (1-t^2)^{\frac{\a}{2}-1} \prod_{j=1}^{d} \left(\frac{1}{2c_{(\a_j-1)/2}} + \ii \int_{0}^{t} (1-s^2)^{-\frac{\a_j}{2} -1} \dd s
\right) \dd t  \label{eq:a_quant_integral_1}
\\
&=
\int_{-\pi/2}^{+\pi/2} \cos^{\a-1} (u) \prod_{j=1}^{d} \left(\frac{1}{2c_{(\a_j-1)/2}} + \ii \int_{0}^{u} \cos^{-\a_j - 1} (v) \dd v
\right) \dd u. \label{eq:a_quant_integral_2}
\end{align}
\end{proposition}
\begin{proof}
First using~\eqref{eq:def_a} and~\eqref{eq:F_a_imaginary_argument},  then substituting $s= \tanh(y)$  and after that $t = \tanh(x)$ gives
\begin{align*}
a\left(\a;\a_1,\dots,\a_d\right)
&= \int_{-\infty}^{+\infty} \cosh^{-\a}(x) \prod_{j=1}^{d} \left(\frac{1}{2c_{(\a_j-1)/2}} + \ii \cdot  \int_{0}^{x} \cosh^{\a_j} (y) \dd y   \right) \dd x \\
	&= \int_{-\infty}^{+\infty} \cosh^{-\a}(x) \prod_{j=1}^{d} \left(\frac{1}{2c_{(\a_j-1)/2}} + \ii \int_{0}^{\tanh(x)} (1-s^2)^{-\frac{\a_j}{2} -1} \dd s   \right) \dd x \\
	&= \int_{-1}^{+1} (1-t^2)^{\frac{\a}{2} -1} \prod_{j=1}^{d} \left(\frac{1}{2c_{(\a_j-1)/2}} + \ii \int_{0}^{t} (1-s^2)^{-\frac{\a_j}{2} -1} \dd s   \right) \dd t.
\end{align*}
This proves~\eqref{eq:a_quant_integral_1}. To prove~\eqref{eq:a_quant_integral_2}, we substitute $t= \sin u$ and $s = \sin v$.
\end{proof}

\begin{theorem}[Recurrence relations for $a$- and $a_1$-quantities] \label{theo:dfn_a_quantities}
Let $d\in \N_0$,  $\alpha_1,\ldots, \alpha_d \geq 0$ and $\alpha >\alpha_1+\ldots + \alpha_d$. Then, the following relations hold:
\begin{align}
&\a \cdot a(\a;\a_1,\dots,\a_d) - (\a+1) \cdot a(\a+2; \a_1,\dots,\a_d) = \sum_{j=1}^{d} a_1(\a-\a_j;\widehat{\a_j}), \label{eq:a_as_sum_a_1} \\
&\a \cdot a_1(\a; \a_1,\dots,\a_d) = - \sum_{j=1}^{d} a(\a - \a_j; {\widehat \alpha_j}), \label{eq:a_1_as_sum_a}
	\end{align}
where the notation $a(\gamma;\widehat{\a_j}):=a(\gamma; \a_1,\dots,\a_{j-1},\a_{j+1},\dots,\a_d)$ indicates that $\alpha_j$ has been removed from the list $\alpha_1,\ldots, \alpha_d$.
\end{theorem}
\begin{proof}
For $d=0$, the right-hand sides of both equations are $0$ by definition and the equations can be verified directly using Example~\ref{ex:a_def_case_d_0}.
Let in the following $d\in \N$. Observe that
\begin{equation}\label{eq:a_rec_part_int_boundary_0}
\lim_{x\to\pm \infty} \biggl( \cosh^{-\a}x \cdot \prod_{j=1}^{d} F_{\a_j}(\ii x)\biggr)  = \lim_{x\to\pm \infty} \biggl(\cosh^{-(\a+1)} (x) \sinh(x) \cdot \prod_{j=1}^{d} F_{\a_j}(\ii x)\biggr)= 0
\end{equation}
by the assumption $\alpha > \alpha_1+\ldots +\alpha_d$. To prove  \eqref{eq:a_as_sum_a_1}, we shall evaluate the following integral in two different ways:
$$
I:= \int_{-\infty}^{+\infty} \cosh^{-(\a+1)} (x) \sinh(x) \biggl(\prod_{j=1}^{d} F_{\a_j}(\ii x)\biggr)' \dd x.
$$
On the one hand, we use the generalized product rule for the derivative in  the following way,
\begin{align*}
I=& \int_{-\infty}^{+\infty} \cosh^{-(\a+1)} (x) \sinh(x) \sum_{j=1}^{d} \biggl( \ii\cdot \cos^{\a_j}(\ii x) \prod_{p:p\neq j} F_{\a_p}(\ii x)\biggr) \dd x \\
		=& \ii \cdot \sum_{j=1}^{d}
		\int_{-\infty}^{+\infty} \cosh^{-(\a-\a_j+1)} (x) \sinh (x) \prod_{p:p\neq j} F_{\a_p}(\ii x) \dd x \\
		=& \sum_{j=1}^{d} a_1(\a-\a_j;\a_1,\dots,\a_{j-1},\a_{j+1},\dots,\a_d).
	\end{align*}
On the other hand, using integration by parts leads to
\begin{align*}
I
=& - \int_{-\infty}^{+\infty} \cosh^{-(\a+2)} (x) \left(\cosh^2(x)-(\a+1) \sinh^2(x)\right) \prod_{j=1}^{d} F_{\a_j}(\ii x) \dd x \\
=&\a \cdot \int_{-\infty}^{+\infty} \cosh^{-\a} (x)\prod_{j=1}^{d} F_{\a_j}(\ii x) \dd x - (\a+1) \cdot \int_{-\infty}^{+\infty} \cosh^{-\a-2} (x)\prod_{j=1}^{d} F_{\a_j}(\ii x) \dd x \\
=& \a \cdot a(\a;\a_1,\dots,\a_d) - (\a+1) \cdot a(\a+2; \a_1,\dots,\a_d).
\end{align*}
Note that the boundary term in the partial integration formula vanishes by~\eqref{eq:a_rec_part_int_boundary_0}. Comparing both expressions for $I$ gives us \eqref{eq:a_as_sum_a_1}. To show \eqref{eq:a_1_as_sum_a}, we start with the definition of $a_1$ and then  again use partial integration and the product rule for the derivative,
\begin{align*}
\alpha \cdot a_1(\a;\a_1,\dots, \a_d)
&= \ii \alpha \cdot  \int_{-\infty}^{+\infty} \cosh^{-\a-1} (x) \sinh(x) \prod_{j=1}^{d} F_{\a_j}(\ii x) \dd x
\\
&= - \ii  \cdot  \int_{-\infty}^{+\infty} \left(\cosh^{-\a}x\right)' \prod_{j=1}^{d} F_{\a_j}(\ii x) \dd x
\\
&= + \ii  \cdot  \int_{-\infty}^{+\infty} \left(\cosh^{-\a}x\right) \Biggl(\prod_{j=1}^{d} F_{\a_j}(\ii x)\Biggr)' \dd x
\\
&= -  \int_{-\infty}^{+\infty} \left(\cosh^{-\a}x\right) \sum_{j=1}^{d} \Biggl(\cosh^{\a_j}(x) \prod_{p:p\neq j} F_{\a_p}(\ii x)\Biggr) \dd x
\\
&=
- \sum_{j=1}^{d} a\left(\a - \a_j; {\widehat \alpha_j}\right).
\end{align*}
Note that again the boundary term in the partial integration formula vanishes by~\eqref{eq:a_rec_part_int_boundary_0}.
\end{proof}
\begin{definition}[$b$-quantities]\label{dfn:b_quantities}
For $d\in \N_0$ and  $\alpha_1,\ldots, \alpha_d \geq 0$ we define the quantities
	\begin{align}
		b(\a;\a_1,\dots, \a_d) &= \int_{-\pi /2}^{+\pi /2} \cos^\a (x) \prod_{j=1}^{d} F_{\a_j}(x) \dd x, \qquad \alpha>-1, \label{eq:def_b}\\
		b_1(\a;\a_1,\dots, \a_d) &=  \int_{-\pi /2}^{+\pi /2} \cos^{\a-1} (x) \sin(x) \prod_{j=1}^{d} F_{\a_j}(x) \dd x, \qquad \alpha>0.  \label{eq:def_b_1}
	\end{align}
\end{definition}
\begin{remark}
The integrals in~\eqref{eq:def_b} and~\eqref{eq:def_b_1} converge absolutely since $F_{\alpha_j}(x)$ is bounded on $(-\frac \pi 2, \frac \pi 2)$. A slightly more elaborate argument shows that~\eqref{eq:def_b} and~\eqref{eq:def_b_1} define analytic functions of their arguments on the domain  $\Re \alpha_1 > 0, \ldots, \Re \alpha_d > 0$ and $\Re \alpha > -1$ (for $b$) or $\Re \alpha > 0$ (for $b_1$).
\end{remark}
\begin{example}[$d=0$ and $d=1$]\label{ex:b:def_d_0}
For $d=0$, we have $b(\alpha;\varnothing) = \int_{-\pi/2}^{+\pi/2} \cos^{\a} (x)  \dd x = c_{(\a-1)/2}^{-1}$ and $b_1(\alpha;\varnothing)=0$,  where the integral has been made explicit in~\eqref{eq:int_cos_cosh}. For $d=1$, we observe that the function $F_{\alpha_1}(x) - F_{\alpha_1}(0)$ is odd and use~\eqref{eq:int_cos_cosh} and~\eqref{eq:F_a_imaginary_argument} to obtain
$$
b(\a;\a_1) = \int_{-\pi /2}^{+\pi /2} \cos^\a (x) F_{\a_1}(x) \dd x = F_{\a_1}(0)\int_{-\pi /2}^{+\pi /2} \cos^\a (x) \dd x
=
\frac 1 {2 c_{\frac{\alpha_1-1}{2}} c_{\frac{a-1}{2}}}.
$$
As it will turn out, the assumptions on the arguments imposed in Definition~\eqref{dfn:b_quantities} are too restrictive when $d=0$ or $d=1$.  Instead, we shall use the formulas from the present example to define $b(\alpha;\varnothing)$ and $b(\alpha;\alpha_1)$ as meromorphic functions of $\alpha,\alpha_1 \in \C$.  
\end{example}
\begin{proposition}[Formulas for $b$-quantities]\label{prop:other_integral_representation_of_b}
For  $\alpha_1,\ldots, \alpha_d \geq 0$ and $\alpha > -1$ we have
\begin{align}
b(\a;\a_1,\dots, \a_d)
&=
\int_{-1}^{+1} (1-t^2)^{\frac{\a-1}{2}} \prod_{j=1}^{d} \left(\frac{1}{2c_{(\a_j-1)/2}} + \int_{0}^{t} (1-s^2)^{\frac{\a_j-1}{2}} \dd s
\right) \dd t \label{eq:b:formula_1}
\\
&=
\int_{-\infty}^{+\infty} \cosh^{-\a-1} (u) \prod_{j=1}^{d} \left(\frac{1}{2c_{(\a_j-1)/2}} + \int_{0}^{u} \cosh^{-\a_j-1} (v) \dd v
\right) \dd u. \label{eq:b:formula_2}
\end{align}
\end{proposition}
\begin{proof}
By definition,
\begin{align}
b(\a;\a_1,\dots, \a_d) = \int_{-\pi /2}^{+\pi /2} \cos^\a (x) \prod_{j=1}^{d} \left(\frac{1}{2c_{(\a_j-1)/2}} + \int_{0}^{x} \cos^{\a_j} (y) \dd y\right) \dd x.
\end{align}
To see that~\eqref{eq:b:formula_1} is equivalent to this formula, apply in~\eqref{eq:b:formula_1} the substitution $t = \sin x$, $s= \sin y$. To pass from~\eqref{eq:b:formula_1} to~\eqref{eq:b:formula_2}, make the substitution $t= \tanh u$, $s= \tanh v$ in~\eqref{eq:b:formula_1}.
\end{proof}

\begin{theorem}[Recurrence relations for $b$- and $b_1$-quantities]
Let $d\in \N_0$ and  $\alpha_1,\ldots, \alpha_d \geq 0$. Assuming that $\alpha>1$ for the first relation and  $\alpha>0$ for the second one, we have
\begin{align}
&\a \cdot b(\a;\a_1,\dots,\a_d) - (\a-1) \cdot b(\a - 2; \a_1,\dots,\a_d)  = - \sum_{j=1}^{d} b_1(\a+\a_j;\widehat{\alpha_j}), \label{eq:b_as_sum_b_1} \\
&\a \cdot b_1(\a; \a_1,\dots,\a_d) = \sum_{j=1}^{d} b(\a + \a_j; \widehat{\alpha_j}), \label{eq:b_1_as_sum_b}
	\end{align}
with the usual convention that $\widehat{\alpha_j}$ means that $\alpha_j$ has been excluded from the list $\alpha_1,\ldots, \alpha_d$.
\end{theorem}
\begin{proof}
For $d=0$ the right-hand sides of both equations are defined as $0$ and both identities can be verified using Example~\ref{ex:b:def_d_0}.
Let $d\in \N$. The proof follows the same idea as the proof for Theorem~\ref{theo:dfn_a_quantities}.
For~\eqref{eq:b_as_sum_b_1},  we shall evaluate the following integral in two different ways:
$$
J:= \int_{-\pi /2}^{+\pi /2} \cos^{\a-1} (x) \sin(x) \Biggl( \prod_{j=1}^{d} F_{\a_j}(x)\Biggr)' \dd x.
$$
On the one hand, by the generalized product rule for the derivative,
\begin{align*}
J  %\int_{-\pi /2}^{\pi /2} \cos^{\a-1} (x) \sin(x) \biggl( \prod_{j=1}^{d} F_{\a_j}(x)\biggr)' \dd x
&=
\int_{-\pi /2}^{+\pi /2} \cos^{\a-1} (x) \sin(x) \sum_{j=1}^{d} \biggl(\cos^{\a_j}(x) \prod_{i:i\neq j} F_{\a_j}(x) \biggr)\dd x
=
\sum_{j=1}^{d}b_1(\a+\a_j;\widehat{\a_{j}}).
\end{align*}
On the other hand, we can apply integration by parts on \eqref{eq:def_b_1}, which gives
	\begin{align*}
%		\int_{-\pi /2}^{\pi /2} \cos^{\a-1} (x) \sin(x) \biggl(\prod_{j=1}^{d} F_{\a_j}(x)\biggr)' \dd x
J=& - \int_{-\pi /2}^{+\pi /2} \left(\cos^{\a} (x) - (\a-1) \cos^{\a-2}(x)\sin^2(x)\right) \prod_{j=1}^{d} F_{\a_j}(x) \dd x \\
		=& (\a-1) \cdot  \int_{-\pi /2}^{+\pi /2} \cos^{\a-2} (x)  \prod_{j=1}^{d} F_{\a_j}(x) \dd x -\a \cdot \int_{-\pi /2}^{+\pi /2} \cos^{\a} (x)  \prod_{j=1}^{d} F_{\a_j}(x) \dd x \\
		=& (\a-1)\cdot  b(\a -2; \a_1,\dots,\a_d)- \a \cdot b(\a;\a_1,\dots,\a_d).
	\end{align*}
Combining both equations gives us \eqref{eq:b_as_sum_b_1}.
To show \eqref{eq:b_1_as_sum_b}, we use partial integration and the product rule for the derivative,
\begin{align*}
\alpha \cdot b_1(\a;\a_1,\dots, \a_d)
&=
\alpha\int_{-\pi /2}^{+\pi /2} \cos^{\a-1} (x) \sin(x) \prod_{j=1}^{d} F_{\a_j}(x) \dd x
=
- \int_{-\pi /2}^{+\pi /2} \left(\cos^{\a}(x)\right)'  \prod_{j=1}^{d} F_{\a_j}(x) \dd x
\\
&=
\int_{-\pi /2}^{+\pi /2} \cos^{\a}(x)  \biggl(\prod_{j=1}^{d} F_{\a_j}(x)\biggr)' \dd x
=
\int_{-\pi /2}^{+\pi /2} \cos^{\a}(x)  \sum_{j=1}^d \biggl(\prod_{p: p\neq j}  \cos^{\a_j}(x) F_{\a_j}(x) \biggr) \dd x
\\
&=
\sum_{j=1}^{d} b(\a + \a_j; \widehat{\alpha_j}).
\end{align*}
Note that the boundary terms in partial integration vanish since $\cos (\pm \frac \pi2)= 0$ and $\alpha >0$.
\end{proof}

Next we want to state recurrence relations involving $a$- and $b$-quantities only, without $a_1$ and $b_1$.
\begin{theorem}[Recurrence relations for $a$- and $b$-quantities]\label{theo:recurrence_relations_a_b}
Let $d\in \N_0$ and   $\alpha_1, \ldots, \alpha_d \geq 0$. Assuming $\alpha>0$ for the first relation and $\alpha>1$ for the second one, we have
\begin{align}
&(\a+1)\cdot a(\a+2;\a_1,\dots,\a_d) - \a\cdot a(\a;\a_1,\dots,\a_d) = \sum_{\substack{j_1,j_2\in \{1,\ldots, d\}\\ j_1 \neq j_2}} \frac{a(\a-\a_{j_1}-\a_{j_2};\widehat{\a_{j_1}},\widehat{\a_{j_2}})}{\a-\a_{j_1}}, \label{eq:a_first_relation}
\\
&
(\a-1) \cdot  b(\a-2;\a_1,\dots,\a_d) - \a \cdot b(\a;\a_1,\dots,\a_d)
= \sum_{\substack{j_1,j_2\in \{1,\ldots, d\}\\ j_1 \neq j_2}} \frac{b(\a+\a_{j_1}+\a_{j_2};\widehat{\a_{j_1}},\widehat{\a_{j_2}})}{\a+\a_{j_1}},
\label{eq:b_first_relation}
\end{align}
where the notation such as $a(\gamma;\widehat{\a_{j_1}}, \widehat{\a_{j_2}}) := a(\gamma;\{\a_1,\ldots, \a_d\}\bsl \{\a_{j_1}, \a_{j_2}\})$ indicates that $\a_{j_1}, \a_{j_2}$ have been excluded from the list $\alpha_1,\ldots,\alpha_d$.
\end{theorem}
\begin{proof}
We first show \eqref{eq:a_first_relation}. For this, we begin with \eqref{eq:a_as_sum_a_1}, and then apply \eqref{eq:a_1_as_sum_a} to every $a_1$-term in the sum,
\begin{align*}
\a a(\a;\a_1,\dots,\a_d)
&= (\a+1)\cdot a(\a+2;\a_1,\dots,\a_d) + \sum_{j_1=1}^{d} a_1 (\a-\a_{j_1};\widehat{\a_{j_1}})
\\
&= (\a+1)\cdot a(\a+2;\a_1,\dots,\a_d) - \sum_{j_1=1}^d \frac{1}{\a - \a_{j_1}}\sum_{\substack{j_2\in \{1,\ldots, d\}\\ j_2\neq j_1}} a(\a - \a_{j_1} - \a_{j_2}; \widehat{\a_{j_1}},\widehat{\a_{j_2}}).
\end{align*}
(For $d=0$ and $d=1$, the sums on the right-hand side vanish).
Rearranging the terms gives \eqref{eq:a_first_relation}. To show \eqref{eq:b_first_relation}, we start with \eqref{eq:b_as_sum_b_1} and then use \eqref{eq:b_1_as_sum_b} to get
\begin{align*}
(\a-1)\cdot b(\a-2;\a_1,\dots,\a_d) &= \a \cdot b(\a;\a_1,\dots,\a_d) + \sum_{j_1=1}^{d} b_1 (\a+\a_{j_1};\widehat{\a_{j_1}})
\\
								 &= \a \cdot b(\a;\a_1,\dots,\a_d) + \sum_{j_1=1}^d\frac{1}{\a + \a_{j_1}} \sum_{\substack{j_2\in \{1,\ldots, d\}\\ j_2\neq j_1}} b(\a + \a_{j_1} + \a_{j_2}; \widehat{\a_{j_1}},\widehat{\a_{j_2}}),
\end{align*}
which gives \eqref{eq:b_first_relation} after rearranging the terms.
\end{proof}

\subsection{Nonlinear relations for \texorpdfstring{$a$}{a}- and \texorpdfstring{$b$}{b}-quantities}\label{subsec:a_and_b_nonlinear_rel}
Let us fix some $d\in \N_0$ and some numbers $\lambda_1,\ldots, \lambda_d \geq 0$. We agree to use the following notational shorthand: For a set $I = \{i_1,\dots, i_k\} \subseteq \{1,\ldots, d\}$ we write
$$
S_I:= \sum_{i\in I} \lambda_i,
\qquad
a(\lambda; I):= a(\lambda; \lambda_{i_1},\ldots, \lambda_{i_k}),
\qquad
b(\lambda; I):= b(\lambda; \lambda_{i_1},\ldots, \lambda_{i_k}).
$$
Let $\zeta\in \{+1,-1\}$. We are interested in the following sum,
\begin{equation}\label{eq:s_d_def_sum_ab}
s_{d,\zeta}(\lambda) \coloneqq \sum_{I \subseteq \{1,\dots,d\}}
	\zeta^{\#I} a\left( \lambda + S_I +2; I \right) \left(\lambda+ S_I + 1\right) b \left( \lambda + S_I;  I^c \right),
\qquad \lambda>-1.
\end{equation}
For convenience, we suppressed the dependence on $\lambda_1,\ldots, \lambda_d$ in the notation.  Our aim is to prove that  $s_{d,+1}(\lambda)$ and $s_{d,-1}(\lambda)$ are constant in $\lambda$, more precisely,
$$
s_{d,-1}(\lambda)
=
\begin{cases}
2\pi, &\text{if } d=0,\\
0, &\text{if } d\geq 1,
\end{cases}
\qquad \text{ and } \qquad
s_{d,+1}(\lambda)
=
2\pi \prod_{i=1}^d  c_{\frac{1}{2}\lambda_i - \frac{1}{2}}^{-1}.
$$
This will be done in two steps. First, we shall prove that $s_{d,\zeta}(\lambda)$ is periodic with period $2$. To this end, we shall use the recurrence relations for $a$- and $b$-quantities stated in Theorem~\ref{theo:recurrence_relations_a_b}.  Then we shall compute the limits of $s_{d,+1}(\lambda)$ and $s_{d,-1}(\lambda)$ as $\lambda\to +\infty$ by determining the asymptotic behavior of the $a$- and $b$-quantities.
\begin{theorem}[Periodicity]\label{theo:s_d_is_periodic}
For all $d\in \N_0$, $\lambda >-1$ and $\zeta\in \{+1,-1\}$ it holds that
\begin{align*}
s_{d,\zeta}(\lambda) = s_{d,\zeta}(\lambda +2). % \qquad \text{ and } \qquad 	s^+_d(\lambda) =s^+_d(\lambda+2).
\end{align*}
\end{theorem}
\begin{proof}
We shall transform the sums by repeatedly using Theorem~\ref{theo:recurrence_relations_a_b}.
We start by using the recurrence equation \eqref{eq:b_first_relation} for the $b$-quantity with  $\a = \lambda + S_I+2$,
\begin{align*}
s_{d,\zeta}(\lambda) = &\sum_{I \subseteq \{1,\dots,d\}}
\zeta^{\#I} a\left( \lambda + S_I +2; I \right) \left(\lambda+ S_I + 1\right) b \left( \lambda + S_I;  I^c \right) \\
=&
\sum_{I \subseteq \{1,\dots,d\}}
\zeta^{\#I} a\left( \lambda +S_I +2;  I  \right) \left(\lambda+S_I + 1\right)\\
 & \qquad \qquad \times\biggl(\frac{\lambda + S_I +2}{\lambda + S_I +1} b(\lambda + S_I +2;I^c) + \sum_{\substack{j_1,j_2 \in I^c\\j_1 \neq j_2}} \frac{b(\lambda + S_I +\lambda_{j_1}+\lambda_{j_2} +2;I^c \bsl \{j_1,j_2\})}{(\lambda + S_I +1)(\lambda + S_I + \lambda_{j_1} +2)} \biggr)
 \\
		=& \sum_{I \subseteq \{1,\dots,d\}}
		\zeta^{\#I} a\left( \lambda +S_I +2;  I  \right) \left(\lambda+S_I + 2\right)  b(\lambda + S_I +2;I^c) \\
		& \qquad \qquad +\sum_{I \subseteq \{1,\dots,d\}} \zeta^{\#I} \sum_{\substack{j_1,j_2 \in I^c\\j_1 \neq j_2}} \frac{ a\left( \lambda +S_I +2;  I  \right)b(\lambda + S_I +\lambda_{j_1}+\lambda_{j_2}+2;I^c \bsl \{j_1,j_2\})}{\lambda + S_I + \lambda_{j_1} +2}.
	\end{align*}
Let us denote the first sum by $S_1$ and the second one by $S_2$. Further, we set $K = I \cup \{j_1,j_2\}$, which gives $K^c = I^c \bsl \{j_1,j_2\}$.
Recall that $\zeta \in \{+1,-1\}$  and hence $\zeta^{\#I} = \zeta^{\#K}$.
This lets us write $S_2$ in the following way,
	\begin{align*}
		S_2 = \sum_{K \subseteq \{1,\dots,d\}} \zeta^{\#K} b(\lambda + S_K +2;K^c) \sum_{\substack{j_1,j_2 \in K\\j_1 \neq j_2}} \frac{ a\left( \lambda +S_K- \lambda_{j_1}- \lambda_{j_2} +2;  K \bsl \{j_1;j_2\}  \right)}{\lambda + S_K - \lambda_{j_2} +2}.
	\end{align*}
Note that we omitted the condition $\# K \geq 2$ since for $\# K \in \lbrace 0,1 \rbrace$ the sum on the right-hand side becomes empty and is defined as $0$.
Next, we apply the recurrence relation given in~\eqref{eq:a_first_relation} (which also holds for $\# K=0,1$) with $\a = \lambda + S_K +2$ to the sum of the $a$-quantities on the right-hand side,
	\begin{align*}
		S_2 &= \sum_{K \subseteq \{1,\dots,d\}} \zeta^{\#K} b(\lambda + S_K +2;K^c) \bigl((\lambda + S_K + 3)a (\lambda + S_K +4;K) - (\lambda + S_K+2)a(\lambda + S_K +2;K) \bigr) \\
		&= \sum_{I \subseteq \{1,\dots,d\}} \zeta^{\#I} b(\lambda + S_I +2;I^c) \bigl((\lambda + S_I + 3)a (\lambda + S_I +4;I) - (\lambda + S_I+2)a(\lambda + S_I +2;I) \bigr),
	\end{align*}
where in the second line we changed the notation from $K$ to $I$.
	Thus, for $S_1+S_2$, we get,
	\begin{align*}
		S_1+S_2 =& \sum_{I \subseteq \{1,\dots,d\}} \zeta^{\#I} a\left( \lambda +S_I +2;  I  \right) \left(\lambda+S_I + 2\right)  b(\lambda + S_I +2;I^c) \\+& \sum_{I \subseteq \{1,\dots,d\}} \zeta^{\#I} b(\lambda + S_I +2;I^c) \bigl((\lambda + S_I + 3)a (\lambda + S_I +4;I) - (\lambda + S_I+2)a(\lambda + S_I +2;I) \bigr) \\
		=& \sum_{I \subseteq \{1,\dots,d\}} \zeta^{\#I} a\left( \lambda +S_I +4;  I  \right) \left(\lambda+S_I + 3\right)  b(\lambda + S_I +2;I^c),
	\end{align*}
which is a shifted version of the equation we started with. In particular, we have shown that $s_{d,\zeta}(\lambda) = s_{d,\zeta}(\lambda +2)$,  which completes the proof.
\end{proof}

\begin{lemma}[Asymptotics of $a$ and $b$]\label{lem:asymptotics_a_b}
Let $\lambda_1, \dots, \lambda_d \geq 0$ be fixed.  Then, as $T\to +\infty$,
	\begin{align}
		b(T; \lambda_1, \dots, \lambda_d) &\sim \sqrt{\frac{2\pi}{T}} \prod_{i=1}^{d} F_{\lambda_i}(0) \label{eq:asymptotics_b},\\
		a(T; \lambda_1, \dots, \lambda_d) &\sim \sqrt{\frac{2\pi}{T}} \prod_{i=1}^{d} F_{\lambda_i}(0)\label{eq:asymptotics_a}.
	\end{align}
\end{lemma}
\begin{proof}
We start with the first asymptotics. By \eqref{eq:def_b} and the Laplace asymptotics \cite[Theorem B.7, p. 758]{Flajolet_book}, we have
\begin{align*}
b(T; \lambda_1, \dots, \lambda_d) = \int_{-\pi /2}^{+\pi /2} \cos^{T} (x) \prod_{j=1}^{d} F_{\lambda_j}(x) \dd x
=\int_{-\pi /2}^{+\pi /2} \eee^{T \log \cos x}  \prod_{j=1}^{d} F_{\lambda_j}(x) \dd x
\sim \sqrt{\frac{2\pi}{T}}\prod_{j=1}^{d} F_{\lambda_j}(0).
\end{align*}
The second asymptotics can be done analogously. By \eqref{eq:def_a} and the Laplace asymptotics,
\begin{align*}
a(T; \lambda_1, \dots, \lambda_d) = \int_{-\infty}^{+\infty} \cosh^{-T} (x) \prod_{j=1}^{d} F_{\lambda_j}(\ii x) \dd x
= \int_{-\infty}^{+\infty} \eee^{-T \log \cosh x} \prod_{j=1}^{d} F_{\lambda_j}(\ii x) \dd x
\sim \sqrt{\frac{2\pi}{T}}\prod_{j=1}^{d} F_{\lambda_j}(0).
	\end{align*}
	This completes the proof.
\end{proof}

We are now able to compute the limit of $s_{d,-1}(T)$ and $s_{d,+1}(T)$  as $T\to+\infty$.
\begin{theorem}\label{theo:limit_s_d}
For $d\in \N_0$ and any fixed  $\lambda_1,\ldots, \lambda_d \geq 0$,
\begin{align}\label{eq:limit_s_d}
		\lim_{T \to +\infty} s_{d,-1}(T) =  \begin{cases}
			2 \pi, &\text{if } d=0,\\
			0, &\text{if } d\geq 1,
			\end{cases}
\qquad 		
\lim_{T \to +\infty} s_{d,+1}(T) =  2\pi \prod_{i=1}^d  c_{\frac{1}{2}\lambda_i - \frac{1}{2}}^{-1}.
\end{align}
\end{theorem}
%This lets us prove Theorem \ref{theo:limit_s_d}
For the proof, we make use of a short Lemma.
\begin{lemma}\label{lem:c_shifted_by_1/2}
Let $\alpha \geq -1/2$. Then, we have
	\begin{align*}
		c_{\a} \cdot c_{\a  - \frac{1}{2}} = (2\a + 1) \cdot \frac{1}{2\pi}.
	\end{align*}
\end{lemma}
\begin{proof}
	We calculate directly:
	\begin{align*}
		c_{\a} \cdot c_{\a  - \frac{1}{2}} = \frac{\Gamma(\a + 3/2)}{\Gamma(\a + 1)} \cdot \frac{\Gamma(\a + 1)}{\Gamma(\a + 1/2)} \cdot \frac{1}{\pi} =\frac{\Gamma(\a + 3/2)}{\Gamma(\a + 1/2)} \cdot \frac{1}{\pi}= (2\a + 1) \cdot \frac{1}{2\pi}.
	\end{align*}
\end{proof}
\begin{proof}[Proof of Theorem \ref{theo:limit_s_d}]
By definition of $s_{d,\zeta}$ given in~\eqref{eq:s_d_def_sum_ab},
	\begin{align*}
		s_{d,\zeta}(T) = \sum_{I \subseteq \{1,\dots,d\}}
		\zeta^{\#I} a\left(T + S_I+2; I \right) \left(T + S_I +1\right) b \left(T + S_I;  I^c \right).
	\end{align*}
For $d=0$, we have
$$
s_{0, \zeta}(T) = a(T +2; \varnothing) (T +1) b (T; \varnothing) = (T+1) \left(c_{T/2} \cdot c_{(T-1)/2}\right)^{-1} = 2\pi
$$
by Lemma~\ref{lem:c_shifted_by_1/2}. In the following, let $d\in \N$.
For every subset $I \subseteq \{1,\dots,d\}$, the asymptotics established in Lemma \ref{lem:asymptotics_a_b} gives us, as $T\to +\infty$,
\begin{align*}
a\left(T + S_I+2; I \right) \left(T + S_I +1\right) b \left(T + S_I;  I^c \right)
&\sim
\sqrt{\frac{2\pi}{T+2}} \Biggl(\prod_{i \in I} F_{\lambda_i} (0)\Biggr) \cdot \left(T+ S_I + 1\right) \cdot \sqrt{\frac{2\pi}{T}} \prod_{i \in I^c} F_{\lambda_i}
(0)
\\
&\to 2\pi \prod_{i=1}^{d} F_{\lambda_i} (0).
\end{align*}
Note that the limit does not depend on $I\subseteq \{1,\ldots, d\}$. It follows that
\begin{align*}
\lim_{T \to +\infty} s_{d,\zeta}(T) =& \sum_{I \subseteq \{1,\dots,d\}} \zeta^{\#I} \lim_{T \to +\infty}  a\left(T + S_I + 2; I \right) \left(T+ S_I + 1\right) b \left(T + S_I;  I^c \right)
\\
	=& 2\pi \Biggl( \prod_{i=1}^{d} F_{\lambda_i} (0)\Biggr) \cdot \sum_{i=0}^{d} \zeta^i \binom{d}{i},
\end{align*}
where we used that there are $\binom di$ subsets $I\subseteq\{1,\ldots, d\}$ with $\#I = i$. Now, if $\zeta = -1$, then
$$
\lim_{T \to +\infty} s_{d,-1}(T) = 2\pi \Biggl(\prod_{i=1}^{d} F_{\lambda_i} (0)\Biggr) \cdot \sum_{i=0}^{d} (-1)^i \binom{d}{i} =  2\pi \Biggl(\prod_{i=1}^{d} F_{\lambda_i} (0)\Biggr) \cdot (1-1)^d = 0,
$$
where we used that $d\neq 0$.  This gives the first equation in~\eqref{eq:limit_s_d}. Next, if $\zeta = +1$, then
\begin{align*}
\lim_{T \to +\infty} s_{d,+1}(T)
= 2\pi \Biggl(\prod_{i=1}^{d} F_{\lambda_i} (0)\Biggr) \cdot \sum_{i=0}^{d} \binom{d}{i}
= 2\pi \Biggl(\prod_{i=1}^{d} F_{\lambda_i} (0)\Biggr) \cdot 2^d
=  2\pi \prod_{i=1}^d  c_{\frac{1}{2}\lambda_i - \frac{1}{2}}^{-1}
	\end{align*}
since, for any $i\in \{1,\ldots, d\}$, we have
\begin{align*}
F_{\lambda_i} (0) = \int_{-\pi/2}^{0} \cos^{\lambda_i} (x)\dd x = \frac 12  \int_{-\pi/2}^{+\pi/2} \cos^{\lambda_i}(x) \dd x
= \frac 12  \cdot c_{\frac12 \lambda_i - \frac 12}^{-1}
\end{align*}
by~\eqref{eq:int_cos_cosh} or~\eqref{eq:F_a_imaginary_argument}. This completes the proof of the second equation in~\eqref{eq:limit_s_d}.
\end{proof}
We are now ready to state the main result of this section, a set of nonlinear relations between the $a$-quantities and the $b$-quantities.
\begin{theorem}[Relations between $a$ and $b$]\label{theo:relations_a_b}
	For all $d \in \N_0$, $\lambda>-1$ and $\lambda_1,\ldots, \lambda_d\geq 0$, we have
\begin{align*}
s_{d,-1}(\lambda)
&=
\sum_{I \subseteq \{1,\dots,d\}}
		(-1)^{\#I}
a\left( \lambda +\sum_{i \in I} \lambda_i +2; \{ \lambda_i: i \in I \} \right) \left(\lambda+\sum_{i \in I} \lambda_i + 1\right)
b \left( \lambda +\sum_{i \in I} \lambda_i; \{ \lambda_i : i \in I^c \} \right) \\
&=
\begin{cases}
2\pi, &\text{if } d=0,\\
0, &\text{if } d\geq 1.
\end{cases}
\end{align*}
Moreover, under the same assumptions,
\begin{align*}
s_{d,+1}(\lambda)
&=
\sum_{I \subseteq \{1,\dots,d\}}
a\left( \lambda +\sum_{i \in I} \lambda_i +2; \{ \lambda_i: i \in I \} \right)
\left(\lambda+\sum_{i \in I} \lambda_i + 1\right) b \left( \lambda +\sum_{i \in I} \lambda_i; \{ \lambda_i : i \in I^c \} \right) \\
&=
2\pi \prod_{i=1}^d  c_{\frac{1}{2}\lambda_i - \frac{1}{2}}^{-1}.
\end{align*}
\end{theorem}
\begin{proof}
Let $\zeta\in \{-1,+1\}$. We start with a repeated use of Theorem \ref{theo:s_d_is_periodic}, which lets us write
$$
s_{d,\zeta}(\lambda) = s_{d,\zeta}(\lambda + 2) = s_{d,\zeta}(\lambda + 4) = \ldots = \lim_{n\to\infty}s_{d,\zeta}(\lambda + 2n).
$$
The limit on the right-hand side has been identified in Theorem \ref{theo:limit_s_d}.
\end{proof}

\begin{proposition}\label{prop:relation_a_b_even_odd}
Let $d\in \N$ (note that $d=0$ is excluded) and $\lambda_1,\ldots, \lambda_d \geq 0$, $\lambda>-1$.
Then,
		\begin{align*}
			& \sum_{\substack{I \subseteq \{1,\dots,d\} \\ \#I \text{ odd}}} a\left(\lambda + \sum_{i \in I} \lambda_i +2;\{\lambda_i: i \in I\}\right)  \left(\lambda + \sum_{i \in I} \lambda_i +1\right) b\left(\lambda + \sum_{i \in I} \lambda_i;\{\lambda_i: i \in I^c\}\right) \\
			&\qquad = \sum_{\substack{I \subseteq \{1,\dots,d\} \\ \#I \text{ even}}} a\left(\lambda + \sum_{i \in I} \lambda_i +2;\{\lambda_i: i \in I\}\right)  \left(\lambda + \sum_{i \in I} \lambda_i +1\right) b\left(\lambda + \sum_{i \in I} \lambda_i;\{\lambda_i: i \in I^c\}\right) \\
			&\qquad \qquad \qquad= \pi \cdot  \left( \prod_{i=1}^{d} c_{\frac12 \lambda_i- \frac 12}^{-1}\right).
		\end{align*}
	\end{proposition}
	\begin{proof}
Observe that $\sum_{\#I \text{ even}}  = \frac 12 (s_{d,+1}(\lambda) + s_{d,-1}(\lambda))$ and $\sum_{\#I \text{ odd}}  = \frac 12 (s_{d,+1}(\lambda) - s_{d,-1}(\lambda))$. It remains to apply Theorem \ref{theo:relations_a_b}. In particular, $s_{d,-1}(\lambda) = 0$ since $d\neq 0$.
\end{proof}
\begin{example}\label{ex:a_and_b_for_d_equals_1}
	If we choose $d=1$ in the above proposition, we get a system of linear equations in the unknowns $a(\l + \l_1 +2;\l_1)$ and $b(\l ;\l_1)$ that can be uniquely solved for $x,y > -1$,
	\begin{align*}
		a(x+2;y) = \frac{\pi}{x+1} \frac{c_{\frac{x-1}{2}}}{c_{\frac{y-1}{2}}} = \left(2c_{\frac{x}{2}}c_{\frac{y-1}{2}}\right)^{-1} \qquad \text{and} \qquad b(x;  y ) = \frac{\pi}{x+1} \frac{c_{\frac{x}{2}}}{c_{\frac{y-1}{2}}} = \left(2c_{\frac{x-1}{2}}c_{\frac{y-1}{2}}\right)^{-1},
	\end{align*}
	where we have also used Lemma~\ref{lem:c_shifted_by_1/2}. These formulas have been obtained in Examples~\ref{ex:a_def_case_d_0}, \ref{ex:b:def_d_0}.
\end{example}

\subsection{Angles of beta cones}
In this section we express internal and external angles of beta cones, as well as many related quantities, through the $a$- and $b$-quantities which, as we recall, are defined by
$$
a(\a;\a_1,\dots, \a_d) = \int_{-\infty}^{+\infty} \cosh^{-\a} (x) \prod_{j=1}^{d} F_{\a_j}(\ii x) \dd x,
\qquad
b(\a;\a_1,\dots, \a_d) = \int_{-\pi /2}^{+\pi /2} \cos^\a (x) \prod_{j=1}^{d} F_{\a_j}(x) \dd x.
$$
Let us begin with the $A$- and $B$-quantities, which were introduced and studied in subsection~\ref{subsec:A-and_B-quantities}.
\begin{theorem}\label{lem:A_through_a_B_through_b}
Let $d\geq 2$ and $\gamma_1,\ldots, \gamma_d \geq \frac d2 - 1$. Assuming that $\gamma\geq \frac d2-1$ in the $A$-case and $\gamma-(\gamma_1 + \ldots + \gamma_d) \geq  \frac d2 - 1$ in the $B$-case, we have
	\begin{align}
		A(\g;\g_1,\dots,\g_d) &= c_{\g+\sum_{i=1}^{d}\g_i} \cdot \prod_{j=1}^{d} c_{\g_j - \frac{1}{2}} \cdot a\left(2\g + 2\sum_{i=1}^{d}\g_i +2;2\g_1,\dots,2\g_d\right), \label{eq:A_through_a}\\
		B(\g;\g_1,\dots,\g_d) &= c_{\g-\sum_{i=1}^{d}\g_i -\frac{1}{2}} \cdot \prod_{j=1}^{d} c_{\g_j - \frac{1}{2}} \cdot b\left(2\g - 2\sum_{i=1}^{d}\g_i;2\g_1,\dots,2\g_d\right). \label{eq:B_through_b}
	\end{align}
For $d=0,1$, the formulas hold with the conventions on the extended ranges of the $a$ and $b$ quantities made in  Examples~\ref{ex:a_def_case_d_0} and~\ref{ex:b:def_d_0}.
\end{theorem}
\begin{proof}
We start with showing \eqref{eq:B_through_b}. For $d=0$, it holds true since the left-hand side equals $1$ by Definition~\ref{def:int_and_ext}, while the right-hand side simplifies to $1$ by using the definition of $b(\g;\varnothing)$ given in Example~\ref{ex:b:def_d_0}. If $d=1$, the left-hand side simplifies to $1/2$, while we can use Example~\ref{ex:a_and_b_for_d_equals_1} to simplify the right-hand side,
\begin{align*}
	c_{\g-\g_1-\frac{1}{2}} c_{\g_1-\frac{1}{2}} (2c_{\g-\g_1-\frac{1}{2}}c_{\g_1-\frac{1}{2}})^{-1} = \frac{1}{2}.
\end{align*}
Next, let $d \geq 2$. We use Equation~\eqref{eq:B_through_Ext} and Corollary~\ref{cor:external_beta_formula} to get
	\begin{align*}
		B(\g;\g_1,\dots,\g_d) &= \extern\left( \g - \sum_{i=1}^{d}\g_i - \frac{d}{2} ; \g_1- \frac{d}{2}, \dots, \g_d - \frac{d}{2}\right)\\
		&= \int_{-1}^{+1} c_{\g-\sum_{i=1}^{d}\g_i-1/2}(1-t^2)^{\g-\sum_{i=1}^{d}\g_i-1/2} \prod_{j=1}^{d} \left(\int_{-1}^{t} c_{\g_j- \frac{1}{2}} (1-s^2)^{\g_j- \frac{1}{2}} \dd s   \right) \dd t\\
		&= \int_{-\pi /2}^{+\pi /2} c_{\g-\sum_{i=1}^{d}\g_i -\frac{1}{2}}\cos(x)^{2\g - 2\sum_{i=1}^{d}\g_i } \prod_{j=1}^{d} c_{\g_j - \frac{1}{2}} F_{2\g_j}(x) \dd x \\
		&= c_{\g-\sum_{i=1}^{d}\g_i -\frac{1}{2}} \cdot \prod_{j=1}^{d} c_{\g_j - \frac{1}{2}} \cdot b\left(2\g - 2\sum_{i=1}^{d}\g_i;2\g_1,\dots,2\g_d\right),
	\end{align*}
where we have substituted $t=\sin(x)$ and $s=\sin(y)$ and used \eqref{eq:def_F_as_cos_integral}. In total, this shows~\eqref{eq:B_through_b} for $d \in \N_0$. To prove \eqref{eq:A_through_a}, we once more look at the system of linear equations with unknowns $f_I$ that was studied in Theorem~\ref{theo:uniqueness_of_solution},
	\begin{align}\label{eq:system_of_linear_equations}
		\sum_{I \subseteq \{1,\dots,d\}} (-1)^{\#I} f_I B \left(\g + \sum_{k=1}^d \g_k ; \;\; \left\{\g_{i} : i \in I^c\right\}\right) =
		\begin{cases}
			1, &\text{if } d=0,\\
			0, &\text{if } d\geq 1.
		\end{cases}
	\end{align}
	We know from Theorem \ref{theo:uniqueness_of_solution} that this system has a unique solution. Setting $f_I=A(\gamma;\{\g_i: i \in I\})$ one recovers the second relation between the $A-$ and $B$-quantities given in Theorem~\ref{theo:beta_cones_relations_ext_int_second_version}, which means that $A(\gamma;\{\g_i: i \in I\})$ solves this system. Next, we claim that
	\begin{align}\label{eq:another_unique_solution}
		f_I(\gamma) \coloneqq c_{\g+\sum_{i \in I}\g_i} \cdot \Biggl(\prod_{i \in I} c_{\g_i - \frac{1}{2}}\Biggr) \cdot a\left(2\g + 2\sum_{i \in I}\g_i +2;\{2\g_i: i \in I\}\right)
	\end{align}
	also solve it, which will directly prove Equation~\ref{eq:A_through_a} by uniqueness of the solution. For this, we plug \eqref{eq:another_unique_solution} and Equation \eqref{eq:B_through_b} into \eqref{eq:system_of_linear_equations} (we remind the reader that if $\# I \in \lbrace 0,1 \rbrace$, we use Examples~\ref{ex:a_def_case_d_0} and~\ref{ex:a_and_b_for_d_equals_1} for the definition of $a\left(2\g + 2\sum_{i \in I}\g_i +2;\{2\g_i: i \in I\}\right)$). This then gives us
\begin{multline}\label{eq:some_equation}
	\sum_{I \subseteq \{1,\dots,d\}} (-1)^{\#I} c_{\g+\sum_{i \in I}\g_i} \cdot \Biggl(\prod_{i \in I} c_{\g_i - \frac{1}{2}}\Biggr) \cdot a\left(2\g + 2\sum_{i \in I}\g_i +2;\{2\g_i: i \in I\}\right) \\
	\times c_{\g+\sum_{i \in I}\g_i -\frac{1}{2}} \cdot \Biggl(\prod_{i \in I^c} c_{\g_i - \frac{1}{2}} \Biggr)\cdot b\left(2\g + 2\sum_{i \in I}\g_i;\{2\g_i: i \in I^c\}\right).
\end{multline}
Next, we use Lemma~\ref{lem:c_shifted_by_1/2} to get
	\begin{align*}
		c_{\g+\sum_{i \in I}\g_i} \cdot c_{\g+\sum_{i \in I}\g_i  - 1/2} = (2\g+2\sum_{i \in I}\g_i+1) \cdot \frac{1}{2\pi}.
	\end{align*}
	This lets us write \eqref{eq:some_equation} as
	\begin{multline*}
		\frac{1}{2\pi}\Biggl(\prod_{j=1}^{d} c_{\g_j - \frac{1}{2}}\Biggr) \sum_{I \subseteq \{1,\dots,d\}} \Biggl( (-1)^{\#I} a\left(2\g+2\sum_{i \in I}\g_i +2;\{2\lambda_i: i \in I\}\right) \\ \times \left(2\g+2\sum_{i \in I}\g_i +1\right) b\left(2\g + 2\sum_{i \in I}\g_i;\{2\lambda_i: i \in I^c\}\right)\Biggr)
		=\begin{cases}
			1, &\text{if } d=0,\\
			0, &\text{if } d\geq 1,
		\end{cases}
	\end{multline*}
	where the last equation holds for the following reason. In the case $d\geq 1$, it immediately follows with the help of Proposition~\ref{prop:relation_a_b_even_odd}. To verify it in the case $d=0$, we use the definitions of the $a-$ and $b$-quantities given in Examples~\ref{ex:a_def_case_d_0} and~\ref{ex:b:def_d_0} as well as Lemma~\ref{lem:c_shifted_by_1/2} to see that the above equation simplifies to
	\begin{align*}
		\frac{1}{2\pi}\cdot  a(2\g +2;\varnothing) \cdot (2\g +1) \cdot b(2\g;\varnothing) =\frac{1}{2\pi}\cdot  (2\g +1) \cdot \left(c_{\g} \cdot c_{\g- 1/2}\right)^{-1} = 1.
	\end{align*}
	In total, we showed that $f_I(\gamma)$ indeed solves \eqref{eq:system_of_linear_equations} and therefore coincides with $A(\gamma;\{\g_i: i \in I\})$, as the system has a unique solution. This completes the proof.
\end{proof}
This relation between the $A/B$ and $a/b$-quantities lets us also express the internal and external quantities through the $a$- and $b$-quantities.
\begin{proposition}\label{prop:int_ext_through_a_b}
	For $d \geq 2$, $\b\geq -1$  and $\b_1, \dots, \b_d \geq - 1$, we can express the internal and external quantities in the following way,
\begin{align*}
	\intern (\b; \b_1,\ldots, \b_d) =& c_{\b + \frac{d}{2} + \sum_{i=1}^d \left(\b_i + \frac{d}{2}\right)} \left(\prod_{i=1}^d c_{\b_i + \frac{d-1}{2}}\right) a\left(2\b+d+\sum_{i=1}^d \left(2\b_i + d\right)+2 ; 2\b_1 +d,\dots, 2\b_d+d \right),\\
	\extern (\b; \b_1,\ldots, \b_d) =& c_{\b + \frac{d-1}{2} } \left(\prod_{i=1}^d c_{\b_i + \frac{d-1}{2}}\right) b\left(2\b+d ; 2\b_1 +d,\dots,2\b_d+d\right).
\end{align*}
For $d=0,1$, the formulas hold with the conventions on the extended range of $a$ and $b$ made in Examples~\ref{ex:a_def_case_d_0}, and~\ref{ex:b:def_d_0}.
\end{proposition}
\begin{proof}
	We only show the first equation, since the second one can be shown the same way. Using equations \eqref{eq:A_through_Int_converse} and \eqref{eq:A_through_a}, we get
	\begin{align*}
		\intern (\b; \b_1,\ldots, \b_d) =& A\left(\b+\frac{d}{2};\b_1+\frac{d}{2},\dots,\b_d+\frac{d}{2}\right)\\
		=& c_{\b + \frac{d}{2} + \sum_{i=1}^d \left(\b_i + \frac{d}{2}\right)} \Biggl(\prod_{i=1}^d c_{\b_i + \frac{d-1}{2}}\Biggr) a\left(2\b+d+\sum_{i=1}^d \left(2\b_i +d\right)+2 ; 2\b_1 +d,\dots, 2\b_d+d \right).
	\end{align*}
\end{proof}
Since the $a$- and $b$-quantities have integral representations, this also gives us integral representations for the $A$- and $B$-quantities.
\begin{theorem}\label{theo:A_integral_representation}
Let $d\geq 2$ and $\gamma_1,\ldots, \gamma_d \geq \frac d2 - 1$. Assuming that $\gamma\geq \frac d2-1$ in the $A$-case and $\gamma-(\gamma_1 + \ldots + \gamma_d) \geq  \frac d2 - 1$ in the $B$-case, we have the following integral representations,
	\begin{align*}
		A(\g;\g_1,\dots,\g_d) = \int_{-1}^{+1} c_{\g+\sum_{i=1}^{d}\g_i}(1-t^2)^{\g+\sum_{i=1}^{d}\g_i} \prod_{j=1}^{d} \left(\frac{1}{2} + \ii \int_{0}^{t} c_{\g_j -\frac{1}{2}} (1-s^2)^{-\g_j -1} \dd s   \right) \dd t, \\
		B(\g;\g_1,\dots,\g_d) =  \int_{-1}^{+1} c_{\g-\sum_{i=1}^{d}\g_i-\frac{1}{2}}(1-t^2)^{\g-\sum_{i=1}^{d}\g_i-\frac{1}{2}} \prod_{j=1}^{d} \left(\int_{-1}^{t} c_{\g_j-\frac{1}{2}} (1-s^2)^{\g_j-\frac{1}{2}} \dd s   \right) \dd t.
	\end{align*}
Also, for all $d\geq 2$ and $\beta, \beta_1,\ldots, \beta_d \geq -1$,
\begin{align*}
				\intern (\b;\b_1,\dots,\b_d) &= \int_{-1}^{+1} c_{\b+\frac{d}{2}+\sum_{i=1}^{d}\left(\b_i+\frac{d}{2}\right)}(1-t^2)^{\b+\frac{d}{2}+\sum_{i=1}^{d}\left(\b_i+\frac{d}{2}\right)}
\\ &\qquad \qquad \qquad
				\times \prod_{j=1}^{d} \left(\frac{1}{2} + \ii \int_{0}^{t} c_{\b_j +\frac{d-1}{2}} (1-s^2)^{-\b_j-\frac{d}{2} -1} \dd s   \right) \dd t, \\
				\extern (\b;\b_1,\dots,\b_d) &=
				\int_{-1}^{+1}  c_{\b + \frac {d-1}{2}} (1-t^2)^{\b +  \frac {d-1}{2}}
				\prod_{j =1}^{d} \left(\int_{-1}^t c_{\b_j + \frac{d-1}{2}} (1-s^2)^{\b_j + \frac{d-1}{2}}\,\dd s\right) \,\dd t.
		\end{align*}
\end{theorem}
\begin{proof}
	For $B(\g;\g_1,\dots,\g_d)$, we use Corollary \ref{cor:external_beta_formula} and the expression of $B(\g ;\g_1,\dots,\g_d)$ through the external quantity, Equation \eqref{eq:B_through_Ext},
	\begin{align*}
		B(\g;\g_1,\dots,\g_d) %&= B[\b +\frac{1}{2};\b_1+\frac{1}{2},\dots,\b_d+\frac{1}{2}] \\
		&= \extern\left( \g - \sum_{i=1}^{d}\g_i - \frac{d}{2} ; \g_1- \frac{d}{2}, \dots, \g_d - \frac{d}{2}\right) \\
		&= \int_{-1}^{+1} c_{\g-\sum_{i=1}^{d}\g_i-\frac{1}{2}}(1-t^2)^{\g-\sum_{i=1}^{d}\g_i-\frac{1}{2}} \prod_{j=1}^{d} \left(\int_{-1}^{t} c_{\g_j-\frac{1}{2}} (1-s^2)^{\g_j-\frac{1}{2}} \dd s   \right) \dd t.
	\end{align*}
	Setting $\b \coloneqq \g - \sum_{i=1}^{d}\g_i - \frac{d}{2}$ and $\b_i \coloneqq \g_i - \frac{d}{2}$ gives the integral representation for the external quantity, where the ranges given in the statement of this theorem follow from the ranges of $\g,\g_1,\dots,\g_d$. For $A(\g;\g_1,\dots,\g_d)$, we use Equation \eqref{eq:A_through_a} and Proposition~\ref{prop:other_integral_representation_of_a} to see
	\begin{align*}
		A(\g;\g_1,\dots,\g_d) &= c_{\g+\sum_{i=1}^{d}\g_i} \cdot \prod_{j=1}^{d} c_{\g_j - \frac{1}{2}} \cdot a\left(2\g + 2\sum_{i=1}^{d}\g_i +2;2\g_1,\dots,2\g_d\right) \\
				&= \int_{-1}^{+1} c_{\g+\sum_{i=1}^{d}\g_i}(1-t^2)^{\g+\sum_{i=1}^{d}\g_i} \prod_{j=1}^{d} \left(\frac{1}{2} + \ii \int_{0}^{t} c_{\g_j -\frac{1}{2}} (1-s^2)^{-\g_j -1} \dd s   \right) \dd t.
	\end{align*}
Setting $\b \coloneqq \g -\frac{d}{2}$ and $\b_i \coloneqq \g_i - \frac{d}{2}$ and using \eqref{eq:A_through_Int} gives the integral representation for the internal quantities with the correct ranges.
This completes the proof.
\end{proof}

\section{Results on beta cones}\label{sec:beta_cones_results}
\subsection{The \texorpdfstring{$\Theta$}{Theta}-function}\label{sec:Theta_quantity}
Now we are in the position to express all possible quantities related to beta cones through the $A$- and $B$-quantities, as well as the $a$- and $b$-quantities. Since all of those are sums involving products of $A$- and $B$-quantities of a certain structure, we introduce the following function.
\begin{definition}[$\Theta$-function]\label{def:theta}
Let $Y=\{y_1,\ldots, y_{\ell}\}$ and $Z= \{z_1,\dots, z_k\}$ be finite multisets of nonnegative numbers and let $x> -1/2$. We define
	\begin{multline}
		\Theta (x;Y;Z) \coloneqq \frac{1}{2\pi} \prod_{\omega \in Y \sqcup Z}c_{\omega -\frac{1}{2}}
\\
\times a\left(2x+2\sum_{\omega \in Y} \omega + 2;2Y\right)\cdot\left(2x+2\sum_{\omega \in Y} \omega+1\right)\cdot b\left(2x+2\sum_{\omega \in Y} \omega ;2Z\right).
\end{multline}
\end{definition}
\begin{remark}[Multiset notation]
A \emph{multiset} is a set which is allowed to have repeated elements. The multiplicity of an element $\omega$ in a multiset $\Lambda$ is denoted by $m_\Lambda(\omega)\in \N_0$.
For two multisets $Y$ and $Z$ we denote by $Y \sqcup Z$ their disjoint union (or sum), that is a multiset which contains all elements from $Y$ and all elements from $Z$, adding their multiplicities: $m_{Y\sqcup Z}(\omega) = m_Y(\omega) + m_Z(\omega)$.  The number of elements of a multiset $\Lambda= \{\lambda_1,\ldots, \lambda_m\}$, counting multiplicities, is denoted by $\#\Lambda = m$. We use the product notation $\prod_{\omega\in \Lambda}\varphi(\omega):= \varphi(\lambda_1) \ldots \varphi(\lambda_m)$. We write $a(\lambda; \Lambda)$ for $a(\lambda; \lambda_1,\ldots, \lambda_m)$ and similarly with $b(\lambda; \Lambda)$.
\end{remark}

When one of the multisets is empty, the $\Theta$-function simplifies as follows.
\begin{lemma}\label{lem:Theta_with_one_set_empty}
Let $x>-1/2$ and $Y=\{y_1,\ldots, y_{\ell}\}$ and $Z= \{z_1,\dots, z_k\}$ be finite multisets of nonnegative numbers. Then,
\begin{align*}
\Theta (x;Y;\varnothing)
&=
c_{x+\sum_{\omega \in Y} \omega} \cdot  \prod_{\omega \in Y} c_{\omega - \frac{1}{2}}  \cdot a\biggl(2x+2\sum_{\omega \in Y} \omega+2;2Y\biggr)
\\
\text{as well as}\qquad
\Theta (x;\varnothing;Z)
&=
c_{x-1/2} \cdot \prod_{\omega \in Z} c_{\omega - \frac{1}{2}}   \cdot b(2x;2Z).
\end{align*}
	In particular, we have $\Theta(x;\varnothing;\varnothing)=1$ (defining an empty product to be $1$).
\end{lemma}
\begin{proof}
The proof is a direct computation using the formulas $a(\alpha;\varnothing) = c_{(\a/2)-1}^{-1}$ and $b(\alpha;\varnothing)= c_{(\a-1)/2}^{-1}$; see Examples~\ref{ex:a_def_case_d_0}, \ref{ex:b:def_d_0}.  To simplify the constants, use the formula  		$(2\a + 1)/(2\pi) =  c_{\a} \cdot c_{\a  - (1/2)}$; see Lemma~\ref{lem:c_shifted_by_1/2}.
\end{proof}

\begin{theorem}[Formulas for the $\Theta$-function]\label{theo:theta_as_prod_of_thetas}
	For $x>-1/2$ and $Y$, $Z$ finite multisets of nonnegative numbers, we have
	\begin{align}
		\Theta (x ;Y;Z) = \Theta(x;Y;\varnothing) \cdot \Theta \left(\alpha;\varnothing;Z\right), \label{eq:theo:theta_as_prod_of_thetas}
	\end{align}
where we defined $\alpha:= x+\sum_{\omega \in Y} \omega$. Moreover, for the factors on the right-hand side we have
\begin{align*}
\Theta(x;Y;\varnothing)
&=
\int_{-1}^{+1} c_{\a} (1-t^2)^{\alpha} \prod_{\omega \in Y} \left(\frac{1}{2} + \ii \int_{0}^{t} c_{\omega -\frac{1}{2}} (1-s^2)^{-\omega -1}\, \dd s   \right) \dd t
\\
&=
 \int_{-\pi/2}^{+\pi/2} c_{\a} \cos^{2\alpha+1}(v) \prod_{\omega \in Y} \left(\frac{1}{2} + \ii \int_{0}^{v} c_{\omega -\frac{1}{2}} \cos^{-2\omega -1}(u)\, \dd u   \right) \dd v
\\
&=
\int_{-\infty}^{+\infty} c_{\a} \cosh^{-2\alpha-2}(z) \prod_{\omega \in Y} \left(\frac{1}{2} + \ii \int_{0}^{z} c_{\omega -\frac{1}{2}} \cosh^{2\omega}(y)\, \dd y   \right) \dd z
\end{align*}
as well as
\begin{align*}
\Theta \left(\alpha;\varnothing;Z\right)
&=
\int_{-1}^{+1} c_{\a-\frac{1}{2}} (1-t^2)^{\alpha -\frac{1}{2}}
		\prod_{\omega \in Z} \left(\frac 12 + \int_{0}^t c_{\omega - \frac{1}{2}} (1-s^2)^{\omega - \frac{1}{2}}\,\dd s\right) \,\dd t
\\
&=
\int_{-\pi/2}^{+\pi/2} c_{\a-\frac{1}{2}} \cos^{2\alpha}(v)
		\prod_{\omega \in Z} \left(\frac 12 + \int_{0}^v c_{\omega - \frac{1}{2}} \cos^{2\omega}(u)\,\dd u\right) \,\dd v
\\
&=
 \int_{-\infty}^{+\infty} c_{\a-\frac{1}{2}} \cosh^{-2\alpha -1}(z)
		\prod_{\omega \in Z} \left(\frac 12 + \int_{0}^z c_{\omega - \frac{1}{2}} \cosh^{-2\omega - 1} (y) \,\dd y\right) \,\dd z.
\end{align*}
\end{theorem}
\begin{proof}
Equation~\eqref{eq:theo:theta_as_prod_of_thetas} follows from Definition~\ref{def:theta} and Lemma~\ref{lem:Theta_with_one_set_empty}, taking also into account the identity $(2\a + 1)/(2\pi) =  c_{\a} \cdot c_{\a  - (1/2)}$; see Lemma~\ref{lem:c_shifted_by_1/2}. The formulas for $\Theta(x;Y;\varnothing)$ and $\Theta \left(\alpha;\varnothing;Z\right)$ follow from Lemma~\ref{lem:Theta_with_one_set_empty} combined with the formulas for $a$ and $b$; see Definitions~\ref{dfn:a_quantities} and~\ref{dfn:b_quantities} as well as Propositions~\ref{prop:other_integral_representation_of_a} and~\ref{prop:other_integral_representation_of_b}.
\end{proof}

Further, the structure of the expressions given in Lemma~\ref{lem:Theta_with_one_set_empty} resembles the expression of the internal and external quantities through the $a$- and $b$-quantities. This leads to the following.
\begin{proposition}\label{prop:int_ext_through_theta}
Let $d\in \N_0$ and  $\beta, \beta_1,\ldots, \beta_d \geq -1$. Then,
\begin{align*}
\intern (\beta; \beta_1,\ldots, \beta_d)  =& \Theta \left(\beta + \frac{d}{2} ;\left\{\beta_1 + \frac{d}{2}, \dots,\beta_d + \frac{d}{2}\right\}; \varnothing\right),\\
\extern (\beta; \beta_1,\ldots, \beta_d) =&  \Theta \left(\beta + \frac{d}{2}; \varnothing ;\left\{\beta_1 + \frac{d}{2}, \dots,\beta_d + \frac{d}{2}\right\}\right).
\end{align*}
\end{proposition}
\begin{proof}
Both equations follow directly from Proposition~\ref{prop:int_ext_through_a_b} and Lemma~\ref{lem:Theta_with_one_set_empty}.
\end{proof}

\begin{proposition}[Nonlinear relations for the $\Theta$-function]\label{prop:relation_theta_even_odd}
For all $n\in \N_0$, $\g>-1/2$, $\g_1,\dots,\g_n\geq 0$,
\begin{align*}
\sum_{I \subseteq \{1,\dots,n\}} \Theta ( \gamma ;\{\g_i:i \in I\};\{\g_i: i \in I^c\})  = 1,
\end{align*}
where $I^c:= \{1,\ldots, n\} \bsl I$ is the complement of $I$. Assuming additionally that  $n\neq 0$,
\begin{align*}
&\sum_{\substack{I \subseteq \{1,\dots,n\} \\ \#I \text{ odd}}} \Theta ( \gamma ;\{\g_i:i \in I\};\{\g_i: i \in I^c\}  )  =\sum_{\substack{I \subseteq \{1,\dots,n\} \\ \#I \text{ even}}} \Theta ( \gamma ;\{\g_i:i \in I\};\{\g_i: i \in I^c\}  ) = \frac{1}{2}.
\end{align*}
	\end{proposition}
	\begin{proof}
		The proof follows directly from Definition~\ref{def:theta} and Proposition~\ref{prop:relation_a_b_even_odd}.
	\end{proof}

\begin{proposition}[Limit as $x\to+\infty$]\label{prop:Theta_function_gamma_infinity}
For arbitrary multisets $Y$ and $Z$ of nonnegative numbers,
$$
\lim_{x \to +\infty} \Theta(x;Y;Z) = 2^{-(\# Y + \#Z)}.
$$
\end{proposition}
\begin{proof}
Indeed, it follows from Propositions~\ref{prop:int_ext_through_theta} and~\ref{prop:intern_extern_beta_infinity} that $\Theta(x;Y;\varnothing) \to 2^{-\#Y}$ and  $\Theta(x;\varnothing;Z) \to 2^{-\#Z}$ as $x \to +\infty$. Actually, this argument imposes an unnecessary lower bound on the numbers in $Y$ and $Z$. To remove it, one can combine Lemma~\ref{lem:Theta_with_one_set_empty} with the asymptotics of $a$ and $b$ given in Lemma~\ref{lem:asymptotics_a_b}. It gives the same result after some computations.
\end{proof}
\begin{example}
Letting $\gamma\to +\infty$ in Proposition~\ref{prop:relation_theta_even_odd} gives the identities $\sum_{k=0}^n \binom nk = 2^n$ and, for $n\neq 0$, $\sum_{k \text{ even}} \binom nk = \sum_{k \text{ odd}} \binom nk = 2^{n-1}$.
\end{example}

\subsection{Expected angles and  conic intrinsic volumes}
We now begin to calculate expectations of various functionals associated with beta cones. All results will be stated in terms of the $\Theta$-function introduced in Section~\ref{sec:Theta_quantity}. In the rest of Section~\ref{sec:beta_cones_results}, we work in the following setting. Let
$$
\sC=\sC_{n,d}^{\beta; \beta_1,\ldots, \beta_n} = \pos(Z_1-Z,\ldots, Z_n-Z)\sim\BetaCone(\R^d;\b;\b_1,\dots,\b_n)
$$
be a beta cone in $\R^d$ with $d \in \N$, $n \in \N$ and parameters $\beta, \beta_1,\ldots, \beta_n \geq -1$. This means that $Z\sim f_{d,\beta}, Z_1\sim f_{d,\beta_1},\ldots, Z_n\sim f_{d,\beta_n}$ are independent beta-distributed points in $\BB^d$.
To simplify the resulting formulas, we introduce the variables
$$
\gamma := \beta + \frac{d}{2} \quad \text{ and } \quad \g_i := \b_i + \frac d 2,  \quad \text{ for all } i \in \lbrace 1, \dots,n \rbrace.
$$
To exclude degenerate cases, for  $d=1$ we suppose additionally that $\beta, \beta_1,\ldots, \beta_n \neq -1$. Under these assumptions,  all conclusions of  Lemma~\ref{lem:beta_cones_generic_properties} hold.

In the next theorem we compute the probability that $\sC=\R^d$ along with the expected conic intrinsic volumes of $\sC$ and its expected angle. We consider full-dimensional beta cones only. For cones with empty interior, see Theorem~\ref{theo:beta_cones_conic_intrinsic_vol_as_A_B_non_full_dim}.
\begin{theorem}[Expected conic intrinsic volumes and angles of beta cones] \label{theo:beta_cones_conic_intrinsic_vol_as_A_B}
Let $\sC = \sC_{n,d}^{\beta; \beta_1,\ldots, \beta_n}$ be a beta cone with $n\geq d$. Then,
\begin{align}
	\P [\sC=\R^d]
	&=2\sum_{\substack{I \subseteq \{1,\dots,n\} \\ \# I \in \lbrace d+1,d+3,\ldots \rbrace}}\Theta ( \gamma;\left\{ \g_i : i\in I\right\};\left\{ \g_i : i\in I^c\right\}), \label{eq:theo:beta_cones_conic_intrinsic_vol_as_A_B_1} \\
	\P [\sC\neq \R^d]
&= 2\sum_{\substack{I \subseteq \{1,\dots,n\} \\ \# I \in \lbrace d-1,d-3,\ldots \rbrace}} \Theta ( \gamma ;\left\{ \g_i : i\in I\right\};\left\{ \g_i : i\in I^c\right\}), \label{eq:theo:beta_cones_conic_intrinsic_vol_as_A_B_2}
\end{align}
where $I^c = \{1,\ldots, n\}\bsl I$. Further, for all $k \in \{0,\dots,d-1\}$, we have
\begin{equation}\label{eq:beta_cones_conic_intrinsic_vol_as_A_B}  
\E \left[\upsilon_k(\sC)\right]=\E \left[\upsilon_k(\sC) \ind_{\{\sC\neq \R^d\}}\right]
=
\sum_{\substack{I \subseteq \{1,\dots,n\} \\ \# I=k}}\Theta ( \gamma ;\left\{ \g_i : i\in I\right\};\left\{ \g_i : i\in I^c\right\}).
\end{equation}
In the excluded case $k=d$ we have
\begin{align}
&\E \left[\a(\sC)\right] = \E \left[\upsilon_d(\sC)\right] \notag
\\
&= 1- \sum_{\substack{I \subseteq \{1,\dots,n\} \\ \# I\leq d-1}}\Theta ( \gamma ;\left\{ \g_i : i\in I\right\};\left\{ \g_i : i\in I^c\right\})
=\sum_{\substack{I \subseteq \{1,\dots,n\} \\ \# I\geq d}}\Theta ( \gamma ;\left\{ \g_i : i\in I\right\};\left\{ \g_i : i\in I^c\right\}).
\label{eq:expect_upsilon_d_beta_cone_1}
\end{align}
The expected solid angle on the event that $\sC\neq \R^d$ is given by
\begin{align}
\E \left[\a(\sC)\ind_{\{\sC\neq \R^d\}}\right] =  \E \left[\upsilon_d(\sC)\ind_{\{\sC\neq \R^d\}}\right]
= \sum_{\substack{I \subseteq \{1,\dots,n\} \\ \# I\leq d-1}} (-1)^{d-1-\#I} \Theta ( \gamma ;\left\{ \g_i : i\in I\right\};\left\{ \g_i : i\in I^c\right\}). \label{eq:upislon_R_d_ind_BetaCone}
\end{align}
\end{theorem}
\begin{example}[Expected external angle]
Taking $k=0$ in~\eqref{eq:beta_cones_conic_intrinsic_vol_as_A_B} yields
\begin{align*}
\E[\alpha (\sC^\circ) \ind_{\{\sC\neq \R^d\}}] &= \E[\upsilon_0(C)\ind_{\{\sC\neq \R^d\}}] = \E[\upsilon_0(C)] = \Theta(\gamma; \varnothing; \{\gamma_1,\ldots, \gamma_n\}),
\\
\E[\alpha (\sC^\circ)] &= \E[\alpha (\sC^\circ) \ind_{\{\sC\neq \R^d\}}] + \P[\sC= \R^d]  = \Theta(\gamma; \varnothing; \{\gamma_1,\ldots, \gamma_n\}) + \P[\sC= \R^d].
\end{align*}
For the second formula we used that $\alpha ((\R^d)^\circ) = \alpha (\{0\}) = 1$ as well as the first formula.
\end{example}
\begin{proof}[Proof of~\eqref{eq:beta_cones_conic_intrinsic_vol_as_A_B}]
%		Let us begin with the first equation.
Since $\upsilon_{k}(\R^d) = 0$ for $k\neq d$, we have $\E [\upsilon_k(\sC)] =\E [\upsilon_k(\sC) \ind_{\{\sC\neq \R^d\}}]$.
Theorem~\ref{theo:beta_cones_conic_intrinsic_vol_as_ext_int} lets us express $\E [\upsilon_k(\sC)]$ as a sum of products of internal and external quantities. We can use Proposition~\ref{prop:int_ext_through_theta}  to express each factor of the product through the $\Theta$-function, that is we have for every index set $I \subseteq \{1,\ldots,n\}$ with $\# I = k$ that
\begin{align*}
	\intern\left(\b + \frac{d-k}{2}; \;\; \left\{ \b_i + \frac{d-k}{2}: i\in I\right\}\right) = \Theta \left(\b+\frac{d}{2}; \left\{\b_i + \frac{d}{2}: i \in I \right\}; \varnothing\right)
\end{align*}
as well as
\begin{multline*}
	\extern\left(\beta +\frac{d}{2} + \sum_{i\in I} \left( \beta_i +  \frac {d}{2} \right) - \frac{n-k}{2}; \;\; \left\{\beta_{i} + \frac {d}{2} - \frac{n-k}{2}: i \in I^c\right\}\right) \\ = \Theta \left(\beta +\frac{d}{2} + \sum_{i\in I} \left( \beta_i +  \frac {d}{2} \right); \varnothing ;  \left\{\beta_{i} + \frac {d}{2} : i\in I^c\right\}\right).
\end{multline*}
In particular, as we have a product of $\Theta$-functions, we can apply~\eqref{eq:theo:theta_as_prod_of_thetas} from Theorem~\ref{theo:theta_as_prod_of_thetas} to get
\begin{multline*}
	\Theta \left(\b+\frac{d}{2}; \left\{\b_i + \frac{d}{2}: i \in I \right\}; \varnothing\right)\Theta \left(\beta +\frac{d}{2} + \sum_{i\in I} \left( \beta_i +  \frac {d}{2} \right); \varnothing ;  \left\{\beta_{i} + \frac {d}{2} : i\in I^c\right\}\right) \\
	= \Theta \left(\b+\frac{d}{2}; \left\{\b_i + \frac{d}{2}: i \in I \right\}; \left\{\beta_{i} + \frac {d}{2} : i\in I^c\right\}\right).
\end{multline*}
Plugging the above into the representation of $\E [\upsilon_k(\sC)]$ given in~\eqref{eq:beta_cones_conic_intrinsic_vol_as_ext_int} completes the proof.
\end{proof}
%In the next section we shall compute various expected functionals of beta cones in terms of the $\Theta$-quantity.
To prove the remainder of Theorem~\ref{theo:beta_cones_conic_intrinsic_vol_as_A_B}, we need a general result which expresses the expected $d$-th conic intrinsic volume of a random cone $\cC$ in $\R^d$ (as well as some related quantities) through $\E \upsilon_k(\cC)$ with $k\leq d-1$.
\begin{proposition}\label{prop:v_d_expressed_through_lower_indices}
Let $\cC$ be any random closed convex cone in $\R^d$ such that the probability that $\cC$ is a linear subspace of dimension strictly smaller than $d$ is $0$. The probability that $\cC=\R^d$ may but need not  be positive.  Then,
\begin{align}
\E \left[\upsilon_d(\cC)\right]
&=
%1 - \sum_{k=0}^{d-1} \E [\upsilon_k(\cC)],
1 - \E \upsilon_{d-1}(\cC) - \E \upsilon_{d-2}(\cC) - \E \upsilon_{d-3}(\cC) -\ldots -  \E \upsilon_{0}(\cC),
\label{eq:expect_upsilon_rand_cone_1}
\\
\E \left[\upsilon_d(\cC)\ind_{\{\cC\neq \R^d\}}\right]
&=
\E \upsilon_{d-1}(\cC) - \E \upsilon_{d-2}(\cC) + \E \upsilon_{d-3}(\cC) -\ldots + (-1)^{d-1} \E \upsilon_{0}(\cC),
\label{eq:expect_upsilon_rand_cone_2}
\\
\P[\cC= \R^d]
&=
1 - 2\left(\E \upsilon_{d-1}(\cC) + \E \upsilon_{d-3}(\cC) + \E \upsilon_{d-5}(\cC) + \ldots \right),
\label{eq:expect_upsilon_rand_cone_3}
\\
\P[\cC \neq  \R^d]
&=
2\left(\E \upsilon_{d-1}(\cC) + \E \upsilon_{d-3}(\cC) + \E \upsilon_{d-5}(\cC) + \ldots \right).
\label{eq:expect_upsilon_rand_cone_4}
\end{align}
\end{proposition}
\begin{proof}
We use the relations for the conic intrinsic volumes given in Section~\ref{subsec:angles_conic_intrinsic_defs},
\begin{align*}
&\upsilon_0(\cC) + \upsilon_1(\cC) + \ldots + \upsilon_d(\cC) = 1,\\
&\upsilon_d(\cC) - \upsilon_{d-1}(\cC) + \upsilon_{d-2}(\cC)  - \ldots + (-1)^d \upsilon_0(\cC) = 0  \quad \text{ on } \quad  \{\cC \neq \lin \cC\}.
\end{align*}
Note further that by assumption,  $\{\cC\neq \lin \cC\}$ coincides with $\{\cC \neq \R^d\}$ up to a zero event.
Taking the expectation gives
\begin{align*}
&\E [\upsilon_0(\cC)] + \E [\upsilon_{1}(\cC)] + \ldots + \E [\upsilon_{d}(\cC)] = 1,\\
&\E\left[\upsilon_d(\cC)\ind_{\{\cC\neq \R^d\}}\right] - \E\left[\upsilon_{d-1}(\cC)\right] + \E\left[\upsilon_{d-2}(\cC)\right]  - \ldots + (-1)^d \E\left[\upsilon_0(\cC)\right] = 0,
\end{align*}
where in the second formula we used that on the event $\{\cC \neq \R^d\}$ we have $\upsilon_k(\cC) = 0$ whenever $k\neq d$ and hence
$$
\E \left[\upsilon_k(\cC) \ind_{\{\cC\neq \R^d\}}\right] = \E \left[\upsilon_k(\cC)\right], \qquad k\in\{0,\ldots, d-1\}.
$$
This proves~\eqref{eq:expect_upsilon_rand_cone_1} and~\eqref{eq:expect_upsilon_rand_cone_2}.  To prove the remaining identities observe that
$$
\E [\upsilon_d(\cC)] =  \E\left[\upsilon_d(\cC)\ind_{\{\cC\neq \R^d\}}\right] + \E\left[\upsilon_d(\cC)\ind_{\{\cC = \R^d\}}\right]
=
 \E\left[\upsilon_d(\cC)\ind_{\{\cC\neq \R^d\}}\right] + \P[\cC= \R^d]
$$
since $\upsilon_d(\R^d) = 1$. This gives, in view of~\eqref{eq:expect_upsilon_rand_cone_1} and~\eqref{eq:expect_upsilon_rand_cone_2},
$$
\P[\cC= \R^d] =  \E [\upsilon_d(\cC)] - \E\left[\upsilon_d(\cC)\ind_{\{\cC\neq \R^d\}}\right] = 1 - 2\left(\E \upsilon_{d-1}(\cC) + \E \upsilon_{d-3}(\cC) + \E \upsilon_{d-5}(\cC) + \ldots \right).
$$
The formula for $\P[\cC \neq  \R^d]$ follows by taking the complementary event.
\end{proof}

\begin{proof}[Proof of the rest of Theorem~\ref{theo:beta_cones_conic_intrinsic_vol_as_A_B}]
By Lemma~\ref{lem:beta_cones_generic_properties}, the probability that a beta cone $\sC$ is a linear subspace of dimension strictly smaller than $d$ vanishes.  Equations~\eqref{eq:expect_upsilon_d_beta_cone_1} and~\eqref{eq:upislon_R_d_ind_BetaCone} of Theorem~\ref{theo:beta_cones_conic_intrinsic_vol_as_A_B} follow from~\eqref{eq:expect_upsilon_rand_cone_1} and~\eqref{eq:expect_upsilon_rand_cone_2} combined with~\eqref{eq:beta_cones_conic_intrinsic_vol_as_A_B} and the first equation of Proposition~\ref{prop:relation_theta_even_odd}.
%The remaining equations then follow directly from the statements about the conic intrinsic volumes given in %Proposition~\ref{prop:v_d_expressed_through_lower_indices}. In particular,
Equation~\ref{eq:expect_upsilon_rand_cone_4} combined with~\eqref{eq:beta_cones_conic_intrinsic_vol_as_A_B} gives 		
\begin{align*}
\P[\sC \neq \R^{d}] = 2\sum_{\substack{I \subseteq \{1,\dots,n\} \\ \# I \in \lbrace d-1,d-3,\ldots \rbrace}} \Theta ( \gamma ;\left\{ \g_i : i\in I\right\};\left\{ \g_i : i\in I^c\right\}).
\end{align*}
By the second equation of Proposition~\ref{prop:relation_theta_even_odd}, the complementary probability $\P[\sC = \R^{d}] = 1-\P[\sC \neq \R^{d}]$ is given by the same formula but with summation extended over $I$ satisfying  $\# I \in \lbrace d+1,d+3,\ldots \rbrace$.
\end{proof}

\begin{theorem}[Beta cones with empty interior]\label{theo:beta_cones_conic_intrinsic_vol_as_A_B_non_full_dim}
Consider a beta cone $\sC=\sC_{n,d}^{\beta; \beta_1,\ldots, \beta_n}$ with $n \in \{1,\ldots, d-1\}$.  Then, for all $k\in \{0,\ldots, n\}$,
\begin{equation}\label{eq:beta_cones_conic_intrinsic_vol_as_A_B_nonfull_dim}
\E \left[\upsilon_k(\sC)\right]
=
\sum_{\substack{I \subseteq \{1,\dots,n\} \\ \# I=k}}\Theta ( \gamma ;\left\{ \g_i : i\in I\right\};\left\{ \g_i : i\in I^c\right\}).
\end{equation}
In particular, expected internal and external angles of $\sC$ are given by
\begin{align*}
\E \left[\a(\sC)\right]
&= \E \left[\upsilon_n(\sC)\right] = \Theta (\gamma ;\left\{ \g_1, \dots, \g_n \right\}; \varnothing),
\\
\E \left[\a(\sC^\circ)\right]
&= \E \left[\upsilon_0(\sC)\right] = \Theta (\gamma ;\varnothing; \left\{ \g_1, \dots, \g_n \right\}).
\end{align*}
\end{theorem}
\begin{proof}
For $k\in \{0,\ldots, n-1\}$, the same argument based on Theorem~\ref{theo:beta_cones_conic_intrinsic_vol_as_ext_int} as in the proof of~\eqref{eq:beta_cones_conic_intrinsic_vol_as_A_B} applies. For $k=n$, we use the identity $\upsilon_0(\sC) + \ldots + \upsilon_n(\sC) = 1$ and then Proposition~\ref{prop:relation_theta_even_odd} to get
$$
\E \left[\upsilon_n(\sC)\right] = 1 - \sum_{k=0}^{n-1} \E \left[\upsilon_k(\sC)\right]
=
1 - \sum_{k=0}^{n-1} \sum_{\substack{I \subseteq \{1,\dots,n\} \\ \# I=k}}\Theta ( \gamma ;\left\{ \g_i : i\in I\right\};\left\{ \g_i : i\in I^c\right\})
=
\Theta (\gamma ;\left\{ \g_1, \dots, \g_n \right\}; \varnothing).
$$
This proves~\eqref{eq:beta_cones_conic_intrinsic_vol_as_A_B_nonfull_dim} for $k=n$.
\end{proof}

\subsection{Faces, tangent and normal cones}
In the next theorem we collect formulas related to faces of beta cones. In particular, we compute the probability that a given collection of vectors spans a face and the expected values of various angles associated with this face.
\begin{theorem}[Faces and angles in beta cones]\label{theo:beta_cones_angles_as_A_B}
%Let $n \in \N$, $d \in \N$ and let $C=\pos(Z_1-Z,\ldots, Z_n-Z)=\BetaCone(\R^d;\b;\b_1,\dots,\b_n)$ be a beta cone.
For $k \in \lbrace 0, 1 ,\dots, \min (n,d)-1 \rbrace$, choose $K = \lbrace i_1 , \dots, i_k \rbrace \subseteq \lbrace 1, \dots, n \rbrace$ and let $G_K \coloneqq \pos (Z_{i_1}-Z,\dots,Z_{i_k}-Z)$ be a possible face of $\sC$ with the convention that $G_\varnothing:= \{0\}$. Also, let $T(G_K, \sC)$ be the tangent cone of $\sC$ at $G_K$ with the convention that it is defined as $\R^d$  if $G_K$ fails to be a face of $\sC$.
		Then,
		\begin{align}
		\P [G_K \text{ is a face of }\sC]
		= 2\sum_{\substack{I \subseteq K^c \\ \# I \in \lbrace d-k-1,d-k-3,\ldots \rbrace}}
		\Theta \left(\g  + \sum_{i \in K } \g_i  ;\left\{ \g_i : i\in I\right\}; \left\{\g_{i}: i\in K^c\bsl I\right\}  \right),
\label{eq:theo:beta_cones_angles_as_A_B_1} \\
		\P [G_K \text{ is not a face of }\sC]
		= 2\sum_{\substack{I \subseteq K^c \\ \# I \in \lbrace d-k+1,d-k+3,\ldots \rbrace}}
		\Theta \left(\g + \sum_{i \in K} \g_i  ;\left\{ \g_i: i\in I\right\}; \left\{\g_{i} : i\in K^c\bsl I\right\}  \right). \label{eq:theo:beta_cones_angles_as_A_B_2}
		\end{align}
%For simplicity, set $\g \coloneqq \beta+d/2$ and $\g_i = \b_i + d/2$ for $i \in \lbrace 1, \dots, n \rbrace$.
Further, the expected angle of $G_K$ is given by
	\begin{align}\label{eq:alpha_face_beta_cone}
		\E [\alpha(G_K)] = A\left(\g ;\lbrace \g_i: i \in K \rbrace \right) %=& c_{\g + \sum_{i=1}^k \g_i} \cdot \prod_{i=1}^k c_{\g_i - \frac{1}{2}} \cdot a\left(2\g+ \left(2\sum_{i=1}^k \g_i \right)+2 ; 2\g_1,\dots, 2\g_k \right) \\
		= \Theta (\gamma ; \lbrace \g_i: i \in K \rbrace;\varnothing).
	\end{align}
Additionally, for all $\ell\in \{k,\ldots, \min (n,d)-1\}$,
\begin{equation}\label{eq:upsilon_tangent_beta_cone}
\E [\upsilon_{\ell} (T(G_K,\sC))] =
\sum_{\substack{I \subseteq K^c \\ \# I=\ell-k}}
\Theta \left( \gamma + \sum_{i \in K} \g_i;\left\{ \g_i : i\in I\right\};\left\{ \g_i : i\in K^c \bsl I\right\}\right).
\end{equation}
Further,  we can express the expected solid angle of the tangent cone of $\sC$ at $G_K$ as
	\begin{align*}
		\E [\alpha(T(G_K,\sC))]
&=
\E [\upsilon_{\min (n,d)} (T(G_K,\sC))]
\\
		&=\sum_{\substack{I \subseteq K^c \\ \# I \geq \min(n-k,d-k)}} \Theta \left( \gamma + \sum_{i \in K} \g_i;\left\{ \g_i : i\in I\right\};\left\{ \g_i : i\in K^c\bsl I\right\}\right).
	\end{align*}
Assuming that $n\leq d$, this formula simplifies to
\begin{align*}
\E [\alpha(T(G_K,\sC))]
=
\Theta \left( \gamma + \sum_{i \in K} \g_i;\left\{ \g_i : i\in K^c\right\}; \varnothing\right).
\end{align*}
Assuming that $n > d$, the expected internal angle on the event that $G_K$ is a face of $\sC$ is given by
	\begin{align*}
		\E \left[\a (T(G_K,\sC))\ind_{\{G_K \text{ is a face of }\sC\}}\right]
		&= \sum_{\substack{I \subseteq K^c \\ \# I \leq d-k-1}}  (-1)^{d-k-1-\#I}
		\Theta \left( \gamma + \sum_{i \in K} \g_i;\left\{ \g_i : i\in I\right\};\left\{ \g_i : i\in K^c\bsl I\right\}\right).
	\end{align*}
For the expected solid angle of the normal cone of $\sC$ at $G_K$, we have,
	\begin{align*}
		\E [\alpha(N(G_K,\sC))]
		=\Theta \left(\gamma + \sum_{i \in K} \g_i; \varnothing;\left\{\g_{i}: i \in K^c\right\}\right) + \P [G_K \text{ is not a face of }\sC].
	\end{align*}
Assuming that $n > d$, the expected external angle on the event that $G_K$ is a face of $\sC$ is given by
\begin{align*}
\E \left[\alpha(N(G_K,\sC))\ind_{\{G_K \text{ is a face of } \sC\}}\right]
=
\Theta \left(\gamma + \sum_{i \in K} \g_i; \varnothing;\left\{\g_{i}: i \in K^c\right\}\right).
%= B\left(\g +\sum_{i=1}^{n} \g_i; \left\{\g_{i}: i \in K^c\right\} \right).
\end{align*}
\end{theorem}
\begin{remark}[The case $k=0$]
For $k=0$, the statements in the above theorem can be interpreted as follows. Since $G_\varnothing = \lbrace0\rbrace$, the event that $G_\varnothing \text{ is a face of } \sC$ happens if and only if $\sC \neq \R^d$. Further, $T(\{0\},\sC) = \sC$, $N(\{0\}, \sC) = \sC^\circ$ and $\alpha (\{0\}) = 1$.
\end{remark}
	\begin{proof}[Proof of Theorem~\ref{theo:beta_cones_angles_as_A_B}]
To prove~\eqref{eq:alpha_face_beta_cone}, we use Theorem \ref{theo:beta_cones_faces_alpha_and_gamma}, Equation \eqref{eq:A_through_Int} and Proposition~\ref{prop:int_ext_through_theta}, which give
	\begin{align*}
		\E \alpha(G_K) = \intern\left(\g - \frac{k}{2}; \g_{i_1} - \frac{k}{2}, \dots, \g_{i_k} - \frac{k}{2}\right) = A\left(\g ;\lbrace \g_i: i \in K \rbrace \right) = \Theta \left(\gamma ; \lbrace \g_i: i \in K \rbrace;\varnothing\right).
	\end{align*}
Next we prove~\eqref{eq:theo:beta_cones_angles_as_A_B_1} and~\eqref{eq:theo:beta_cones_angles_as_A_B_2}. We know from Theorem~\ref{theo:representation_angles_and_construction_of_Cone} that $T(G_K,\sC)$ is isometric to
		\begin{align*}
			\R^k \oplus \BetaCone  \left(\R^{d-k}; \b + \frac{k}{2} +  \sum_{i \in K} \left(\beta_i+\frac{d}{2}\right); \left\{ \b_i + \frac{k}{2}: i \in K^c  \right\} \right).
		\end{align*}
Let $\mathcal{D}_K$  denote the beta cone on the right-hand side. In view of  Proposition~\ref{prop:face_events_as_absorbtion_events_polytopes},
		\begin{align*}
\P [G_K \text{ is a face of }\sC]
&= \P[T(G_K,\sC) \neq \R^{d}] = \P[\mathcal{D}_K \neq \R^{d-k}]
\\
&=
2\sum_{\substack{I \subseteq K^c \\ \# I \in \lbrace d-k-1,d-k-3,\ldots \rbrace}}
		\Theta \left(\g  + \sum_{i \in K } \g_i  ;\left\{ \g_i : i\in I\right\}; \left\{\g_{i}: i\in K^c\bsl I\right\}  \right),
		\end{align*}
		where we have used Theorem~\ref{theo:beta_cones_conic_intrinsic_vol_as_A_B} for the last equation. This proves~\eqref{eq:theo:beta_cones_angles_as_A_B_1}.  Proposition~\ref{prop:relation_theta_even_odd} gives~\eqref{eq:theo:beta_cones_angles_as_A_B_2}.

We proceed to the proof of~\eqref{eq:upsilon_tangent_beta_cone}. For $\ell \in \{k,\ldots, \min (n,d)-1\}$ we have $\upsilon_{\ell} (\R^k \oplus \mathcal{D}_K) =\upsilon_{\ell -k}(\mathcal{D}_K)$.
%while we  have $\upsilon_{\ell} (\R^k \oplus \mathcal{D}_K) = 0$ in the case $\ell < k$; see, e.g., \cite[p.~378]{amelunxen_lotz}.
 Hence,
		\begin{align*}
			\E [\upsilon_{\ell} (T(G_K,\sC))] =  \E [\upsilon_{\ell} (\R^k \oplus \mathcal{D}_K)] = \E [\upsilon_{\ell -k}(\mathcal{D}_K)].
		\end{align*}
		 Therefore, we can use~\eqref{eq:beta_cones_conic_intrinsic_vol_as_A_B} of Theorem~\ref{theo:beta_cones_conic_intrinsic_vol_as_A_B} (if $n\geq d$) or Theorem~\ref{theo:beta_cones_conic_intrinsic_vol_as_A_B_non_full_dim} (if $n<d$) to express
		\begin{equation}\label{eq:upsilon_D_k_tangent}
			\E \left[\upsilon_{\ell-k}(\mathcal{D}_K)\right]
			=
			\sum_{\substack{I \subseteq K^c \\ \# I=\ell-k}}
			\Theta \left( \gamma + \sum_{i \in K} \g_i;\left\{ \g_i : i\in I\right\};\left\{ \g_i : i\in K^c \bsl I\right\}\right).
		\end{equation}
Notice that when using these theorems, we have to add $\frac{d-k}{2}$ to the parameters of the beta cone $\mathcal{D}_K$,  $d-k$ being the dimension of the space $\mathcal{D}_K$ lives in. This turns the existing $\frac{k}{2}$ in these parameters into $\frac d2$ and thus we arrive at $\gamma$ and $\gamma_i$ in the above equation. This proves~\eqref{eq:upsilon_tangent_beta_cone}.

Next we  calculate the expected solid angle of $T(G_K,\sC)$ using~\eqref{eq:expect_upsilon_d_beta_cone_1} of Theorem~\ref{theo:beta_cones_conic_intrinsic_vol_as_A_B} (if $n\geq d$) or Theorem~\ref{theo:beta_cones_conic_intrinsic_vol_as_A_B_non_full_dim} (if $n<d$),
\begin{align*}
\E [\a (T(G_K,\sC))] %= \E [\upsilon_{\min(n,d)}(T(G_K,\sC))] = \E [\upsilon_{\min(n-k,d-k)}(\mathcal{D}_K)] \\
&=
\E[\alpha(\R^k \oplus \mathcal{D}_K))]
=
\E[\alpha(\mathcal{D}_K)]
=
\E [\upsilon_{\min (n-k,d-k)} (\mathcal{D}_K)]
\\
&= 1-
			\sum_{\substack{I \subseteq K^c \\ \# I \leq \min(n-k-1,d-k-1)}}
			\Theta \left( \gamma + \sum_{i \in K} \g_i;\left\{ \g_i : i\in I\right\};\left\{ \g_i : i\in K^c \bsl I\right\}\right)\\
&=
\sum_{\substack{I \subseteq K^c \\ \# I \geq \min(n-k,d-k)}}
			\Theta \left( \gamma + \sum_{i \in K} \g_i;\left\{ \g_i : i\in I\right\};\left\{ \g_i : i\in K^c \bsl I\right\}\right).
\end{align*}

Next we compute the expectation of $\a (T(G_K,\sC))$ on the event that  $G_K$ is a face of $\sC$.  We assume that  $n > d$, since if $n \leq d$, $G_K$ is always a face of $\sC$. The event that $G_K$ is a face of $\sC$ takes place if and only if $T(G_K,\sC)\neq \R^d$. Also, $\dim \mathcal{D}_K=d-k$, which gives
		\begin{multline*}
			\E [\a (T(G_K,\sC))\ind_{\{G_K \text{ is a face of }\sC\}}]= \E [\a (T(G_K,\sC))\ind_{\{T(G_K,\sC)\neq\R^d\}}] = \E [\upsilon_{d-k}(\mathcal{D}_K)\ind_{\{\mathcal{D}_K \neq \R^{d-k}\}}] \\
			= \sum_{\substack{I \subseteq K^c \\ \# I \leq d-k-1}}  (-1)^{d-k-1-\#I} \Theta \left( \gamma + \sum_{i \in K} \g_i;\left\{ \g_i : i\in I\right\};\left\{ \g_i : i\in K^c\bsl I\right\}\right).,
		\end{multline*}
where in the last step we used~\eqref{eq:upislon_R_d_ind_BetaCone} of Theorem~\ref{theo:beta_cones_conic_intrinsic_vol_as_A_B}.

For the expected  solid angle of the normal cone $N(G_K,\sC)$, we can use~\eqref{eq:upsilon_D_k_tangent} to write
		\begin{multline*}
		\E [\a (N(G_K,\sC))\ind_{\{G_K \text{ is a face of }\sC\}}] =\E [\upsilon_{0}(\mathcal{D}_K)] =
		\Theta \left( \gamma + \sum_{i \in K} \g_i; \varnothing; \left\{ \g_i : i\in K^c \right\}\right)
\\
		=\extern \left(\gamma-\frac{n-k}{2}+\sum_{i \in K}\gamma_i;\left\{\gamma_i -\frac{n-k}{2}: i \in K^c\right\}\right).
		%=B\left(\g +\sum_{i=1}^{n} \g_i ; \left\{\g_i: i \in K^c\right\} \right).
		\end{multline*}
		Observe that this  recovers the formula for $\E [\gamma(G_K,\sC)\ind_{\{G_K \text{ is a face of } \sC\}}]$ established in~\eqref{eq:beta_cones_external_angles_as_external_quant}. To prove a similar formula without the indicator function, write
$$
\E [\alpha(N(G_K,\sC))] = \E \left[\alpha(N(G_K,\sC))\ind_{\{G_K \text{ is a face of } \sC\}}\right] + \E \left[\alpha(N(G_K,\sC))\ind_{\{G_K \text{ is a not face of } \sC\}}\right]
$$
Since $\alpha(N(G_K,\sC)) = 1$   if $G_K$ is not a face of $\sC$, this proves the formula for $\E [\alpha(N(G_K,\sC))]$.
\end{proof}

\begin{theorem}[Exected $f$-vector of beta cones] \label{theo:beta_cone_expected_f_k}
For every $k\in \{0,1,\ldots, d-1\}$, the expected number of $k$-dimensional faces of $\sC \sim \BetaCone(\R^d;\b;\b_1,\dots,\b_n)$ is given by
$$
\E f_k(\sC) = 2 \sum_{\substack{K\subseteq \{1,\ldots, n\}\\ \# K = k}}
\sum_{\substack{I \subseteq K^c \\ \# I \in \lbrace d-k-1,d-k-3,\ldots \rbrace}}
		\Theta \left(\g  + \sum_{i \in K } \g_i  ;\left\{ \g_i : i\in I\right\}; \left\{\g_{i}: i\in K^c\bsl I\right\}  \right).
$$
\end{theorem}
\begin{proof}
Write $\E f_k(\sC) =  \sum_{\substack{K \subseteq \{1,\dots,n\},\# K=k}} \P [G_K\text{ is a face of } \sC]$ and apply Equation~\eqref{eq:theo:beta_cones_angles_as_A_B_1} from Theorem~\ref{theo:beta_cones_angles_as_A_B} to each term.
\end{proof}

\subsection{Two extremal regimes for beta cones}
We are now going to consider a beta cone $\sC = \pos(Z-Z_1,\ldots, Z-Z_n) \sim \sC_{n,d}^{\beta; \beta_1,\ldots, \beta_{n}}$ in two extremal cases: when $\beta= -1$ and when $\beta \uparrow +\infty$.

\subsubsection{Beta cones with \texorpdfstring{$\beta = -1$}{beta = -1}}
If $\sC$ is a beta cone with $\beta= -1$ meaning that $Z\sim f_{d,-1}$ is uniformly distributed on the unit sphere $\bS^{d-1}$, then $\P[\sC= \R^d] = 0$. This leads to a non-trivial formula for the $a$-function which does not directly follow from its analytic definition.
\begin{proposition}[Special value of $a$-quantity]\label{prop:a_quant_special_val_d}
For all $d\in \N$ and  $\ell \in \N_0$ such that $d-2\ell \geq 1$,  and all $\alpha_1,\ldots, \alpha_{d+1} \geq 0$ we have
\begin{equation}\label{eq:a_quant_0}
a(d - 2\ell + \alpha_1+\ldots+\alpha_{d+1}; \alpha_1,\ldots, \alpha_{d+1}) = 0.
\end{equation}
\end{proposition}
\begin{proof}
First we prove the claim for $\ell = 0$.
Take arbitrary $\beta_1,\ldots,\beta_{d+1} \geq -1$, let $\beta:= -1$ and consider the beta cone $\sC = \sC_{d+1,d}^{-1; \beta_1,\ldots, \beta_{d+1}}$.  Since $\beta= -1$ corresponds to the uniform distribution on the sphere, we have $\P[\sC = \R^d] = 0$. On the other hand, \eqref{eq:theo:beta_cones_conic_intrinsic_vol_as_A_B_1}	with $n=d+1$ gives
\begin{align*}
0 = \P [\sC=\R^d]
&=
2 \cdot \Theta \left(\frac d2-1; \{\g_1,\ldots, \g_{d+1}\};\varnothing\right)
\\
&=
2 \cdot c_{\frac d2 - 1 + \gamma_1 + \ldots + \gamma_{d+1}} \cdot  c_{\gamma_1 - \frac 12}\ldots c_{\gamma_{d+1}- \frac 12} \cdot  a(d + 2\gamma_1 + \ldots + 2 \gamma_{d+1}; 2\g_1,\dots, 2\g_{d+1}),
\end{align*}
where we used the first formula from Lemma~\ref{lem:Theta_with_one_set_empty} in the second equation. Since $2\gamma_1,\ldots, 2\gamma_{d+1}$ may take arbitrary values $\geq d-2$, it follows that
$$
a(d+\alpha_1+\ldots+\alpha_{d+1}; \alpha_1,\ldots, \alpha_{d+1})=0
$$
for all $\alpha_1,\ldots, \alpha_{d+1} \geq d-2$. By analytic continuation, the identity continues to hold for $\alpha_1,\ldots, \alpha_{d+1}\geq 0$. This proves~\eqref{eq:a_quant_0} for $\ell= 0$.

Now we argue by induction to extend~\eqref{eq:a_quant_0} to all $\ell\in \N_0$ such that $d-2\ell \geq 1$. Suppose we proved~\eqref{eq:a_quant_0} for $d\in \N$  and $\ell\in \N_0$, as well as for all strictly smaller values of $d$.  Let also $d- 2(\ell+1) \geq 1$.  To prove that~\eqref{eq:a_quant_0} holds with $\ell$ replaced by $\ell+1$, we use  recurrence relation~\eqref{eq:a_first_relation} from Theorem~\ref{theo:recurrence_relations_a_b} with $\alpha = d - 2(\ell+1) + \a_1 +\ldots+\a_{d+1}\geq 1$:
\begin{multline*}
(\a+1) \cdot a(d - 2\ell + \a_1 +\ldots+\a_{d+1};\a_1,\dots,\a_{d+1}) - \a \cdot a(d - 2(\ell+1) + \a_1 +\ldots+\a_{d+1};\a_1,\dots,\a_{d+1})
\\
= \sum_{\substack{j_1,j_2\in \{1,\ldots, d+1\}\\ j_1 \neq j_2}} \frac{a(d-2 - 2\ell + (\a_1 +\ldots+\a_{d+1}-\a_{j_1}-\a_{j_2});\widehat{\a_{j_1}},\widehat{\a_{j_2}})}{\a-\a_{j_1}}.
\end{multline*}
The first term on the left-hand side vanishes by induction assumption. All summands on the right-hand side also vanish since by the induction assumption~\eqref{eq:a_quant_0} holds with $d$ replaced by $d-2$. Hence, the remaining term also vanishes: $a(d - 2(\ell+1) + \a_1 +\ldots+\a_{d+1};\a_1,\dots,\a_{d+1}) = 0$, and the induction is complete.
\end{proof}

\begin{corollary}\label{cor:theta_quant_vanishes_d/2-1}
Let $d\in \N$ and let $Y$ and $Z$ be multisets of non-negative numbers such that $\#Y \in \{d+1,d+3,\ldots\}$. Then,
$$
\Theta \left(\frac{d}{2}-1 ;Y;Z\right) = 0.
$$
\end{corollary}
\begin{proof}
The formula for $\Theta (\frac{d}{2}-1 ;Y;Z)$ from Definition~\ref{def:theta} contains a term  $a(d+2\sum_{\omega \in Y} \omega;2Y)$ which vanishes by Proposition~\ref{prop:a_quant_special_val_d}.
\end{proof}

\subsubsection{Beta cones with \texorpdfstring{$\beta\uparrow +\infty$}{beta = + infinity}}
Let $Z_1,\ldots, Z_n$ be independent and uniformly distributed on the unit sphere $\bS^{d-1}$. The so-called \emph{Donoho--Tanner} or \emph{Wendel cone}
$$
\sW_{n,d}:= \pos (Z_1,\ldots, Z_n) \subseteq \R^d
$$
can be considered as the limit of $\sC_{n,d}^{\beta; \beta_1,\ldots, \beta_n}$ as $\beta \to+\infty$.
%since in this regime the point $Z$ converges weakly to $0$.
Recall from Proposition~\ref{prop:Theta_function_gamma_infinity} that $\Theta(\gamma;Y;Z) \to  2^{-\# Y - \#Z}$ as $\gamma \to +\infty$. Inserting this limit into Theorems~\ref{theo:beta_cones_conic_intrinsic_vol_as_A_B}, \ref{theo:beta_cones_angles_as_A_B} and~\ref{theo:beta_cone_expected_f_k} yields the following
\begin{theorem}[Wendel--Donoho--Tanner cones]\label{theo:wendel_cones_angles}
For all $d\in \N$, $n\geq d$ we have
\begin{align}
&
\P[\sW_{n,d} \neq \R^d] =  \frac 1  {2^{n-1}} \sum_{\ell=0}^{d-1} \binom{n-1}{\ell},
\qquad
\P[\sW_{n,d} = \R^d] =  \frac 1  {2^{n-1}} \sum_{\ell=d}^{n-1} \binom{n-1}{\ell},
\label{eq:donoho_tanner_cone_expect_1}
\\
&\E \left[\upsilon_k(\sW_{n,d})\right]
=
\E \left[\upsilon_k(\sW_{n,d}) \ind_{\{\sW_{n,d}\neq \R^d\}}\right]
=
\frac 1 {2^n} \binom nk,
\qquad k \in \{0,\dots,d-1\},
\label{eq:donoho_tanner_cone_expect_2}
\\
&\E \left[\upsilon_d(\sW_{n,d})\right]
= 1- \frac 1 {2^n} \sum_{\ell=0}^{d-1}  \binom n\ell = \frac 1 {2^n} \sum_{\ell=d}^{n}  \binom n\ell,
\label{eq:donoho_tanner_cone_expect_3a}
\\
&\E \left[\upsilon_d(\sW_{n,d})\ind_{\{\sW_{n,d}\neq \R^d\}}\right]
=
\frac{1}{2^n} \sum_{\ell=0}^{d-1} (-1)^{d-1-\ell} \binom{n}{\ell}
=
\frac 1 {2^n} \binom{n-1}{d-1},
\label{eq:donoho_tanner_cone_expect_3b}\\
&\E f_k(\sW_{n,d})
=
\frac{1}{2^{n-k-1}}\binom{n}{k}\sum_{\ell=0}^{d-k-1}\binom{n-k-1}{\ell}, \qquad 0\le k < d\le n.
\label{eq:donoho_tanner_cone_expect_4}
\end{align}
Additionally, for $d\in \N$ and $n\in \{1,\ldots, d\}$ we have $\E [\upsilon_k(\sW_{n,d})] = \frac 1 {2^n} \binom nk$,  $k \in \{0,\dots,n\}$.
\end{theorem}
These formulas are known. Equation~\eqref{eq:donoho_tanner_cone_expect_1} follows from Wendel's formula, see~\cite{Wendel} and~\cite[Theorem~8.2.1]{schneider_weil_book}. Equations~\eqref{eq:donoho_tanner_cone_expect_2}, \eqref{eq:donoho_tanner_cone_expect_3a}, \eqref{eq:donoho_tanner_cone_expect_3b} can be deduced from~\cite[Corollaries 4.2 and 4.3]{hug_schneider_conical_tessellations} or~\cite[Lemma~5.1]{godland_kabluchko_thaele_cones_high_dim_part_I}. Finally, Equation~\eqref{eq:donoho_tanner_cone_expect_4} was proved in~\cite[Theorem 1.6]{donoho_tanner1}. Conditioning $\sW_{n,d}$ on the event $\{\sW_{n,d} \neq \R^d\}$ leads to Cover--Efron cones studied by~\citet{hug_schneider_conical_tessellations}.   Note, however, that all these references prove the formulas under less restrictive distributional assumptions on $Z_1,\ldots, Z_n$.

\section{Results on beta polytopes}\label{sec:beta_polytopes_results}
We can now use the results obtained in 
the previous sections to give explicit formulas for expected functionals of beta polytopes. Unless stated otherwise, this section considers the following setting. Fix a dimension $d\in \N$ and a number of points $n\geq 2$. Let $X_1,\dots,X_n$ be independent beta-distributed random points in $\BB^d$, that is, for each $i \in \lbrace 1, \dots, n \rbrace$, let $X_i\sim f_{d,\beta_i}$, where $\b_i \geq -1$.
We set
$$
\sP \coloneqq \sP_{n,d}^{\b_1,\dots,\b_n}=[X_1,\dots,X_n] \sim \BetaPoly(\R^d;\b_1,\dots,\b_n)
$$
to be the beta polytope these points generate. Also, we continue to use the notational convention
$$
\g_i := \b_i + \frac d 2  \qquad \text{ for all } i \in \lbrace 1, \dots,n \rbrace.
$$
If $d=1$, we additionally require that $\beta_1,\ldots, \beta_n\neq -1$. Under these assumptions,  Lemma~\ref{lem:beta_poi_general_affine} guarantees that the polytope $\sP$ is simplicial a.s.,\ that is all of its proper faces are simplices.

\subsection{Tangent cones, angles and faces}
Let us start by characterizing the tangent cones at faces of the polytope $\sP$.
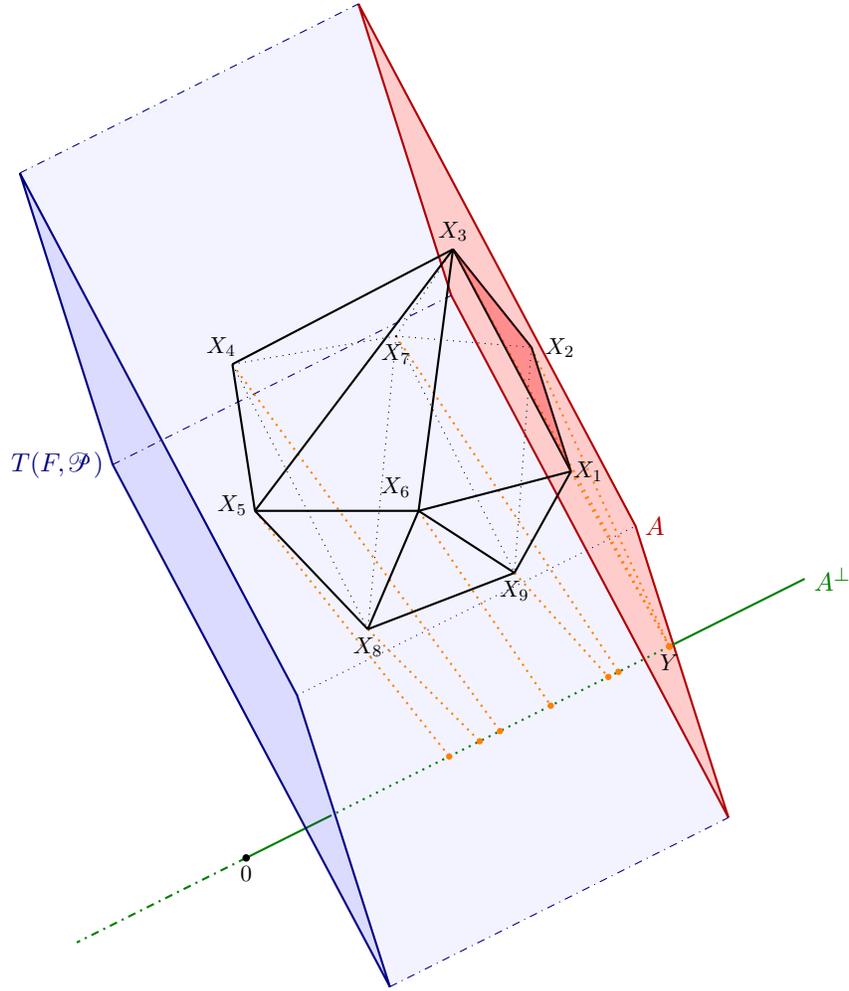
\begin{figure}[h!]
	\centering
	\label{fig:beta_polytope}
	\begin{tikzpicture}[scale=1.5]

		%--------------------------------------------------------------------------
		% 1) Define original points
		%--------------------------------------------------------------------------
		\coordinate (Z1) at (0.5,2.6); %X4
		\coordinate (Z2) at (0.7,1.3); %X5
		\coordinate (Z3) at (2.15,1.3); %X6
		\coordinate (Z4) at (3.5,1.65); %X1
		\coordinate (Z5) at (3.15,2.75); %X2
		\coordinate (Z6) at (1.95,2.85); %X7
		\coordinate (Z7) at (1.7,0.25); %X8
		\coordinate (Z8) at (3.0,0.75); %X9
		\coordinate (Z9) at ($(3.5,1.65) + 0.7*(-0.35,1.1) + 0.8*(-1,1.5)$); %X3

		%--------------------------------------------------------------------------
		% 2) Define the red plane spanned by v1 = (Z5 - Z4), v2 = (Z9 - Z4)
		%--------------------------------------------------------------------------
		\coordinate (v1) at ($(Z5)-(Z4)$);
		\coordinate (v2) at ($(Z9)-(Z4)$);
		
		%\coordinate (J) at
		  %($(Z4) + 0.7*(v1) + 0.8*(v2)$);
		  %($(3.5,1.65) + 0.7*(-0.35,1.1) + 0.8*(-1,1.5)$);
		%\filldraw[black] (J) circle (5pt);

		\pgfmathsetmacro{\alpha}{-1}
		\pgfmathsetmacro{\beta}{1.35}
		
		\coordinate (P1) at ($(Z4) + \alpha*(v1) + \alpha*(v2)$);
		\coordinate (P2) at ($(Z4) + \beta*(v1)  + \alpha*(v2)$);
		\coordinate (P3) at ($(Z4) + \beta*(v1)  + \beta*(v2)$);
		\coordinate (P4) at ($(Z4) + \alpha*(v1) + \beta*(v2)$);
		
		% Fill + outline the red plane
		\fill[red, opacity=0.2] (P1) -- (P2) -- (P3) -- (P4) -- cycle;
		\draw[red!70!black, thick] (P1) -- (P2) -- (P3) -- (P4) -- cycle
		node[red!70!black, right, scale=0.9] at (P2) {$A$};
		
		%--------------------------------------------------------------------------
		% 3) Pick a point M inside the red plane
		%    We'll do M = Z4 + 0.2*(v1) + 0.2*(v2)
		%--------------------------------------------------------------------------
		\coordinate (M) at
		  ($(Z4) + 0.2*(v1) + -0.9*(v2)$);
		
		% Numerically, M is about (3.23, 2.17).
		
		%--------------------------------------------------------------------------
		% 4) A direction vector orthogonal to the red plane in 2D: shiftVec = (-3,-1.5)
		%--------------------------------------------------------------------------
		\coordinate (shiftVec) at (-3.0, -1.5);
		
		%--------------------------------------------------------------------------
		% 5) Define the "orthogonal" line from M, extended to fill the diagram
		%--------------------------------------------------------------------------
		% We'll just show it from M - 1.5*shiftVec to M + 2*shiftVec
		\coordinate (Lstart) at
		  ($(M) + -0.4*(shiftVec)$);
		\coordinate (Lmiddle) at
		  ($(M) +  1.0*(shiftVec)$);
		  \coordinate (Lmiddle2) at
		  ($(M) +  1.25*(shiftVec)$);
		\coordinate (Lend) at
		  ($(M) +  1.75*(shiftVec)$);
		
		% Draw it in green
		\draw[thick, green!50!black] (Lstart) -- (M)
		node[ green!50!black, right, scale=0.9] at (Lstart) {$A^\perp$};
		\draw[thick, dotted,green!50!black] (M) -- (Lmiddle);
		\draw[thick,green!50!black] (Lmiddle) -- (Lmiddle2);
		\draw[thick,dashdotted,green!50!black] (Lmiddle2) -- (Lend);
		% 4) Another plane parallel to the red one
		%    We'll shift Z4 by some offset (dx, dy), then use the same (v1, v2).
		%--------------------------------------------------------------------------
		% SHIFT: pick, for example, +1.0 in the x-direction and -0.5 in y

		% We'll define an anchor for the new plane: A = Z4 + shiftVec
		\coordinate (A) at ($(Z4) + (shiftVec)$);

		\coordinate (Q1) at
			($(A) + \alpha*(v1) + \alpha*(v2)$);
		\coordinate (Q2) at
			($(A) + \beta*(v1)  + \alpha*(v2)$);
		\coordinate (Q3) at
			($(A) + \beta*(v1)  + \beta*(v2)$);
		\coordinate (Q4) at
			($(A) + \alpha*(v1) + \beta*(v2)$);

		% We'll draw it in a different color, e.g. blue
		\fill[blue, opacity=0.1] (Q1) -- (Q2) -- (Q3) -- (Q4) -- cycle
		node [pos=0, left, opacity=1, scale=0.9, blue!50!black] {$T(F,\sP)$};
		\draw[blue!50!black, thick] (Q1) -- (Q2) -- (Q3) -- (Q4) -- cycle;
		\draw[blue!50!black,dashdotted] (Q1) -- (P1);
		\draw[blue!50!black, dotted] (Q2) -- (P2);
		\draw[blue!50!black, dashdotted] (Q3) -- (P3);
		\draw[blue!50!black,dashdotted] (Q4) -- (P4);
		\fill[blue, opacity=0.05] (Q1) -- (Q4) -- (P4) -- (P1) -- cycle;
		\fill[blue, opacity=0.05] (Q4) -- (Q3) -- (P3) -- (P4) -- cycle;

		%--------------------------------------------------------------------------
% 3) Variable anchor point M that *definitely* lies in the red plane
%    We define: M = Z4 + 0.2*(v1) + 0.3*(v2).
%--------------------------------------------------------------------------
%

%--------------------------------------------------------------------------
% 4) Define direction vector orthogonal to the red plane in 2D
%--------------------------------------------------------------------------
\coordinate (dir) at (-3,-1.5);

%--------------------------------------------------------------------------
% 5) Manually computed orthogonal projections of Z1..Z9
%    (You can adapt c-values based on the dot-product formula, or just
%     plug in the scalars you already computed.)
%--------------------------------------------------------------------------
%
% For demonstration, let's just pick approximate c-values. In practice,
% you'd compute c_i = ((Z_i - M)·dir) / (dir·dir).  We'll do a quick example:

% Z1
\coordinate (Z1proj) at
  ($(M) + 0.50*(dir)$);

% Z2
\coordinate (Z2proj) at
  ($(M) + 0.65*(dir)$);

% Z3
\coordinate (Z3proj) at
  ($(M) + 0.35*(dir)$);

% Z4
\coordinate (Z4proj) at
  ($(M) + 0.00*(dir)$);

% Z5
\coordinate (Z5proj) at
  ($(M) + -0.0*(dir)$);

% Z9
\coordinate (Z9proj) at
  ($(M) + 0.0*(dir)$);

% Z7
\coordinate (Z7proj) at
  ($(M) + 0.56*(dir)$);

% Z8
\coordinate (Z8proj) at
  ($(M) + 0.18*(dir)$);

% Z6
\coordinate (Z6proj) at
  ($(M) + 0.15*(dir)$);
		
		% Connect each Z_i to its projection with an orange dotted line
		\draw[dotted, thick, orange] (Z1) -- (Z1proj);
		\draw[dotted, thick, orange] (Z2) -- (Z2proj);
		\draw[dotted, thick, orange] (Z3) -- (Z3proj);
		\draw[dotted, thick, orange] (Z4) -- (Z4proj);
		\draw[dotted, thick, orange] (Z5) -- (Z5proj);
		\draw[dotted, thick, orange] (Z9) -- (Z9proj);
		\draw[dotted, thick, orange] (Z7) -- (Z7proj);
		\draw[dotted, thick, orange] (Z8) -- (Z8proj);
		\draw[dotted, thick, orange] (Z6) -- (Z6proj);
		
		% Optionally put small orange dots at the projection points
		\foreach \pt in {Z1proj,Z2proj,Z3proj,Z4proj,Z5proj,Z9proj,Z7proj,Z8proj,Z6proj} {
		  \fill[orange] (\pt) circle (0.8pt);
		}
		%-------------------------------------------------
		%add the origin
		%-------------------------------------------------
		\coordinate (origin) at ($(Z4proj)+1.25*(shiftVec)$);
		\filldraw[black] (origin) circle (0.8pt);
		% 7) Draw the edges of the polytope on top
		%--------------------------------------------------------------------------
		\draw[thick, black] (Z1) -- (Z2);
		\draw[thick, black] (Z2) -- (Z3);
		\draw[thick, black] (Z3) -- (Z4);
		\draw[thick, black] (Z4) -- (Z5);
		\draw[thick, black] (Z1) -- (Z9);
		\draw[thick, black] (Z2) -- (Z9);
		\draw[thick, black] (Z3) -- (Z9);
		\draw[thick, black] (Z4) -- (Z9);
		\draw[thick, black] (Z5) -- (Z9);
		\draw[thick, black] (Z2) -- (Z7);
		\draw[thick, black] (Z3) -- (Z7);
		\draw[thick, black] (Z7) -- (Z8);
		\draw[thick, black] (Z8) -- (Z4);
		\draw[thick, black] (Z3) -- (Z8);
		\draw[dotted, black] (Z1) -- (Z7);
		\draw[dotted, black] (Z5) -- (Z8);
		\draw[dotted, black] (Z1) -- (Z6);
		\draw[dotted, black] (Z6) -- (Z5);
		\draw[dotted, black] (Z6) -- (Z9);
		\draw[dotted, black] (Z6) -- (Z7);
		\draw[dotted, black] (Z6) -- (Z8);
		
		% And re-draw the small red triangle face Z4--Z5--Z9
		\fill[red, opacity=0.3] (Z4) -- (Z5) -- (Z9) -- cycle;
		
		%--------------------------------------------------------------------------
		% 8) Labels
		%--------------------------------------------------------------------------
		\node[above, scale=0.8] at ($(Z1)+(-0.1,0)$) {$X_4$};
		\node[left, scale=0.8]  at ($(Z2)+(0,0.05)$) {$X_5$};
		\node[above left, scale=0.8] at ($(Z3)+(0,0.05)$) {$X_6$};
		\node[right, scale=0.8] at ($(Z4)+(-0.05,0)$) {$X_1$};
		\node[right, scale=0.8] at ($(Z5)+(0.05,0)$) {$X_2$};
		\node[below, scale=0.8] at (Z6) {$X_7$};
		\node[below, scale=0.8] at (Z7) {$X_8$};
		\node[below, scale=0.8] at (Z8) {$X_9$};
		\node[above, scale=0.8] at (Z9) {$X_3$};

		\node[below, scale=0.8] at (Z4proj) {$Y$};
		\node[below, scale=0.8] at (origin) {$0$};
		
		\end{tikzpicture}
  	\caption{Visualization of the proof of Theorem~\ref{theo:tangent_cone_is_beta_cone}. Here, we have a beta polytope $\sP=\sP_{9,3}^{\b_1,\dots,\b_9}$. The simplex $F=[X_1,X_2,X_3]$ is a 2-dimensional face of the polytope. The red plane is the affine subspace $A=\aff(X_1,X_2,X_3)$, and its orthogonal complement is the green line. The projected points $Y_4, \dots, Y_9$ are the orange dots on the dashed green line. The cone $\pos(Y_4-Y,\dots,Y_9-Y)$ is the part of the green line that is dashdotted. One can clearly see that the blue tangent cone $T(F,\sP)$ (which is in fact a half-space) is isometric to $(A-\bar{X}) \oplus \pos(Y_4-Y,\dots,Y_9-Y)$, where $\bar{X}=\frac{X_1+X_2+X_3}{3}$.
	}
  	\label{fig:beta_polytopes_projection}
\end{figure}

\begin{theorem}[Tangent cones of beta polytopes]\label{theo:tangent_cone_is_beta_cone}
For $k\in \{1,\ldots, \min(d,n-1)\}$ and  $K = \{i_1,\dots,i_k\}  \subseteq \{1,\dots,n\}$, denote by $F_K \coloneqq [X_{i_1},\dots,X_{i_k}]$ a possible face of $\sP$.
Let $T(F_K,\sP)$ be the tangent cone of $\sP$ at $F_K$ (defined as $\R^d$ if $F_K$ fails to be a face of $\sP$). Then,  $T(F_K,\sP)$ is isometric to
\begin{align*}
\R^{k-1} \oplus \BetaCone\left(\R^{d-k+1};\left(\sum_{i\in K} \gamma_i \right) -\frac{d-k+1}{2};\left\{\gamma_i-\frac{d-k+1}{2}; i\in K^c\right\}\right).
\end{align*}
Here, we recall that  $K^c = \{1,\ldots, n\}\bsl K$.
\end{theorem}
\begin{remark}
The number $\min (d,n-1)$ appearing above is the number of points generating facets of $\sP$. For example, if $n\leq d$, then $\sP$ is a simplex with $n$ vertices and its facets have $n-1$ vertices.
\end{remark}
\begin{example}
Let $k=1$ and, without loss of generality, $K=\{1\}$. Then, $T(\{X_{1}\},\sP) = \pos (X_2 - X_{1}, \ldots, X_n-X_1) \sim \BetaCone(\R^d; \beta_1; \beta_2,\ldots, \beta_n)$ by definition of beta cones.
\end{example}
\begin{proof}[Proof of Theorem~\ref{theo:tangent_cone_is_beta_cone}]
The proof relies on Proposition~\ref{prop:beta_cones_joint_distr_projections} and follows the idea of the proof of Theorem~\ref{theo:representation_angles_and_construction_of_Cone}. Without restriction of generality, we take $K= \{1,\ldots, k\}$.
Let $A \coloneqq \aff (X_1,\ldots,X_k)$.
We first project the points $X_1,\dots,X_n$ onto $A^\perp$ using the orthogonal projection $\Pi_{A^\perp}$. This gives the points
\begin{align*}
Y_{k+1} \coloneqq \Pi_{A^\perp} (X_{k+1}) \in A^\perp, \qquad  \dots, \qquad
Y_n \coloneqq \Pi_{A^\perp} (X_n) \in A^\perp
\end{align*}
as well as
\begin{align*}
Y \coloneqq \Pi_{A^\perp} (X_{1}) = \dots =  \Pi_{A^\perp} (X_{k}).
\end{align*}
If $F_K$ is a face if $\sP$, then the tangent cone of $\sP$ at $F_K$ is
\begin{align*}
T(F_K,\sP) = \pos (X_1-\bar{X},\dots,X_k-\bar{X},X_{k+1}-\bar{X},\dots, X_n-\bar{X})
\end{align*}
where $\bar{X}=(X_1+\dots+X_k)/k$ is almost surely contained in the relative interior of the simplex $F_K$. Furthermore, the positive hull of $X_1-\bar{X},\dots,X_k-\bar{X}$ is the linear space $A-\bar{X}$, which lets us express said tangent cone as a direct orthogonal sum,
\begin{align*}
T(F_K,\sP) = (A-\bar{X}) \oplus \pos(Y_{k+1}-Y, \dots, Y_n-Y).
\end{align*}
Note that if $F_K$ is not a face of $\sP$, this identity is still true if we define $T(F_K,\sP) = \R^d$ since in this case we have $\pos(Y_{k+1}-Y, \dots, Y_n-Y) = A^\perp$ by Proposition~\ref{prop:face_events_as_absorbtion_events_polytopes}.
Embedding all points into $\R^{d-k+1}$ using the isometry $I_{A^\perp}$ leads to the points
\begin{align*}
	Y_{k+1}' = I_{A^\perp}(Y_{k+1}),\qquad  \dots, \qquad Y_{n}'=I_{A^\perp}(Y_n), \qquad Y' = I_{A^\perp} (Y).
\end{align*}
Their joint distribution is given by Proposition~\ref{prop:beta_cones_joint_distr_projections} (taking into account that, in our case, $A$ is generated by $k$ points, while in Proposition~\ref{prop:beta_cones_joint_distr_projections} it is generated by $k+1$ points), namely,
\begin{itemize}
\item[(i)] $Y', Y_{k+1}',\dots,Y_{n}'$ are independent points in $\R^{d-k+1}$;
\item[(ii)] $Y_{i}'$ has density $f_{d-k+1, \b_{i} + \frac{k-1}{2}}$ for each $i \in \{k+1,\dots,n\}$;
\item[(iii)] $Y'$ has density $f_{d-k+1,\beta_1 + \ldots + \beta_k + \frac{(k-1)(d+1)}{2}}$ .
\end{itemize}
By definition, $\pos(Y_{k+1}'-Y', \dots, Y_n'-Y')$ is a beta cone, which gives the statement.
\end{proof}

Now we can express various quantities related to beta polytopes through the $\Theta$-function.

\begin{theorem}[Expected angles and face probabilities of beta polytopes]\label{theo:probabily_that_F_is_face}
For $k\in \{1,\ldots, \min (d,n-1)\}$ and  $K = \{i_1,\dots,i_k\}  \subseteq \{1,\dots,n\}$, let $F_K \coloneqq [X_{i_1},\dots,X_{i_k}]$ be a possible face of $\sP$. Then,
\begin{align}
	\P [F_K \text{ is a face of } \sP] &=
	2 \sum_{\substack{I \subseteq K^c \\ \# I \in \{d-k,d-k-2,\dots\}}} \Theta \left(\sum_{i\in K }\g_i;\{\g_i: i \in I\};\{\g_i: i \in K^c\bsl I\}\right),  \label{eq:F_K_face_beta_polytope}\\
	\P [F_K \text{ is not a face of } \sP] &=
	2\sum_{\substack{I \subseteq K^c \\ \# I \in \{d-k+2,d-k+4,\dots\}}} \Theta \left(\sum_{i\in K }\g_i;\{\g_i: i \in I\};\{\g_i: i \in K^c\bsl I\}\right). \label{eq:F_K_face_beta_polytope_not}
	\end{align}
	Further, for $\ell \in \lbrace k-1, \dots, \min (d,n-1)-1 \rbrace$, we have
	\begin{align}
		\E [\upsilon_{\ell} (T(F_K,\sP))] = \sum_{\substack{I \subseteq K^c \\ \# I=\ell-k+1}}\Theta \left( \sum_{i \in K } \g_i;\left\{ \g_i : i\in I\right\};\left\{ \g_i : i\in K^c\bsl I\right\}\right). \label{eq:exp_ups_tangent_beta_poly}
	\end{align}
	Additionally, we can express the expected solid angle of the tangent cone of $\sP$ at $F_K$ as
	\begin{align*}
		\E [\alpha(T(F_K,\sP))]
&= \E [\upsilon_{\min (d,n-1)} (T(F_K,\sP))]\\
&=\sum_{\substack{I \subseteq K^c \\ \# I\geq\min(n-k,d-k+1)}} \Theta \left( \sum_{i \in K} \g_i;\left\{ \g_i : i\in I\right\};\left\{ \g_i : i\in K^c\bsl I\right\}\right).
	\end{align*}
Assuming that $n > d+1$, the expected internal angle  on the event that $F_K$ is a face of $\sP$ is given by
	\begin{align}
		\E \left[\a (T(F_K,\sP))\ind_{\{F_K \text{ is a face of }\sP\}}\right]
		=\sum_{\substack{I \subseteq K^c \\ \# I\leq d-k}}  (-1)^{d-k-\# I}
		\Theta \left( \sum_{i \in K} \g_i;\left\{ \g_i : i\in I\right\};\left\{ \g_i : i\in K^c\bsl I\right\}\right).\label{eq:exp_alpha_tangent_beta_poly_on_event}
	\end{align}
The expected solid angle of the normal cone of $\sP$ at $F_K$ on the event that $F_K$ is a face of $\sP$ is given by
\begin{align*}
\E \left[\alpha(N(F_K,\sP))\ind_{\{F_K \text{ is a face of } \sP\}}\right] = \Theta \left(\sum_{i \in K} \g_i; \varnothing;\lbrace \g_{i}: i \in K^c \rbrace\right).
\end{align*}
The expected external angle is given by
\begin{align*}
	\E \left[\alpha(N(F_K,\sP))\right] = \Theta \left(\sum_{i \in K} \g_i; \varnothing;\lbrace \g_{i}: i \in K^c \rbrace\right) + \P [F_K \text{ is not a face of } \sP].
\end{align*}
\end{theorem}
	\begin{proof}
	 We know from Theorem \ref{theo:tangent_cone_is_beta_cone} that $T(F_K,\sP)$ is isometric to
		\begin{align*}
			\R^{k-1} \oplus \BetaCone\left(\R^{d-k+1};\sum_{i \in K} \left(\b_i + \frac{d}{2}\right) -\frac{d-k+1}{2};\left\{ \b_{i}+\frac{d}{2}-\frac{d-k+1}{2}: i \in K^c\right\}\right).
		\end{align*}
We denote the beta cone on the right-hand side with $\mathcal{D}_K$ (which is not the same cone as the one we used in the proof of Theorem~\ref{theo:beta_cones_angles_as_A_B}).  To make the application of Theorem~\ref{theo:beta_cones_conic_intrinsic_vol_as_A_B} more straightforward, let us denote the parameters of $\mathcal{D}_K$ as $d' \coloneqq d-k+1$, as well as
		\begin{align*}
			\b' \coloneqq  \sum_{i \in K} \left( \b_i + \frac{d}{2}\right) -\frac{d-k+1}{2} \qquad \text{and} \qquad  \b_i' \coloneqq \b_{i}+\frac{d}{2}-\frac{d-k+1}{2}, \qquad \text{for } i \in K^c.
		\end{align*}
		Since $\mathcal{D}_K$ is a beta cone in dimension $d-k+1$, we then set
		\begin{align*}
			\g' \coloneqq \b' + \frac{d-k+1}{2} = \sum_{i \in K} \g_i \qquad \text{and} \qquad \g_i' \coloneqq \b_i' + \frac{d-k+1}{2} = \g_i, \qquad \text{for } i \in K^c.
		\end{align*}
We know from Proposition~\ref{prop:face_events_as_absorbtion_events_cones} that $F_K$ is a face of $\sP$ if and only if $\mathcal{D}_K\neq \R^{d-k+1}$.
		We can apply Theorem~\ref{theo:beta_cones_conic_intrinsic_vol_as_A_B} on $\mathcal{D}_K \sim \BetaCone\left(\R^{d'}; \b'; \left\{ \b'_{i}: i \in K^c\right\}\right)$ to get
		\begin{align*}
			\P [F_K \text{ is a face of } \sP] &= \P[\mathcal{D}_K \neq \R^{d-k+1}]
\\
			& = 2 \sum_{\substack{I \subseteq K^c \\ \# I \in \{d-k,d-k-2,\dots\}}} \Theta \left(\sum_{i\in K }\g_i;\{\g_i: i \in I\};\{\g_i: i \in K^c\bsl I\}\right).
		\end{align*}
	In both cases of $d-k$ being either even or odd, we can use Proposition~\ref{prop:relation_theta_even_odd} for the complementary event,
	\begin{align*}
		\P [F_K \text{ is not a face of } \sP] & = 1 - \P [F_K \text{ is a face of } \sP] \\
		& =1- 2 \sum_{\substack{I \subseteq K^c \\ \# I \in \{d-k,d-k-2,\dots\}}} \Theta \left(\sum_{i\in K }\g_i;\{\g_i: i \in I\};\{\g_i: i \in K^c\bsl I\}\right)
		\\
		& = 2\sum_{\substack{I \subseteq K^c \\ \# I \in \{d-k+2,d-k+4,\dots\}}} \Theta \left(\sum_{i\in K }\g_i;\{\g_i: i \in I\};\{\g_i: i \in K^c\bsl I\}\right).
	\end{align*}
		The next steps are analogous to the proof of Theorem~\ref{theo:beta_cones_angles_as_A_B}. First, we have for $\ell \in \{k-1,\ldots, \min (d,n-1)\}$,
		\begin{align*}
			\E [\upsilon_{\ell} (T(F_K,\sP))] =  \E [\upsilon_{\ell} (\R^{k-1} \oplus \mathcal{D}_K)] = \E [\upsilon_{\ell -k+1}(\mathcal{D}_K)].
		\end{align*}
		If additionally $\ell \neq \min (d,n-1)$, we can apply Theorems~\ref{theo:beta_cones_conic_intrinsic_vol_as_A_B} and~\ref{theo:beta_cones_conic_intrinsic_vol_as_A_B_non_full_dim} to get
		\begin{align}\label{eq:upsilon_tangent_beta_polytope}
			\E [\upsilon_{\ell} (T(F_K,\sP))] = \E [\upsilon_{\ell -k+1}(\mathcal{D}_K)] = \sum_{\substack{I \subseteq K^c \\ \# I=\ell-k+1}}\Theta \left( \sum_{i \in K } \g_i;\left\{ \g_i : i\in I\right\};\left\{ \g_i : i\in K^c\bsl I\right\}\right).
		\end{align}
Applying~\eqref{eq:expect_upsilon_d_beta_cone_1}, analogously to the proof of Theorem~\ref{theo:beta_cones_angles_as_A_B}, we get for the tangent cone $T(F_K,\sP)$ that
		\begin{align*}
			\E [\alpha(T(F_K,\sP))]
&=
\E [\upsilon_{\min(n-k,d-k+1)}(\mathcal{D}_K)]
\\
&=
\sum_{\substack{I \subseteq K^c \\ \# I\geq\min(n-k,d-k+1)}} \Theta \left( \sum_{i \in K} \g_i;\left\{ \g_i : i\in I\right\};\left\{ \g_i : i\in K^c\bsl I\right\}\right).
		\end{align*}
Assuming that $n>d+1$ and observing that $F_K$ is a face of $\sP$ if and only if $\mathcal{D}_K\neq \R^{d-k+1}$, we apply~\eqref{eq:upislon_R_d_ind_BetaCone} of Theorem~\ref{theo:beta_cones_conic_intrinsic_vol_as_A_B} to get
		\begin{align*}
\E \left[\a (T(F_K,\sP))\ind_{\{F_K \text{ is a face of }\sP\}}\right]
&=
\E \left[\upsilon_{d-k+1}(\mathcal{D}_K)\ind_{\{\mathcal{D}_K\neq \R^{d-k+1}\}}\right]
\\
&=
\sum_{\substack{I \subseteq K^c \\ \# I\leq d-k}}  (-1)^{d-k-\# I}
\Theta \left( \sum_{i \in K} \g_i;\left\{ \g_i : i\in I\right\};\left\{ \g_i : i\in K^c\bsl I\right\}\right).
\end{align*}
Next, for the normal cone, we use~\eqref{eq:upsilon_tangent_beta_polytope} with $\ell = k-1$ to get
\begin{align*}
\E \left[\alpha(N(F_K,\sP))\ind_{\{F_K \text{ is a face of } \sP\}}\right]
= \E [\upsilon_{0}(\mathcal{D}_K)]=
\Theta \left(\sum_{i \in K} \g_i; \varnothing;\lbrace \g_{i}: i \in K^c \rbrace\right).
\end{align*}	
Lastly, a formula for $\E[\alpha(N(F_K,\sP))]$ follows directly  observing  that $\alpha(N(F_K,\sP)) = 1$ if $F_K$ is  not a face of $\sP$.
This completes the proof.
\end{proof}
\begin{remark}[Comparison of formulas for face probabilities]
Which of the equations, \eqref{eq:F_K_face_beta_polytope} or~\eqref{eq:F_K_face_beta_polytope_not}, is more efficient for computing the face probabilities, depends on the parameters $n,d,k$. If $k = d-p$ with ``small'' $p\in\N$, then Equation~\eqref{eq:F_K_face_beta_polytope} is efficient since it involves summation over $I\subseteq K^c$ with $\# I \in \{p,p-2,\ldots\}$. In particular, if $k=d$ and $K = \lbrace i_1, \dots, i_d \rbrace \subseteq \{1,\ldots,n\}$, then~\eqref{eq:F_K_face_beta_polytope}  simplifies to
\begin{align*}
\P [F_K \text{ is a facet of } \sP] = 2 \cdot \Theta \left(\sum_{i \in K} \g_i ;\varnothing; \lbrace \g_i: i \in K^c  \rbrace \right).
\end{align*}
On the other hand, if $n = d + q$ with ``small'' $q\in \N$, while $k\in \{1,\ldots, d\}$ is arbitrary,  then~\eqref{eq:F_K_face_beta_polytope_not} becomes efficient since the size of $I \subseteq  K^c$ satisfies $\#I\in \{d-k+2,d-k+4,\ldots,\}$ and $\# I \leq d-k+q$.   In particular, if $n= d+2$ and $k\in \{1,\ldots, d\}$ is arbitrary, then $I= K^c$ and~\eqref{eq:F_K_face_beta_polytope_not} becomes
$$
\P [F_K \text{ is a not a face of } \sP] =  2 \cdot \Theta \left(\sum_{i\in K }\g_i;\{\g_i: i \in K^c\}; \varnothing\right).
$$
Further, if $n=d+3$ and  $k\in \{1,\ldots, d\}$ is arbitrary,  then $\#I = d-k+2$ meaning that $I = K^c\bsl \{j\}$ for some $j\in K^c$, and~\eqref{eq:F_K_face_beta_polytope_not} becomes
$$
\P [F_K \text{ is not a face of } \sP] =
	2 \cdot  \sum_{j\in K^c} \Theta \left(\sum_{i\in K}\g_i;\{\g_i: i \in K^c \bsl\{j\}\};\{\g_j\}\right).
$$
\end{remark}

\begin{theorem}[Expected $f$-vector of beta polytopes]\label{theo:f_vect}
	For $k\in \{1,\ldots, \min (d,n-1)\}$, the expected number of $(k-1)$-dimensional faces of $\sP$ is given by
\begin{align}
\E f_{k-1}(\sP_{n,d}^{\beta_1, \dots, \b_n})
&=2\sum_{\substack{K \subseteq \{1,\dots,n\} \\ \# K =k }} \sum_{\substack{ J \supseteq K \\ \# J = \{d,d-2,\dots\}}} \Theta \left(\sum_{i\in K}\g_i;\{\g_i: i \in J \bsl K\};\{\g_i: i \in J^c\}\right) \label{eq:f_vector},
\\
\binom nk - \E f_{k-1}(\sP_{n,d}^{\beta_1, \dots, \b_n})
&=
2\sum_{\substack{K \subseteq \{1,\dots,n\} \\ \# K =k }} \sum_{\substack{ J \supseteq K \\ \# J = \{d+2,d+4,\dots\}}} \Theta \left(\sum_{i\in K}\g_i;\{\g_i: i \in J \bsl K\};\{\g_i: i \in J^c\}\right). \label{eq:f_vector_loosers}
	\end{align}
\end{theorem}
\begin{proof}
We can write the expected number of $(k-1)$-faces as
\begin{align*}
\E f_{k-1}(\sP_{n,d}^{\beta_1, \dots, \b_n})  = \sum_{\substack{K \subseteq \{1,\dots,n\} \\ \# K=k}} \P [F_K\text{ is a face}],
\end{align*}
where $F_K = \conv(X_i : i\in K)$ is a possible face of $\sP = \sP_{n,d}^{\beta_1, \dots, \b_n}$. We  use Theorem~\ref{theo:probabily_that_F_is_face} to get
	\begin{align}\label{eq:theo:f_vect_proof}
		\E f_{k-1}(\sP)
		= 2\sum_{\substack{K \subseteq \{1,\dots,n\} \\ \# K=k}} \sum_{\substack{I \subseteq K^c \\ \# I \in \{d-k,d-k-2,\dots\}}} \Theta \left(\sum_{i\in K }\g_i;\{\g_i: i \in I\};\{\g_i: i \in K^c \bsl  I\}\right).
	\end{align}
	Setting $J \coloneqq I \cup K$ lets us write the sum in terms of $J$ and yields~\eqref{eq:f_vector}.
The complementary formula~\eqref{eq:f_vector_loosers} follows from Proposition~\ref{prop:relation_theta_even_odd}.
\end{proof}
	\begin{example}[Expected angles of beta simplices: The equal beta's case]\label{ex:simplex_setting}
%		Consider the following setting.
Consider a random simplex $\sP_{n,n-1}^{x,\ldots, x}=[X_1,\dots,X_n]$ in $\R^{n-1}$ spanned by independent random points $X_1,\ldots, X_n \sim f_{n-1,x}$, where $x\geq -1$ is a parameter, the same for all points.
In~\cite{kabluchko_thaele_zaporozhets_beta_polytopes_poi_polyhedra} and~\cite{kabluchko_angles_of_random_simplices_and_face_numbers}, the following notation for the expected internal and external angles of $[X_1,\ldots, X_n]$ has been used:		
\begin{align*}
&\E \a (T([X_1,\dots,X_k],[X_1,\dots,X_n])) = J_{n,k}(x),
\\
&\E \a (N([X_1,\dots,X_k],[X_1,\dots,X_n])) = I_{n,k} (2x+n-1),
\end{align*}
for $k\in \{1,\ldots, n\}$. The corresponding angle sums were denoted by $\bI_{n,k}(x) = \binom nk I_{n,k}(x)$ and $\bJ_{n,k}(x) = \binom nk J_{n,k}(x)$.  An explicit formula for $I_{n,k}(x)$ has been obtained in~\cite[Theorem~1.2]{kabluchko_thaele_zaporozhets_beta_polytopes_poi_polyhedra} (see also Theorem~1.6 there for the definition we use), while a formula for $J_{n,k}(x)$ has been obtained in~\cite[Section~1.3]{kabluchko_angles_of_random_simplices_and_face_numbers} after some partial results in~\cite{kabluchko_angle_sums_dim_3_4,kabluchko_poisson_zero_cell,kabluchko_recursive_scheme}.
We shall now express the quantities $J_{n,k}(x)$ and $I_{n,k}(x)$ through the $\Theta$-function and other quantities introduced in the present paper. For $J_{n,k}(x)$, we use the formula for $\E [\alpha(T(F_K,\sP))]$ in Theorem~\ref{theo:probabily_that_F_is_face},
		\begin{align*}
			J_{n,k}(x) %= \E \b ([X_1,\dots,X_k],[X_1,\dots,X_n])
			=& \Theta \biggl(k\left(x+\frac{n-1}{2}\right);\left\{x+\frac{n-1}{2},\dots,x+\frac{n-1}{2}\right\};\varnothing\biggr) \\
			=& \intern \biggl( kx + \frac{(k-1)n}{2} ; x +\frac{k-1}{2}, \dots, x +\frac{k-1}{2}   \biggr) \\
			=& A\biggl(k\left(x+\frac{n-1}{2}\right); \underbrace{x+\frac{n-1}{2},\dots,x+\frac{n-1}{2}}_{n-k}\biggr),
		\end{align*}
		where we have also used Proposition~\ref{prop:int_ext_through_theta} and \eqref{eq:A_through_Int}.   To identify $I_{n,k}(2x+n-1)$, we use the formula for $\E [\alpha(N(F_K,\sP))]$ in Theorem~\ref{theo:probabily_that_F_is_face} combined with
%compare its integral representation to Corollary \ref{cor:external_beta_formula}
Proposition~\ref{prop:int_ext_through_theta} and~\eqref{eq:B_through_Ext}, which gives
		\begin{align*}
		I_{n,k}(2x+n-1)
		&= \Theta \biggl(k \left(x+\frac{n-1}{2}\right); \varnothing; \left\{x+\frac{n-1}{2},\dots,x+\frac{n-1}{2}\right\}\biggr)\\
		&= \extern \biggl( kx+\frac{(k-1)n}{2} ; x +\frac{k-1}{2}, \dots, x +\frac{k-1}{2}  \biggr)  \\
		&= B\biggl(n\left(x+\frac{n-1}{2}\right);\underbrace{x+\frac{n-1}{2}, \dots, x+\frac{n-1}{2}}_{n-k}\biggr).
		\end{align*}
The results of~\cite{kabluchko_thaele_zaporozhets_beta_polytopes_poi_polyhedra,kabluchko_recursive_scheme,kabluchko_angle_sums_dim_3_4,kabluchko_angles_of_random_simplices_and_face_numbers} (where beta polytopes generated by beta points with equal parameters have been considered) can now be recovered from the results of the present paper. 
\end{example}

\subsection{Expected volume and beta content}
If $P\subseteq \BB^d$ is a polytope, we define its \emph{beta content} with parameter $\beta \geq -1$ as the probability that a beta-distributed point $X\sim f_{d,\beta}$ falls into $P$, that is
$$
\Content_\beta(P) := \P[X\in P] = \int_P f_{d,\beta}(x) \dint x.
$$
For example, the $0$-content is related to the volume via $\Vol_d(P) = \kappa_d \cdot \Content_0(P)$.
		
	\begin{theorem}[Expected beta-content of beta polytopes]\label{theo:beta_content_polys}
		Let $d\geq 2$, $n \geq d+1$ and $\beta,\beta_1,\ldots, \beta_n\geq -1$ and put  $\gamma := \beta + \frac d2$. Then,
		\begin{align*}
		\E \Content_\beta(\sP_{n,d}^{\beta_1,\ldots,\beta_n})
		&=
		1 - 2 \sum_{\substack{I \subseteq \{1,\dots,n\} \\ \# I \in \{d-1,d-3,\dots\}}} \Theta \left(\gamma ;\{\g_i: i \in I\};\{\g_i: i \in I^c\}\right)\\
		&=
		2 \sum_{\substack{I \subseteq \{1,\dots,n\} \\ \# I \in \{d+1,d+3,\dots\}}} \Theta \left(\gamma ;\{\g_i: i \in I\};\{\g_i: i \in I^c\}\right).
		\end{align*}
		\end{theorem}
		\begin{proof}
Let $X_{n+1} \sim f_{d,\beta}$ be an additional random point which is independent of the points $X_1,\ldots, X_n$ that span $\sP_{n,d}^{\beta_1,\ldots,\beta_n}$.  On the one hand,
		$$
		\E \Content_\beta(\sP_{n,d}^{\beta_1,\ldots,\beta_n}) = \P\left[X_{n+1} \in \sP_{n,d}^{\beta_1,\ldots,\beta_n}\right] = 1 - \P\left[X_{n+1} \text{ is a vertex of } \sP_{n+1,d}^{\beta_1,\ldots,\beta_n,\b}\right]
		$$
since the additional point $X_{n+1}$ takes its value inside of $\sP_{n,d}^{\beta_1,\ldots,\beta_n}$ if and only if it is not a vertex of the polytope $\sP_{n,d}^{\beta_1,\ldots,\beta_n,\beta}$ spanned by the points $X_1,\ldots, X_n, X_{n+1}$.
On the other hand, we know from Theorem~\ref{theo:probabily_that_F_is_face} applied to the polytope $\sP_{n,d}^{\beta_1,\ldots,\beta_n,\b}$ that
		\begin{align*}
		\P \left[X_{n+1} \text{ is a vertex of } \sP_{n+1,d}^{\beta_1,\ldots,\beta_n,\b}\right] =
		2 \sum_{\substack{I \subseteq \{1,\dots,n\} \\ \# I \in \{d-1,d-3,\dots\}}} \Theta \left(\gamma ;\{\g_i: i \in I\};\{\g_i: i \in I^c\}\right).
		\end{align*}
Here is another proof of the theorem. We observe that $X_{n+1}$ belongs to the interior of $\sP_{n,d}^{\beta_1,\ldots,\beta_n}$ if and only if the positive hull of the vectors $X_{1}-X_{n+1},\ldots, X_{1}-X_{n+1}$ equals $\R^d$. This leads to
\begin{align*}
\E \Content_\beta(\sP_{n,d}^{\beta_1,\ldots,\beta_n})
=
\P\left[X_{n+1} \in \sP_{n,d}^{\beta_1,\ldots,\beta_n}\right]
&=
\P[\BetaCone(\R^d; \beta; \beta_1,\ldots, \beta_n) = \R^d]
\\
&=
2\sum_{\substack{I \subseteq \{1,\dots,n\} \\ \# I \in \lbrace d+1,d+3,\ldots \rbrace}}\Theta ( \gamma;\left\{ \g_i : i\in I\right\};\left\{ \g_i : i\in I^c\right\}),
\end{align*}
where in the last step we used Equation~\eqref{eq:theo:beta_cones_conic_intrinsic_vol_as_A_B_1} of Theorem~\ref{theo:beta_cones_conic_intrinsic_vol_as_A_B}.
\end{proof}
\begin{example}[Expected $\beta$-content of a simplex]
Taking $n=d+1$ in Theorem~\ref{theo:beta_content_polys} gives
$$
\E \Content_\beta(\sP_{d+1,d}^{\beta_1,\ldots,\beta_{d+1}})
=
2 \cdot \Theta \left(\gamma ;\{\g_1,\ldots, \g_{d+1}\};\varnothing\right).
$$
\end{example}
\begin{theorem}[Expected volume of beta polytopes]\label{theo:exp_volume_beta_polytope}
For arbitrary $d\geq 2$,  $n \geq d+1$ and $\beta_1,\dots, \beta_n\geq -1$, we have
\begin{align}
\E \Vol(\sP_{n,d}^{\beta_1,\ldots,\beta_n})
&=
2\kappa_d \sum_{\substack{I \subseteq \{1,\ldots, n\} \\ \# I=d+1}}  \Theta (d/2; \{\gamma_i: i \in I\}; \{\gamma_i: i \in I^c\})
\label{eq:E_Vol_beta_poly_1}\\
&=
\k_d - 2\k_d \sum_{\substack{I \subseteq \{1,\dots,n\} \\ \# I \in \{d-1,d-3,\dots\}}} \Theta \left(d/2;\{\g_i: i \in I\};\{\g_i: i \in I^c\}\right).
\label{eq:E_Vol_beta_poly_2}
\end{align}
\end{theorem}
\begin{proof}
Taking $\beta= 0$ in Theorem~\ref{theo:beta_content_polys} and applying the identity $\Vol_d(P) = \kappa_d \cdot \Content_0(P)$ yields~\eqref{eq:E_Vol_beta_poly_2} and
$$
\E \Vol(\sP_{n,d}^{\beta_1,\ldots,\beta_n})
=
2\kappa_d \sum_{\substack{I \subseteq \{1,\ldots, n\} \\ \# I\in \{d+1,d+3,\ldots\}}}  \Theta (d/2; \{\gamma_i: i \in I\}; \{\gamma_i: i \in I^c\}).
$$
Since $\Theta (d/2;\{\g_i: i \in I\};\{\g_i: i \in I^c\})=0$  for $\# I \in \{d+3, d+5,\dots\}$ by Corollary~\ref{cor:theta_quant_vanishes_d/2-1}, all terms with $\# I  \neq d+1$ can be omitted, which yields~\eqref{eq:E_Vol_beta_poly_1}.
\end{proof}
\begin{example}[Expected volume of a beta simplex]\label{ex:vol_beta_simplex_as_theta}
Taking $n=d+1$ in Theorem~\ref{theo:exp_volume_beta_polytope} we get
\begin{align*}
\E \Vol(\sP_{d+1,d}^{\beta_1,\ldots,\beta_{d+1}})
&= 2\k_d \cdot \Theta \left(\frac{d}{2};\left\{\b_1 + \frac{d}{2},\dots,\b_{d+1}+\frac{d}{2}\right\};\varnothing\right)
\\
&= 2 \k_d \cdot \intern \left(-\frac{1}{2}; \b_1 - \frac{1}{2}, \dots, \b_{d+1}-\frac{1}{2}\right),
\end{align*}
where the second identity follows from Proposition~\ref{prop:int_ext_through_theta}.
\end{example}

\begin{remark}[Various formulas for the expected volume]\label{rem:volume_as_theta_from_book}
	In Theorem 8.35 in \cite{kabluchko_steigenberger_thaele} as well as Theorem~2.1 in \cite{moseeva2024mixedrandombetapolytopes}, the authors have derived another formula for $\E \Vol_d \sP_{n,d}^{\beta_1,\dots,\b_n}$,
	\begin{align}\label{eq:expected_vol_beta_poly_different_formula}
		\E \Vol_d\left(\sP_{n,d}^{\beta_1,\dots,\b_n}\right)= %\frac{\kappa_d}{(d+1)\kappa_{d+1}}
		\sum_{\substack{I \subseteq \{1,\ldots, n\} \\ \# I=d+1}}
		H_{d+1}^I
		\int_{-1}^{+1} f_{1,\sum_{i \in I} \b_{i} + \frac{(d+1)^2}2 - 1} (h)
		\prod_{j \in I^c} F_{1, \beta_{j} +\frac{d-1}{2}} (h) \, \dint h,
		\end{align}
		where $F_{1,\b}(h):= c_{1,\b} \int_{-1}^h ( 1- x^2 )^\b \,\dint x$ for  $h\in [-1,1]$ and, for any index set $I = \lbrace i_1, \dots, i_{d+1} \rbrace$,
		\begin{align*}
		H_{d+1}^{I}
		&\coloneqq
		\frac{\kappa_d}{2^d} \cdot \binom{\sum_{i \in I} 2\b_{i} + (d+1)^2}{\sum_{i \in I} \b_{i} + \frac{(d+1)^2}2}^{-1}  \prod_{i \in I} \binom{d+1+2\b_i}{\frac{d+1}{2}+\b_i}
=\E \Vol_d \left(\sP_{d+1,d}^{\b_{i_1},\dots,\b_{i_{d+1}}}\right).
		\end{align*}
The last equality has been observed in Remark 8.38 in~\cite{kabluchko_steigenberger_thaele}. Alternatively, it follows by taking $n=d+1$ in~\eqref{eq:expected_vol_beta_poly_different_formula}. It follows from Example~\ref{ex:vol_beta_simplex_as_theta} that
$$
H_{d+1}^{I} = \E \Vol_d \left(\sP_{d+1,d}^{\b_{i_1},\dots,\b_{i_{d+1}}}\right) = 2\kappa_d \cdot \Theta (d/2; \{\gamma_i: i\in I\}; \varnothing).
$$
The integral in~\eqref{eq:expected_vol_beta_poly_different_formula} can also be written in terms of the $\Theta$-function using the integral representation  given in Theorem~\ref{theo:theta_as_prod_of_thetas}, which allows us to rewrite~\eqref{eq:expected_vol_beta_poly_different_formula}  as follows:
\begin{align*}
\E \Vol_d\left(\sP_{n,d}^{\beta_1,\dots,\b_n}\right)=
2\kappa_d \cdot \sum_{\substack{I \subseteq \{1,\ldots, n\} \\ \# I=d+1}}  \Theta \left(\frac d2; \{\gamma_i: i\in I\}; \varnothing\right) \cdot  \Theta \left(\frac d2+\sum_{i \in I}\gamma_i;\varnothing;\left\{  \gamma_i: i \in \{1,\ldots, n\} \backslash I \right\}\right).
\end{align*}
Using Equation~\eqref{eq:theo:theta_as_prod_of_thetas} from Theorem~\ref{theo:theta_as_prod_of_thetas},  the product of two $\Theta$-terms on the right-hand side can be written as $\Theta (d/2; \{\gamma_i: i \in I\}; \{\gamma_i: i \in I^c\})$, and we finally  recover~\eqref{eq:E_Vol_beta_poly_1}.
\end{remark}

In Theorem~6.1 of \cite{kabluchko_steigenberger_thaele}, the following formula has been derived for the volume of the $d$-dimensional beta simplex,
	\begin{align}\label{eq:volume_simplex_as_product_of_Gammas}
		\E \Vol(\sP_{d+1,d}^{\beta_1,\ldots,\beta_{d+1}})
		=(d!)^{-1} \frac{\Gamma \left( \sum_{i=1}^{d+1}  \beta_i +  \frac{d(d+1)}{2}  +1 +\frac{d+1}{2} \right) }{\Gamma \left( \sum_{i=1}^{d+1}  \beta_i + \frac{d(d+1)}{2}  +1 +\frac{d}{2} \right)} \frac{\Gamma \left( \frac{d+1 }{2} \right)}{\Gamma \left( \frac{1}{2} \right)}
		\prod_{i=1}^{d+1} \frac{\Gamma \left( \frac{d}{2} + \beta_i +1 \right)}{\Gamma \left( \frac{d}{2} + \beta_i  +1 +\frac{1}{2} \right)}.
	\end{align}
It can also be found in~\cite{ruben_miles}. We can use this representation to prove the following theorem.
\begin{theorem}[A special $a$-quantity]
For $d\in \N_0$, real numbers $\alpha_1,\ldots, \alpha_{d+1} \geq  0$, we have
\begin{align}
a\left(d+2+\sum_{i=1}^{d+1}\a_i;\a_1,\dots,\a_{d+1}\right)=\pi \prod_{i=1}^{d+1} \frac{1}{\a_i+1}.  \label{eq:a_special_val_d+2}
\end{align}
\end{theorem}
\begin{proof}
For $d=0$ and $d=1$ the formula is true by Example~\ref{ex:a_def_case_d_0}. Let $d\geq 2$.
	Let us look at the expected volume of $\sP_{d+1,d}^{\b_1,\dots,\b_{d+1}}$. On the one hand, we can express it using Equation~\eqref{eq:volume_simplex_as_product_of_Gammas}. On the other hand, a representation is given by Theorem~\ref{theo:exp_volume_beta_polytope} by setting $n=d+1$,
	\begin{multline*}
		\E \Vol(\sP_{d+1,d}^{\beta_1,\ldots,\beta_{d+1}})
		=2\k_d \cdot \Theta \left(\frac{d}{2};\left\{\b_1 + \frac{d}{2},\dots,\b_{d+1}+\frac{d}{2}\right\};\varnothing\right) \\
		= 2 \k_d \cdot c_{\frac{d}{2}+\sum_{i=1}^{d+1}\left(\b_i+\frac{d}{2}\right)} \cdot \prod_{i=1}^{d+1} c_{\b_i+\frac{d}{2}-\frac{1}{2}}  %\\
%		& \qquad \qquad \qquad \qquad \times
\cdot a \left(d+\sum_{i=1}^{d+1}\left(2\b_i+d\right)+2; 2\b_1+d, \dots, 2\b_{d+1}+d\right).
	\end{multline*}
The last expression can be further simplified by computation of the constants. Since
	\begin{align*}
		c_{\frac{d}{2}+\sum_{i=1}^{d+1}\left(\b_i+\frac{d}{2}\right)} = \frac{\Gamma\left(\frac{3}{2}+\frac{d}{2}+\sum_{i=1}^{d+1}\left(\b_i+\frac{d}{2}\right)\right)}{\Gamma\left(1+\frac{d}{2}+\sum_{i=1}^{d+1}\left(\b_i+\frac{d}{2}\right)\right)\sqrt{\pi}}
\qquad \text{and} \qquad
		\prod_{i=1}^{d+1} c_{\b_i+\frac{d}{2}-\frac{1}{2}}  = \prod_{i=1}^{d+1} \frac{\Gamma\left(\b_i + \frac{d}{2} + 1\right)}{\Gamma\left(\b_i + \frac{d}{2} + \frac{1}{2}\right)\sqrt{\pi}},
\end{align*}
as well as $\k_d = \frac{\pi^{d/2}}{\Gamma\left(\frac{d}{2}+1\right)}$,
	we can express the $a$-quantity as
	\begin{align*}
%		a \left(2 \left(\frac{d}{2}+\left(\sum_{i=1}^{d+1}\b_i+\frac{d}{2}\right)\right)+2; 2\b_1+d, \dots, 2\b_{d+1}+d\right)
a \left(d+\sum_{i=1}^{d+1}\left(2\b_i+d\right)+2; 2\b_1+d, \dots, 2\b_{d+1}+d\right)
		= \frac{\sqrt{\pi}}{2}   \frac{\Gamma\left(\frac{d+2}{2}\right)\Gamma \left( \frac{d+1}{2} \right)}{\Gamma \left( d+1\right)}  \prod_{i=1}^{d+1} \frac{\Gamma \left( \frac{d}{2} + \beta_i + \frac{1}{2} \right)}{\Gamma \left( \frac{d}{2} + \beta_i  +\frac{1}{2} +1 \right)}.
	\end{align*}
	Using the Legendre duplication formula, that is
$\Gamma\left(d+1\right) = \frac{2^{d}}{\sqrt{\pi}} \Gamma\left(\frac{d+1}{2}\right)\Gamma\left(\frac{d+2}{2}\right)
$,
	%\end{align*}
as well as the fact that $\Gamma(z)/\Gamma(z+1) = 1/z$ %$\frac{\Gamma\left(z+\frac{1}{2}\right)}{\Gamma\left(z+\frac{3}{2}\right)}=\frac{2}{2z+1}$
we arrive at
	\begin{align*}
		a \left(d+2+\sum_{i=1}^{d+1}\left(2\b_i+d\right); 2\b_1+d, \dots, 2\b_{d+1}+d\right) 
		= \pi \cdot \prod_{i=1}^{d+1} \frac{1}{2\b_i+d+1}.
	\end{align*}
Setting $\a_i \coloneqq 2\b_i+d$ for each $i \in \lbrace 1, \dots, d+1\rbrace$ gives the claim under additional restrictions $\alpha_i \geq d-2$. To remove these, observe that both sides of~\eqref{eq:a_special_val_d+2} define analytic functions of complex variables $\alpha_1,\ldots,\alpha_{d+1}$ on the domain $\Re \alpha_1,\ldots, \Re \alpha_{d+1} > 0$ (which can be continuously extended to its closure). Their difference vanishes for real $\alpha_1,\ldots,\alpha_{d+1}\geq d-2$. By a uniqueness theorem from~\cite[p.~21]{shabat_book_2} this is sufficient to conclude that the difference vanishes identically.
\end{proof}

\begin{remark}[Limit cases and monotonicity]
Recall that for $\beta = -1$ the beta distribution is uniform on the unit sphere and hence $\Content_{-1}(\sP_{n,d}^{\beta_1,\ldots,\beta_n}) = 0$. (This is related to Proposition~\ref{prop:a_quant_special_val_d} and Corollary~\ref{cor:theta_quant_vanishes_d/2-1}.)
On the other extreme, if $\beta \to \infty$, then $X_{n+1}\sim f_{d,\beta}$ converges to $0$ probability and hence
$$
\lim_{\beta \to +\infty} \E \Content_\beta(\sP_{n,d}^{\beta_1,\ldots,\beta_n}) = \P\left[0 \in \sP_{n,d}^{\beta_1,\ldots,\beta_n}\right]
=
\frac 1  {2^{n-1}} \sum_{r=d}^{n-1} \binom{n-1}{r}.
$$
where in the last step we used Wendel's formula; see~\cite{Wendel} and~\cite[Theorem~8.2.1]{schneider_weil_book}. (This is related to Theorem~\ref{theo:wendel_cones_angles}.)
It seems that with increasing $\beta$ it becomes less probable that $X_{n+1}$ is a vertex, which leads to the following
\end{remark}
\begin{conjecture}
For fixed $\beta_1,\ldots,\beta_n \geq -1$,  the expected beta-content function $\beta \mapsto \E \Content_\beta(\sP_{n,d}^{\beta_1,\ldots,\beta_n})$ is increasing in $\beta$.
\end{conjecture}

\subsection{Sylvester's problem}
Next we use Theorem~\ref{theo:f_vect} to solve an analogue of Sylvester's problem in $\R^d$ for independent  beta-distributed points, thus  generalizing~\cite{gusakova_kabluchko_sylvester_beta}, where the case of equal beta's has been considered.
\begin{theorem}[Sylvester's problem for beta-distributed points]
Let $X_1,\ldots, X_{d+2}$ be independent random points in $\R^d$ such that $X_i$ is beta-distributed with parameter $\beta_i \geq -1$ for every $i\in \{1,\ldots, d+2\}$.  Then,
$$
\P[\,[X_1,\dots,X_{d+2}] \text{ is a simplex }] = 2 \cdot \sum_{j=1}^{d+2} \Theta \left(\g_j ; \{ \gamma_1,\ldots, \widehat {\gamma_j}, \ldots, \gamma_{d+2}\}; \varnothing\right),
$$
where we recall that $\g_i = \b_i +d/2$ for each $i \in \{1,\dots,d+2\}$.
\end{theorem}
\begin{proof}
Let $p$ denote the probability that $\sP:=[X_1,\dots,X_{d+2}]$ is a simplex. The number of vertices of $\sP$ is either $d+1$ (with probability $p$) or $d+2$ (with probability $1-p$). This lets us express the expected number of vertices of $\sP$ as
\begin{align*}
\E f_0 (\sP) = p \cdot (d+1)+ (1-p)\cdot(d+2)=d+2-p,
\end{align*}
which in turn gives us a formula for $p = (d+2) -\E f_0 (\sP)$ using the second equation of Theorem~\ref{theo:f_vect} in which we have $n=d+2$, $k=1$ and $J= \{1,\ldots, d+2\}$. Here is another proof:
\begin{align*}
\P[\,X_j \in [X_1,\ldots, \widehat{X_j},\ldots, X_n]\,]
&=
\P[\BetaCone(\R^d; \beta_j; \beta_1,\ldots, \widehat{\beta}_j, \ldots, \beta_n) = \R^d]
\\
&=
2 \cdot\Theta \left(\g_j ; \{ \gamma_1,\ldots, \widehat {\gamma_j}, \ldots, \gamma_{d+2}\}; \varnothing\right),
\end{align*}
where in the last step we used~\eqref{eq:theo:beta_cones_conic_intrinsic_vol_as_A_B_1} with $n=d+1$. Taking the sum over all $j=1,\ldots, d+2$ gives the claim.
\end{proof}

\subsection{Expected \texorpdfstring{$S$}{S}-functional}\label{subsec:S_functional}
In this section we  introduce a general class of functionals of polytopes that allows to treat in a unified way many natural examples.

\subsubsection{Definition of the \texorpdfstring{$S$}{S}-functional}
We start with some preliminaries.

\vspace*{2mm}
\noindent
\textit{Space of simplices.}
Let $\Simpl(\R^k)$  be the set of $k$-dimensional simplices $[x_1,\ldots, x_{k+1}]$ in $\R^k$. Two simplices differing only by the order of vertices are considered equal. A functional $\varphi: \Simpl(\R^k)\to \R$ is called \emph{rotationally invariant} if $\varphi([O x_1,\ldots, O x_{d+1}]) = \varphi([x_1,\ldots, x_{d+1}])$ for every orthogonal transformation $O:\R^k \to \R^k$ and every simplex $[x_1,\ldots, x_{k+1}] \in \Simpl(\R^k)$. Further, $\varphi$ is said to be  \emph{$\delta$-homogeneous} with $\delta\in \R$ if  $\varphi([c x_1,\ldots, c x_{k+1}])=c^\delta \cdot  \varphi([x_1,\ldots, x_{k+1}])$, for all $c >0$ and all $[x_1,\ldots, x_{k+1}]\in \Simpl(\R^k)$. For example, the $k$-dimensional volume is rotationally invariant and $k$-homogeneous. A functional $\varphi: \Simpl(\R^k)\to \R$ is called \emph{measurable} if the function $\bar \varphi: (x_1,\ldots, x_{k+1}) \mapsto \varphi([x_1,\ldots, x_{k+1}])$ is Borel-measurable. The function $\bar \varphi$ is defined on the (open) set of tuples $(x_1,\ldots, x_{k+1}) \in \R^{k(k+1)}$ with the property that $x_1,\ldots, x_{k+1}$ are affinely independent.

\vspace*{2mm}
\noindent
\textit{Space of cones.}
Let $\PolyCones(\R^d)$ be the set of all polyhedral cones in $\R^d$. Let $\PolyCones_k(\R^d)$ be the set of those polyhedral cones whose lineality space has dimension $k\in \{0,\ldots, d\}$.  A functional $\psi: \PolyCones_k(\R^d) \to \R$ is called rotationally invariant if $\psi(OC) = \psi(C)$ for every cone $C\in \PolyCones_k(\R^d)$ and every orthogonal transformation $O:\R^d\to\R^d$. We endow $\PolyCones_k(\R^d)$ with the Borel-$\sigma$-algebra induced by the Hausdorff distance; see Remark~\ref{rem:rem:hausdorff_dist_cones}.

\begin{definition}[$S$-functional]\label{dfn:generl_T_functional}
Let $P \subseteq \R^d$ be a simplicial polytope, and let $k \in \lbrace 0, \dots, \dim P-1\rbrace$. Consider two measurable, rotationally invariant functionals $\varphi:\Simpl(\R^k) \to [0,\infty]$  and  $\psi: \PolyCones_k(\R^d) \to [0,\infty]$.
The \emph{$S$-functional} of  $P$ is then defined as
\begin{align}\label{eq:definition_S_functional}
S_{k;\varphi,\psi} (P) = \sum_{F \in \cF_k(P)} \varphi(I_{\aff F}(F)) \cdot  \psi(T(F,P)).
\end{align}
We recall that $I_{\aff F}$ is an isometry between $\aff F$ and $\R^{k}$. To simplify the notation, we shall write $\varphi(F):= \varphi(I_{\aff F}(F))$.
\end{definition}
As we shall discuss in Section~\ref{subsec:examples_S_functional}, many natural functionals of polytopes can be represented as special cases of the $S$-functional. Examples include the intrinsic volumes $V_k(P)$, the number of $k$-faces $f_k(P)$, the $k$-volume of the $k$-skeleton $\sum_{F\in \cF_k(P)}\Vol_k(F)$, and various generalized angle sums.

\subsubsection{The expected \texorpdfstring{$S$}{S}-functional}
In the next theorem we compute the expected $S$-functional of a beta polytope.
\begin{theorem}[Expected $S$-functional of a beta polytope]\label{theo:expected_S_functional}
Let $\sP \coloneqq \sP_{n,d}^{\b_1,\dots,\b_n}=[X_1,\dots,X_n]$ be  a beta polytope in $\R^d$, $d\geq 2$, with parameters $\beta_1,\ldots, \beta_n\geq -1$. Let $k\in \{0,\ldots, \min (d, n-1)-1\}$ and consider two measurable rotationally invariant functionals  $\varphi: \Simpl(\R^k) \to [0,\infty]$ and $\psi: \PolyCones_k(\R^d)\to [0,\infty]$. Let $\varphi$ be $\delta$-homogeneous with $\delta\geq 0$. Then, for the expected $S$-functional of $\sP$ it holds that
	\begin{align*}
		\E S_{k;\varphi,\psi} (\sP)
		=& \sum_{\substack{I \subseteq \{1,\ldots,n\} \\ \# I = k+1}} \E \left[\varphi(F_I)\right] \E \left[  \psi\left(\R^{k} \oplus \cD_{\delta,I}\right) \ind_{\{\cD_{\delta,I} \neq \R^{d-k}\}}\right],
	\end{align*}
	where $\cD_{\delta,I} \sim
	\BetaCone\left(\R^{d-k}; \sum_{i \in I} \left(\b_i + \frac{d}{2}\right) - \frac{d-\delta-k}{2}; \left\{\b_i + \frac{k}{2}: i \in I^c\right\}\right)$ and $F_I := \conv(X_i: i\in I)$.
\end{theorem}
\begin{proof}
Let us recall  the notation of the canonical decomposition, Theorem~\ref{theo:ruben_miles}. For a collection of indices $I = \{i_1,\ldots, i_{k+1}\}\subseteq [n]= \{1,\ldots, n\}$ of size $k+1 \leq \min (d,n-1)$, we consider $F_I := [X_{i_1},\ldots, X_{i_{k+1}}]$, a possible $k$-dimensional face of $\sP$, and $A_I = \aff (X_{i_1},\dots,X_{i_{k+1}})$, the affine subspace spanned by $F_I$. Let $h(A_I)$ be the distance from $A_I$ to the origin. For each $j \in I$, we set
\begin{align*}
	Z_{i_j} \coloneqq  \frac{I_{A_I}(X_{i_j})}{\sqrt{1-h^2(A_I)}} \in \BB^k.
\end{align*}
Let $\tilde{F}_I = [Z_{i_1},\dots,Z_{i_{k+1}}]\subseteq \R^k$ be the full-dimensional simplex spanned by these points.
To prove the theorem, we start by using the $\delta$-homogeneity of $\varphi$,
	\begin{align*}
		\E S_{k;\varphi,\psi} (\sP) =& \sum_{\substack{I \subseteq [n] \\ \# I = k+1}} \E \left[\varphi(F_I)   \psi(T(F_I,\sP)) \ind_{\{F_I \text{ is a face of }\sP\}}\right]. \\
		=& \sum_{\substack{I \subseteq [n] \\ \# I = k+1}} \E \left[\varphi(\tilde{F}_I) (1-h^2(A_I))^{\delta/2}   \psi(T(F_I,\sP)) \ind_{\{F_I \text{ is a face of }\sP\}}\right].
	\end{align*}
	In order to separate the expectations, we use the canonical decomposition, Theorem~\ref{theo:ruben_miles}. It states that $\tilde{F}_I$ is independent of $A_I$. Moreover, from the construction used in Theorem~\ref{theo:tangent_cone_is_beta_cone} we see that $T(F_I,\sP)$ is a direct orthogonal sum of its lineality space $A_I - 0_{A_I}$ and a beta cone which is stochastically independent of $\tilde F_I$. Note, however, that $T(F_I,\sP)$ is \textit{not} independent of $A_I$. At this point, we therefore only get
	\begin{align}\label{eq:exp_gen_T}
		\E S_{k;\varphi,\psi} (\sP) = \sum_{\substack{I \subseteq [n] \\ \# I = k+1}} \E \left[\varphi(\tilde{F}_I)\right] \E \left[ \left(1-h^2(A_I)\right)^{\delta/2}  \psi(T(F_I,\sP)) \ind_{\{F_I \text{ is a face of }\sP\}}\right].
	\end{align}
To proceed, we need a characterization of $T(F_I,\sP)$. Theorem~\ref{theo:tangent_cone_is_beta_cone} states that $T(F_I,\sP)$ is isometric to $\R^k \oplus \cD$, where
	\begin{align*}
		\cD = \pos (Y_{i}-Y: i \in [n]\backslash I)\sim  \BetaCone\left(\R^{d-k};\sum_{i \in I} \left(\b_i+\frac{d}{2}\right)  -\frac{d-k}{2};\left\{ \b_{i}+\frac{k}{2}: i \in [n] \backslash I\right\}\right).
	\end{align*}
Recall that if $F_I$ is not a face, then $T(F_I,\sP) = \R^{d-k}$. For notational simplicity, we set $\b' = \sum_{i \in I}\left(\b_i+\frac{d}{2}\right)-\frac{d-k}{2}$. Let us now look at the second expectation stated on the right-hand side of~\eqref{eq:exp_gen_T},
	\begin{align*}
		&\E \left[ \left(1-h^2(A_I)\right)^{\delta/2} \cdot \psi\left(T(F_I,\sP)\right) \ind_{\{F_I \text{ is a face of }\sP\}}\right] \\
		& \qquad \qquad = \E \left[ \left(1-\| Y\|^2\right)^{\delta/2} \cdot \psi\left(\R^k \oplus \pos (Y_{i}-Y: i \in [n]\backslash I)\right) \ind_{\{\pos (Y_{i}-Y: i \in [n]\backslash I) \neq \R^{d-k}\}}\right] \\
		&\qquad \qquad= \int_{\BB^{d-k}}\int_{\left(\BB^{d-k}\right)^{n-k-1}}  (1-\|y\|^2)^{\delta/2}\psi\left(\R^k \oplus \pos (y_{i}-y: i \in [n]\backslash I)\right)\ind_{\{\pos (y_{i}-y: i \in [n]\backslash I) \neq \R^{d-k}\}} \\
		&\qquad \qquad\qquad \qquad\times f_{d-k,\b'}(y) \dd y \prod_{i \in [n]\backslash I} \left(f_{d-k,\b_{i}+\frac{k}{2}}(y_i)  \dd y_i\right) \\
		&\qquad \qquad= \int_{\BB^{d-k}} \int_{\left(\BB^{d-k}\right)^{n-k-1}} \frac{c_{d-k,\b'}}{c_{d-k,\b'+\frac{\delta}{2}}}\psi\left(\R^k \oplus \pos (y_{i}-y: i \in [n]\backslash I)\right)\ind_{\{\pos (y_{i}-y: i \in [n]\backslash I) \neq \R^{d-k}\}}\\
		&\qquad \qquad\qquad \qquad\times f_{d-k,\b'+\frac{\delta}{2}}(y)\dd y \prod_{i \in [n]\backslash I} \left(f_{d-k,\b_{i}+\frac{k}{2}}(y_i)  \dd y_i \right)\\
		&\qquad \qquad= \frac{c_{d-k,\b'}}{c_{d-k,\b'+\frac{\delta}{2}}}
		\E \left[  \psi\left(\R^{k} \oplus \pos(Y_{i}-Y^*: i \in [n]\backslash I)\right) \ind_{\{\pos (Y_{i}-Y^*: i \in [n]\backslash I) \neq \R^{d-k}\}}\right] \\
		&\qquad \qquad= \frac{c_{d-k,\b'}}{c_{d-k,\b'+\frac{\delta}{2}}}
		\E \left[  \psi\left(\R^{k} \oplus \cD_{\delta,I}\right) \ind_{\{\cD_{\delta,I} \neq \R^{d-k}\}}\right],
	\end{align*}
	where $Y^* \sim f_{d-k,\b'+\delta/2}$ is independent of the collection $(Y_i:i \in [n]\backslash I)$ and the cone $\cD_{\delta,I}:= \pos(Y_{i}-Y^*: i \in [n]\backslash I)$ satisfies
	\begin{align*}
		\cD_{\delta,I} \sim
		\BetaCone\left(\R^{d-k}; \sum_{i \in I} \left(\b_i + \frac{d}{2}\right) - \frac{d-\delta-k}{2}; \left\{\b_i + \frac{k}{2}: i \in [n] \backslash I\right\}\right).
	\end{align*}
	Next, we can make the constants explicit to get
	\begin{align*}
		\frac{c_{d-k,\b'}}{c_{d-k,\b'+\frac{\delta}{2}}} = \frac{\Gamma \left(1-\frac{d-k-\delta}{2}+\sum_{i \in I}\left(\b_i+\frac{d}{2}\right)\right)\Gamma \left(1+\sum_{i \in I}\left(\b_i+\frac{d}{2}\right)\right)}{\Gamma \left(1-\frac{d-k}{2}+\sum_{i \in I}\left(\b_i+\frac{d}{2}\right)\right)\Gamma \left(1+\frac{\delta}{2}+\sum_{i \in I}\left(\b_i+\frac{d}{2}\right)\right)}.
	\end{align*}
We notice that the above equals $\E [(1-h^2(A_I))^{\delta/2}]$, as $h^2(A_I) \sim  \Beta \left(\frac{d-k}{2},\frac{k(d+1)}{2}+1+\sum_{i \in I}\b_i\right)$ by Corollary 4.33 in \cite{kabluchko_steigenberger_thaele} or Section 4 (II) in \cite{ruben_miles}. Therefore, in total, we get
\begin{align*}
	\E S_{k;\varphi,\psi} (\sP) =& \sum_{\substack{I \subseteq [n] \\ \# I = k+1}} \E \left[\varphi(\tilde{F}_I)\right] \E \left[ (1-h^2(A_I))^{\delta/2} \right] \E \left[  \psi\left(\R^{k} \oplus \cD_{\delta,I}\right) \ind_{\{\cD_{\delta,I} \neq \R^{d-k}\}}\right] \\
	=& \sum_{\substack{I \subseteq [n] \\ \# I = k+1}} \E \left[\varphi(F_I)\right] \E \left[  \psi\left(\R^{k} \oplus \cD_{\delta,I}\right) \ind_{\{\cD_{\delta,I} \neq \R^{d-k}\}}\right],
\end{align*}
where in the last step we have used the independence of $\tilde F_I$ and $A_I$ as well as the  homogeneity of $\varphi$ to arrive at the original face $F_I=[X_{i_{1}},\dots,X_{i_{k+1}}]$.
\end{proof}

\begin{example}[Again the expected $f$-vector]
	The expected $f$-vector of a beta polytope has been made explicit in Theorem~\ref{theo:f_vect}. It can be recovered using the $S$-functional with  $\varphi=\psi=1$ and $\delta = 0$. Then,  $S_{k-1;1,1}$ counts the number of $(k-1)$-dimensional faces, and for its expectation we get
\begin{align*}
\E f_{k-1}(\sP) = \E S_{k-1;1,1} (\sP)
&=
\sum_{\substack{I \subseteq [n] \\ \# I = k}} \P [\cD_{0,I} \neq \R^{d-k+1}]
\\
&=
2 \sum_{\substack{I \subseteq [n] \\ \# I = k}} \sum_{\substack{J \subseteq I^c \\ \# J \in \{d-k,d-k-2,\dots\}}} \Theta \left(\sum_{i\in I }\g_i;\{\g_i: i \in J\};\{\g_i: i \in I^c \bsl J\}\right),
\end{align*}
where in the last step we used~\eqref{eq:theo:beta_cones_conic_intrinsic_vol_as_A_B_2} of Theorem~\ref{theo:beta_cones_conic_intrinsic_vol_as_A_B}. This recovers Theorem~\ref{theo:f_vect} as stated in~\eqref{eq:theo:f_vect_proof}.
\end{example}

\begin{remark}
The \emph{$T$-functional} has been defined by Wieacker in \cite{wieacker} as follows: For $a,b \in \R$ and a polytope $P \subseteq \R^d$
\begin{align*}
T_{a,b}^{d,k} (P) = \sum_{F \in \cF_k(P)} h^a(F)\Vol_k^b(F),
\end{align*}
where $h(F)$ is the distance from the origin to $\aff (F)$. The proof of Theorem~\ref{theo:expected_S_functional} can be adapted to $T$-functionals, but one faces the following difficulty: the density of $Y^*$ is not a beta density anymore since it contains an  additional factor of the form $\|y\|^a$. For $k=d-1$, formulas for $\E T_{a,b}^{d,d-1}(\sP_{n,d}^{\b_1,\dots,\b_n})$ can be found in~\cite{moseeva2024mixedrandombetapolytopes} and~\cite[Section~8.3]{kabluchko_steigenberger_thaele}.
\end{remark}

\subsection{\texorpdfstring{$S$}{S}-functional: Examples}\label{subsec:examples_S_functional}
Theorem~\ref{theo:expected_S_functional} will now be used to give formulas for expected values of several classical functionals applied to beta polytopes.
In the following proofs, we shall often use the beta cone that first appeared in the proof of Theorem~\ref{theo:expected_S_functional}. Let us recall the definition of said cone.   For $d\geq 2$,  parameters $\b_1,\dots, \b_n  \geq -1$, $\delta > 0$, and a collection of indices $I = \{i_1,\ldots, i_{k+1}\}\subseteq \{1,\ldots, n\}$ of size $k+1 \leq \min (d, n-1)$, we set
\begin{align*}
	\cD_{\delta,I} \sim  %\R^{k+1} \oplus
	\BetaCone\left(\R^{d-k}; \sum_{i \in I} \left(\b_i + \frac{d}{2}\right) - \frac{d-\delta-k}{2}; \left\{\b_i + \frac{k}{2}: i \in [n] \backslash I\right\}\right).
\end{align*}
\subsubsection{Intrinsic volumes}
We start with the $k$-th intrinsic volume which, for a polytope $P \subseteq \R^d$, can be defined as
\begin{align*}
V_k (P) = \sum_{F \in \cF_k(P)} \Vol_k (F) \cdot  \a (N(F,P)),
\qquad k \in \lbrace 0, \dots, d\rbrace.
\end{align*}
For example, assuming $\dim P = d$, we have  $V_d(P) = \Vol_d(P)$, $V_{d-1}(P)$ is half the surface area of $P$, and $V_0(P) = 1$.
\begin{theorem}[Expected intrinsic volumes]\label{theo:intrinsic_volume_of_BetaPoly}
	Let $\sP=\sP_{n,d}^{\b_1,\dots,\b_n}$ be a beta polytope. For $k \in \lbrace 0, \dots, \min (d,n-1)-1\rbrace$, the expected $k$-th intrinsic volume of $\sP$ is
	\begin{align*}
		\E V_k (\sP) = 2  \binom{d}{k} \frac{\k_d}{\k_{d-k}} \sum_{\substack{I \subseteq \{1,\ldots, n\} \\ \# I = k+1}}  \Theta \left(\frac{k}{2};  \lbrace \gamma_i : i \in I \rbrace;\left\lbrace \g_{i}: i \in [n] \backslash I \right\rbrace\right),
	\end{align*}
where we recall that  $\g_i \coloneqq \b_i + d/2$ for each $i \in \lbrace 1, \dots,n\rbrace$.
\end{theorem}
\begin{proof}
	We can write the $k$-th intrinsic volume as an $S$-functional. Set $\varphi= \Vol_k$ and $\psi(C) = \a (C^\circ)$ for a polyhedral cone $C$. In particular, $\Vol_k$ is $k$-homogeneous, which means $\delta = k$. Then, Theorem~\ref{theo:expected_S_functional} gives us that
	\begin{align}\label{eq:intrinsic_volume_as_S}
		\E V_k (\sP)
		&= \sum_{\substack{I \subseteq \{1,\ldots, n\} \\ \# I = k+1}}
		\E \left[\Vol_k (F_I) \right]\E \left[  \a\left((\R^k \oplus \cD_{k,I})^\circ\right) \ind_{\{\cD_{k,I} \neq \R^{d-k}\}}\right].
	\end{align}
Let us first investigate the expected volume of $\Vol_k (F_I)$. For this, note that we have shown in Remark~\ref{rem:volume_as_theta_from_book} that
\begin{align*}
	(d!)^{-1} \frac{\Gamma \left( 1 +\frac{d+1}{2} + \sum_{i=1}^{d+1}  \g_i \right) }{\Gamma \left( 1 +\frac{d}{2} + \sum_{i=1}^{d+1}  \g_i \right)} \frac{\Gamma \left( \frac{d+1 }{2} \right)}{\Gamma \left( \frac{1}{2} \right)}\prod_{i=1}^{d+1} \frac{\Gamma \left( \g_i +1 \right)}{\Gamma \left( \g_i  +\frac{3}{2} \right)}=2\kappa_d \cdot \Theta \left(\frac{d}{2}; \gamma_1,\dots,\g_{d+1}; \varnothing\right).
\end{align*}
Furthermore, for any $I \subseteq [n]$ with $\# I = k+1$, the explicit formula for the expected volume $F_I$, which is a non-full dimensional beta simplex in $\R^d$, is
\begin{align*}
	\E \left[ \Vol_k (F_I) \right] =& (k!)^{-1}\frac{\Gamma\left(1+\frac{k+1}{2}+\sum_{i \in I}\g_i\right)}{\Gamma \left(1+\frac{k}{2}+\sum_{i \in I}\g_i \right)}  \frac{\Gamma\left(\frac{d+1}{2}\right)}{\Gamma\left(\frac{d+1-k}{2}\right)} \prod_{i \in I}\frac{\Gamma\left(\g_i+1\right)}{\Gamma\left(\g_i+\frac{3}{2}\right)}
\end{align*}
as can be found in Corollary 6.16 in~\cite{kabluchko_steigenberger_thaele}, as well as in \cite[Sections~7,12]{miles} and~\cite[Corollary on p.~10]{ruben_miles}, where it has been previously shown to hold. Combining these two, we get the following expression,
\begin{align}\label{eq:volume_non_full_as_theta}
	\E \left[ \Vol_k (F_I) \right] = 2 \kappa_k\frac{ \sqrt{\pi}\; \Gamma \left(\frac{d+1}{2}\right)}{\Gamma \left(\frac{k+1}{2}\right)\Gamma \left(\frac{d-k+1}{2}\right)} \Theta \left(\frac{k}{2};  \lbrace \gamma_i : i \in I \rbrace ; \varnothing\right).
\end{align}
Next, let us investigate the second term in the sum on the right-hand side of~\eqref{eq:intrinsic_volume_as_S}. On the event $\{\cD_{k,I} \neq \R^{d-k}\}$, the lineality space of $\cD_{k,I}$ equals $\{0\}$ by Lemma~\ref{lem:beta_cones_generic_properties}. On this event,  Lemma~\ref{lem:cone_intrinsic_volume_lineality_space} gives that,
\begin{align*}
	\a \left((\R^k \oplus \cD_{k,I})^\circ\right) =\upsilon_{k} \left(\R^k \oplus \cD_{k,I}\right)= \upsilon_0(\cD_{k,I})	
\end{align*}
and since $\E [\upsilon_0(\cD_{k,I})\ind_{\{\cD_{k,I} = \R^{d-k}\}}]=0$, we can directly use Equation~\eqref{eq:beta_cones_conic_intrinsic_vol_as_A_B}
to get
	\begin{align*}
		\E \left[  \a \left((\R^k \oplus \cD_{k,I})^\circ\right) \ind_{\{\cD_{k,I} \neq \R^{d-k}\}}\right] = \E [\upsilon_0(\cD_{k,I})] = \Theta \left(\frac{k}{2}+\sum_{i \in I} \g_i; \varnothing;\left\lbrace \g_{i}: i \in [n] \backslash I \right\rbrace\right).
	\end{align*}
	Thus, we have a product of two $\Theta$-functions on the right-hand side of~\eqref{eq:intrinsic_volume_as_S}. We can use Theorem~\ref{theo:theta_as_prod_of_thetas} to combine them into one. This then gives
	\begin{align*}
		\E V_k (\sP) = 2 \kappa_k \frac{\sqrt{\pi}\; \Gamma \left(\frac{d+1}{2}\right)}{\Gamma \left(\frac{k+1}{2}\right)\Gamma \left(\frac{d-k+1}{2}\right)} \sum_{\substack{I \subseteq [n] \\ \# I = k+1}}  \Theta \left(\frac{k}{2};  \lbrace \gamma_i : i \in I \rbrace;\left\lbrace \g_{i}: i \in [n] \backslash I \right\rbrace\right).
	\end{align*}
	Finally, using the Legendre duplication formula three times, one sees that
	\begin{align*}
		\binom{d}{k} \frac{\k_d}{\k_k \k_{d-k}} = \frac{\sqrt{\pi}\; \Gamma \left(\frac{d+1}{2}\right)}{\Gamma \left(\frac{k+1}{2}\right)\Gamma \left(\frac{d-k+1}{2}\right)}
	\end{align*}
	This gives the claim.
\end{proof}

\begin{remark}[Again the expected volume]
Note that $V_d(\sP) = \Vol_d(\sP)$ and let $n\geq d+1$. Although the setting of Theorem~\ref{theo:intrinsic_volume_of_BetaPoly} excludes the case $k=d$ (and its proof breaks down for $k=d$), it should be noted that the formula for the expected volume of $\sP$ given in Equation~\eqref{eq:E_Vol_beta_poly_1} matches what one would obtain by formally putting $k=d$ in Theorem~\ref{theo:intrinsic_volume_of_BetaPoly}. Furthermore, formally setting $k=d$ in Equation~\eqref{eq:volume_non_full_as_theta}, we get the expression for the expected volume of a full-dimensional beta simplex that was obtained in Example~\ref{ex:vol_beta_simplex_as_theta}.
\end{remark}
\begin{remark}[Comparing formulas for expected intrinsic volumes]
In Proposition~8.40 in \cite{kabluchko_steigenberger_thaele}, the following relation between expected volumes and expected intrinsic volumes of beta polytopes has been shown. For $n,d \in \N$ and $n \geq d+1$, it holds that
\begin{align}\label{eq:kubota}
	\E V_k(\sP_{n,d}^{\beta_1,\ldots, \beta_n})&=\binom{d}{k} \frac{\k_d}{\k_k \k_{d-k}} \E \Vol_k\left( \sP_{n,k}^{\beta_1 + \frac{d-k}{2},\ldots, \beta_n + \frac{d-k}{2}}\right),
\end{align}
where $k \in \lbrace 1,\dots,d\rbrace$. This equation has been shown using Kubota's formula together with Fubini's theorem. One sees directly that we can also recover it by starting with Theorem~\ref{theo:intrinsic_volume_of_BetaPoly},
\begin{align*}
	\E V_k (\sP) = 2\binom{d}{k} \frac{\k_d}{ \k_{d-k}}   \sum_{\substack{I \subseteq \{1,\ldots, n\} \\ \# I = k+1}}  \Theta \left(\frac{k}{2};  \left\lbrace \b_i + \frac{d}{2} : i \in I \right\rbrace;\left\lbrace \b_i + \frac{d}{2}: i \in I^c \right\rbrace\right),
\end{align*}
and then applying Theorem~\ref{theo:exp_volume_beta_polytope} with $k=d$ on the $\Theta$-term on the right-hand side.
\end{remark}

\subsubsection{Generalized angle sums}
As one more application of the $S$-functional, we compute expected sums of conic intrinsic volumes of tangent cones at all $k$-faces of a beta polytope.  As a special case, these include sums of internal and external angles at faces of a given dimension of a beta polytope.
\begin{theorem}[Expected conic intrinsic volumes sums of  beta polytopes]\label{theo:conic_volume_sums}
For $0\leq k \leq \ell \leq \min (d,n-1)-1$, we can express the sum of conic intrinsic volumes of $\sP$ as
	\begin{align}
		\sum_{F \in \cF_{k}(\sP)} \E [\upsilon_{\ell} (T(F,\sP))] 	= \sum_{\substack{I \subseteq \{1,\dots,n\} \\ \#I = k+1}} \sum_{\substack{J \subseteq I^c \\ \# J=\ell -k}}\Theta \left( \sum_{i \in I } \g_i;\left\{ \g_i : i\in J\right\};\left\{ \g_i : i\in I^c\bsl J\right\}\right). \label{eq:theo:conic_volume_sums_1}
	\end{align}
If instead $k\in \{0,\ldots,\min (d,n-1)-1\}$ and $\ell =\min (d,n-1)$, we have
\begin{align}
	\sum_{F \in \cF_{k}(\sP)} \E [\upsilon_{\ell} (T(F,\sP))] = \sum_{\substack{I \subseteq \{1,\dots,n\} \\ \#I = k+1}} \sum_{\substack{J \subseteq I^c \\ \# J \leq \ell-k-1}} (-1)^{\ell-k-1-\#J} \Theta \left( \sum_{i \in I} \g_i;\left\{ \g_i : i\in J\right\};\left\{ \g_i : i\in I^c \bsl J\right\}\right). \label{eq:theo:conic_volume_sums_2}
\end{align}
\end{theorem}
\begin{proof}
	We use the expected $S$-functional for $\varphi = 1$, which gives $\delta = 0$, and $\psi = \upsilon_\ell$. Then, by Theorem~\ref{theo:expected_S_functional},
	\begin{align*}
	\sum_{F \in \cF_k(\sP)} \E \left[\upsilon_\ell (T(F,\sP)) \right]&= \E S_{k;1,\upsilon_\ell} (\sP) = \sum_{\substack{I \subseteq [n] \\ \# I = k+1}} \E \left[  \upsilon_\ell \left(\R^{k} \oplus \cD_{0,I}\right) \ind_{\{\cD_{0,I} \neq \R^{d-k}\}}\right] \\
		&= \sum_{\substack{I \subseteq \{1,\dots,n\} \\ \# I = k+1}} \E \left[  \upsilon_{\ell-k} \left(\cD_{0,I}\right) \ind_{\{\cD_{0,I} \neq \R^{d-k}\}}\right] \\
		&= \sum_{\substack{I \subseteq \{1,\dots,n\} \\ \# I = k+1}} \E \left[  \upsilon_{\ell-k} \left(\cD_{0,I}\right) \right] \\
		&= \sum_{\substack{I \subseteq \{1,\dots,n\} \\ \#I = k+1}} \sum_{\substack{J \subseteq I^c \\ \# J=\ell -k}}\Theta \left( \sum_{i \in I } \g_i;\left\{ \g_i : i\in J\right\};\left\{ \g_i : i\in I^c\bsl J\right\}\right),
	\end{align*}
	where the second to last equation holds since $\ell \neq \min (d,n-1)$ and we applied Equation~\eqref{eq:beta_cones_conic_intrinsic_vol_as_A_B} of Theorem~\ref{theo:beta_cones_conic_intrinsic_vol_as_A_B} (and Theorem~\ref{theo:beta_cones_conic_intrinsic_vol_as_A_B_non_full_dim}) in the last equation. This proves~\eqref{eq:theo:conic_volume_sums_1}.
For the proof of~\eqref{eq:theo:conic_volume_sums_2}, a similar argument can be used, but this time, we stop at $\E[  \upsilon_{\ell-k}(\cD_{0,I}) \ind_{\{\cD_{0,I} \neq \R^{d-k}\}}]$. If  $n\geq d+1$, we evaluate it using Equation~\eqref{eq:upislon_R_d_ind_BetaCone} of Theorem~\ref{theo:beta_cones_conic_intrinsic_vol_as_A_B} which we can apply since $\ell-k =\min (d,n-1)-k = \dim \mathcal{D}_{0,I}$. If $n\leq d$, we use Theorem~\ref{theo:beta_cones_conic_intrinsic_vol_as_A_B_non_full_dim} and observe that  the alternating sum of $\Theta$-functions appearing in~\eqref{eq:theo:conic_volume_sums_2} simplifies using~\eqref{prop:relation_theta_even_odd}.
\end{proof}
It is also possible to prove Theorem~\ref{theo:conic_volume_sums} by using Equations~\eqref{eq:exp_ups_tangent_beta_poly} and~\eqref{eq:exp_alpha_tangent_beta_poly_on_event} of Theorem~\ref{theo:probabily_that_F_is_face} and taking the sum over all faces.

\subsubsection{\texorpdfstring{$k$}{k}-volume of the \texorpdfstring{$k$}{k}-skeleton}
%We turn to the next special case of the $S$-functional.
For a polytope $P \subseteq \R^d$, the \emph{$k$-skeleton} of $P$ (with $k \in \lbrace 0,\dots,\dim P-1\rbrace$) and its \emph{$L^p$-volume} (with $p\geq 0$) are defined to be
\begin{align*}
\skel_k P = \bigcup_{F \in \cF_k(P)} F
\quad \text{ and } \quad
\Vol_{k; L^p} (\skel_k P) = \sum_{F\in \cF_k(P)} (\Vol_k F)^p.
\end{align*}
%where $k \in \lbrace 0,\dots,d\rbrace$.
\begin{theorem}\label{theo:expected_volume_of_skeleton}
For $k \in \lbrace 0,\dots,\min (d,n-1)-1\rbrace$ and $p\geq 0$, the expected $L^p$-volume of the $k$-dimensional skeleton of a beta polytope $\sP$ is given by
	\begin{multline*}
		\E \left[\Vol_{k; L^p} (\skel_k \sP) \right]
\\= 2\sum_{\substack{I \subseteq \{1,\ldots, n\} \\ \# I = k+1}} 	\E \left[\Vol_k^p (F_I) \right]  \sum_{\substack{J \subseteq I^c \\ \# J \in \{d-k-1,d-k-3,\dots\}}} \Theta \left(\frac{kp}{2}+\sum_{i\in I }\g_i;\{\g_i: i \in J\};\{\g_i: i \in I^c \bsl J\}\right),
	\end{multline*}
where we set $\g_i \coloneqq \b_i + d/2$ for $i\in \{1,\ldots, n\}$ and
	\begin{align*}
		\E [\Vol_{k}^{p}(F_I)] = (k!)^{-p} \frac{\Gamma\left(\frac{(k+1)p}{2} +1+ \sum_{i \in I} \g_i \right)}{\Gamma\left(\frac{kp}{2} + 1 + \sum_{i \in I} \g_i \right)} \prod_{i \in I} \frac{\Gamma\Big(\g_i  +1\Big)}{\Gamma\Big(  \g_i+\frac {p}2 + 1 \Big)}\prod_{i=1}^{k} \frac{\Gamma\Big(\frac{d+1+p-i}{2} \Big)}{\Gamma\Big(\frac{d+1-i}{2}\Big)}.
	\end{align*}
\end{theorem}
\begin{proof}
	We use the $S$-functional with $\varphi = \Vol_k^p$, which is a $kp$-homogeneous functional, and $\psi=1$,
	\begin{align*}
		\E \left[\Vol_{k; L^p} (\skel_k \sP) \right]=		\E S_{k;\Vol_k^p,1} (\sP) &= \sum_{\substack{I \subseteq \{1,\ldots, n\} \\ \# I = k+1}}
		\E \left[\Vol_k^p (F_I) \right] \P [\cD_{kp,I} \neq \R^{d-k}],
	\end{align*}
	where the probability of the event $\lbrace \cD_{kp,I} \neq \R^{d-k} \rbrace$ can be made explicit using Theorem~\ref{theo:beta_cones_conic_intrinsic_vol_as_A_B}.
Finally, a formula for the $p$-th moment of the volume of the  beta simplex $F_I$ appearing in the above theorem can be found in Corollary~6.16 in~\cite{kabluchko_steigenberger_thaele}, as well as in \cite[Sections~7,12]{miles} and~\cite[Corollary on p.~10]{ruben_miles}, where it was originally derived.
\end{proof}
\begin{example}
Setting $p=0$ we recover the formula for $\E f_k(\sP)$ obtained in Theorem~\ref{theo:f_vect}, see~\eqref{eq:theo:f_vect_proof}.
\end{example}
\begin{example}[Expected $k$-volume of a $k$-skeleton]
The $k$-dimensional volume of the $k$-dimensional skeleton of a polytope $P$ is defined as $\Vol_k (\skel_k(P)) =  \sum_{F \in \cF_k(P)} \Vol_k(F)$; see, e.g., Chapter 10 of~\cite{schneider_weil_book}.  Setting $p=1$ in Theorem~\ref{theo:expected_volume_of_skeleton} gives an explicit formula for $\E \Vol_k (\skel_k(\sP))$.
In particular, by further setting $k=d-1$, we get the expected surface area of $\sP$, which has also been calculated in Corollary~8.34 of~\cite{kabluchko_steigenberger_thaele}, as well as in Theorem~2.2 in~\cite{moseeva2024mixedrandombetapolytopes} (where one takes $a=0$ and $b=1$).
\end{example}

\section*{Acknowledgement}
Supported by the German Research Foundation under Germany’s Excellence Strategy EXC 2044--390685587, Mathematics
Münster: Dynamics -- Geometry -- Structure and by the DFG priority program SPP 2265 Random Geometric Systems.

\bibliography{beta_poly_bib}

\begin{thebibliography}{43}
\providecommand{\natexlab}[1]{#1}
\providecommand{\url}[1]{\texttt{#1}}
\expandafter\ifx\csname urlstyle\endcsname\relax
  \providecommand{\doi}[1]{doi: #1}\else
  \providecommand{\doi}{doi: \begingroup \urlstyle{rm}\Url}\fi

\bibitem[Amelunxen(2011)]{amelunxen_phd}
D.~Amelunxen.
\newblock {Geometric analysis of the condition of the convex feasibility
  problem}.
\newblock PhD Thesis, University of Paderborn. Available at:
  \url{https://digital.ub.uni-paderborn.de/hsx/content/titleinfo/34332}, 2011.

\bibitem[Amelunxen and Lotz(2017)]{amelunxen_lotz}
D.~Amelunxen and M.~Lotz.
\newblock Intrinsic volumes of polyhedral cones: a combinatorial perspective.
\newblock \emph{Discrete Comput. Geom.}, 58\penalty0 (2):\penalty0 371--409,
  2017.

\bibitem[Amelunxen et~al.(2014)Amelunxen, Lotz, McCoy, and
  Tropp]{amelunxen_lotz_mccoy_tropp_living_on_the_edge}
D.~Amelunxen, M.~Lotz, M.~B. McCoy, and J.~A. Tropp.
\newblock Living on the edge: phase transitions in convex programs with random
  data.
\newblock \emph{Inf. Inference}, 3\penalty0 (3):\penalty0 224--294, 2014.

\bibitem[Barvinok(2002)]{barvinok_book}
A.~Barvinok.
\newblock \emph{A course in convexity}, volume~54 of \emph{Graduate Studies in
  Mathematics}.
\newblock American Mathematical Society, Providence, RI, 2002.
\newblock \doi{10.1090/gsm/054}.
\newblock URL \url{https://doi.org/10.1090/gsm/054}.

\bibitem[Bonnet et~al.(2017)Bonnet, Grote, Temesvari, Th\"ale, Turchi, and
  Wespi]{bonnet_etal}
G.~Bonnet, J.~Grote, D.~Temesvari, C.~Th\"ale, N.~Turchi, and F.~Wespi.
\newblock Monotonicity of facet numbers of random convex hulls.
\newblock \emph{J. Math. Anal. Appl.}, 455\penalty0 (2):\penalty0 1351--1364,
  2017.

\bibitem[Bonnet et~al.(2019)Bonnet, Chasapis, Grote, Temesvari, and
  Turchi]{BonnetEtAlThresholds}
G.~Bonnet, G.~Chasapis, J.~Grote, D.~Temesvari, and N.~Turchi.
\newblock Threshold phenomena for high-dimensional random polytopes.
\newblock \emph{Commun. Contemp. Math.}, 21\penalty0 (5):\penalty0 1850038, 30,
  2019.
\newblock \doi{10.1142/S0219199718500384}.
\newblock URL \url{https://doi.org/10.1142/S0219199718500384}.

\bibitem[Bonnet et~al.(2021)Bonnet, Kabluchko, and
  Turchi]{bonnet_kabluchko_turchi}
G.~Bonnet, Z.~Kabluchko, and N.~Turchi.
\newblock Phase transition for the volume of high-dimensional random polytopes.
\newblock \emph{Random Structures Algorithms}, 58\penalty0 (4):\penalty0
  648--663, 2021.
\newblock \doi{10.1002/rsa.20986}.
\newblock URL \url{https://doi.org/10.1002/rsa.20986}.

\bibitem[Br{\o}ndsted(1983)]{broendsted_book}
A.~Br{\o}ndsted.
\newblock \emph{An introduction to convex polytopes}, volume~90 of
  \emph{Graduate Texts in Mathematics}.
\newblock Springer-Verlag, New York-Berlin, 1983.

\bibitem[{Buchta}(2005)]{buchta}
C.~{Buchta}.
\newblock {An identity relating moments of functionals of convex hulls.}
\newblock \emph{{Discrete Comput. Geom.}}, 33\penalty0 (1):\penalty0 125--142,
  2005.
\newblock \doi{10.1007/s00454-004-1109-3}.

\bibitem[Calka and Quilan(2023)]{calka_quilan_first_layers}
P.~Calka and G.~Quilan.
\newblock Limit theory for the first layers of the random convex hull peeling
  in the unit ball.
\newblock \emph{Probab. Theory Related Fields}, 187\penalty0 (3-4):\penalty0
  1037--1091, 2023.
\newblock \doi{10.1007/s00440-023-01224-6}.
\newblock URL \url{https://doi.org/10.1007/s00440-023-01224-6}.

\bibitem[Calka and Yukich(2014)]{calka_yukich_variance_asympt}
P.~Calka and J.~E. Yukich.
\newblock Variance asymptotics for random polytopes in smooth convex bodies.
\newblock \emph{Probab. Theory Related Fields}, 158\penalty0 (1-2):\penalty0
  435--463, 2014.
\newblock \doi{10.1007/s00440-013-0484-1}.
\newblock URL \url{https://doi.org/10.1007/s00440-013-0484-1}.

\bibitem[Calka et~al.(2013)Calka, Schreiber, and
  Yukich]{calka_schreiber_yukich_brownian_limits_local_limits}
P.~Calka, T.~Schreiber, and J.~E. Yukich.
\newblock Brownian limits, local limits and variance asymptotics for convex
  hulls in the ball.
\newblock \emph{Ann. Probab.}, 41\penalty0 (1):\penalty0 50--108, 2013.
\newblock \doi{10.1214/11-AOP707}.
\newblock URL \url{https://doi.org/10.1214/11-AOP707}.

\bibitem[{Donoho} and {Tanner}(2010)]{donoho_tanner1}
D.~L. {Donoho} and J.~{Tanner}.
\newblock {Counting the faces of randomly-projected hypercubes and orthants,
  with applications.}
\newblock \emph{{Discrete Comput. Geom.}}, 43\penalty0 (3):\penalty0 522--541,
  2010.
\newblock \doi{10.1007/s00454-009-9221-z}.

\bibitem[Efron(1965)]{Efron-identity}
B.~Efron.
\newblock The convex hull of a random set of points.
\newblock \emph{Biometrika}, 52\penalty0 (3/4):\penalty0 331--343, 1965.

\bibitem[Feldman and Klain(2009)]{Feldman01102009}
D.~V. Feldman and D.~A. Klain.
\newblock Angles as probabilities.
\newblock \emph{The American Mathematical Monthly}, 116\penalty0 (8):\penalty0
  732--735, 2009.
\newblock \doi{10.4169/193009709X460868}.

\bibitem[Flajolet and Sedgewick(2009)]{Flajolet_book}
P.~Flajolet and R.~Sedgewick.
\newblock \emph{Analytic combinatorics}.
\newblock Cambridge University Press, Cambridge, 2009.

\bibitem[Glasauer(1995)]{glasauer_phd}
S.~Glasauer.
\newblock {Integralgeometrie konvexer K\"orper im sph\"arischen Raum}.
\newblock PhD Thesis, University of Freiburg. Available at:
  \url{http://www.hs-augsburg.de/~glasauer/publ/diss.pdf}, 1995.

\bibitem[Godland et~al.(2022)Godland, Kabluchko, and
  Th\"{a}le]{godland_kabluchko_thaele_cones_high_dim_part_I}
T.~Godland, Z.~Kabluchko, and C.~Th\"{a}le.
\newblock Random cones in high dimensions {I}: {D}onoho-{T}anner and
  {C}over-{E}fron cones.
\newblock \emph{Discrete Anal.}, pages Paper No. 5, 44, 2022.
\newblock \doi{10.19086/da}.
\newblock URL \url{https://doi.org/10.19086/da}.

\bibitem[Grote et~al.(2019)Grote, Kabluchko, and Th\"{a}le]{beta_simplices}
J.~Grote, Z.~Kabluchko, and C.~Th\"{a}le.
\newblock Limit theorems for random simplices in high dimensions.
\newblock \emph{ALEA Lat. Am. J. Probab. Math. Stat.}, 16\penalty0
  (1):\penalty0 141--177, 2019.

\bibitem[Gusakova and Kabluchko(2024)]{gusakova_kabluchko_sylvester_beta}
A.~Gusakova and Z.~Kabluchko.
\newblock Sylvester problem for beta distributions, 2024.
\newblock URL \url{https://arxiv.org/abs/2501.00671}.

\bibitem[Hug and Schneider(2016)]{hug_schneider_conical_tessellations}
D.~Hug and R.~Schneider.
\newblock Random conical tessellations.
\newblock \emph{Discrete Comput. Geom.}, 56\penalty0 (2):\penalty0 395--426,
  2016.
\newblock \doi{10.1007/s00454-016-9788-0}.
\newblock URL \url{https://doi.org/10.1007/s00454-016-9788-0}.

\bibitem[Kabluchko(2020{\natexlab{a}})]{kabluchko_angle_sums_dim_3_4}
Z.~Kabluchko.
\newblock Angle sums of random simplices in dimensions 3 and 4.
\newblock \emph{Proc. Amer. Math. Soc.}, 148\penalty0 (7):\penalty0 3079--3086,
  2020{\natexlab{a}}.
\newblock \doi{10.1090/proc/14934}.
\newblock URL \url{https://doi.org/10.1090/proc/14934}.

\bibitem[Kabluchko(2020{\natexlab{b}})]{kabluchko_poisson_zero_cell}
Z.~Kabluchko.
\newblock Expected {$f$}-vector of the {P}oisson zero polytope and random
  convex hulls in the half-sphere.
\newblock \emph{Mathematika}, 66\penalty0 (4):\penalty0 1028--1053,
  2020{\natexlab{b}}.
\newblock \doi{10.1112/mtk.12056}.
\newblock URL \url{https://doi.org/10.1112/mtk.12056}.

\bibitem[Kabluchko(2021{\natexlab{a}})]{kabluchko_angles_of_random_simplices_and_face_numbers}
Z.~Kabluchko.
\newblock Angles of random simplices and face numbers of random polytopes.
\newblock \emph{Adv. Math.}, 380:\penalty0 Paper No. 107612, 68,
  2021{\natexlab{a}}.
\newblock \doi{10.1016/j.aim.2021.107612}.
\newblock URL \url{https://doi.org/10.1016/j.aim.2021.107612}.

\bibitem[Kabluchko(2021{\natexlab{b}})]{kabluchko_recursive_scheme}
Z.~Kabluchko.
\newblock Recursive scheme for angles of random simplices, and applications to
  random polytopes.
\newblock \emph{Discrete Comput. Geom.}, 66\penalty0 (3):\penalty0 902--937,
  2021{\natexlab{b}}.
\newblock \doi{10.1007/s00454-020-00259-z}.
\newblock URL \url{https://doi.org/10.1007/s00454-020-00259-z}.

\bibitem[Kabluchko(2023)]{kabluchko_on_expected_face_numbers_of_beta_polytopes}
Z.~Kabluchko.
\newblock On expected face numbers of random beta and beta' polytopes.
\newblock \emph{Beitr. Algebra Geom.}, 64\penalty0 (1):\penalty0 155--174,
  2023.
\newblock \doi{10.1007/s13366-022-00626-2}.
\newblock URL \url{https://doi.org/10.1007/s13366-022-00626-2}.

\bibitem[Kabluchko et~al.(2019)Kabluchko, Temesvari, and
  Th\"{a}le]{beta_polytopes}
Z.~Kabluchko, D.~Temesvari, and C.~Th\"{a}le.
\newblock Expected intrinsic volumes and facet numbers of random
  beta-polytopes.
\newblock \emph{Math. Nachr.}, 292\penalty0 (1):\penalty0 79--105, 2019.
\newblock \doi{10.1002/mana.201700255}.
\newblock URL \url{https://doi.org/10.1002/mana.201700255}.

\bibitem[Kabluchko et~al.(2020)Kabluchko, Th\"{a}le, and
  Zaporozhets]{kabluchko_thaele_zaporozhets_beta_polytopes_poi_polyhedra}
Z.~Kabluchko, C.~Th\"{a}le, and D.~Zaporozhets.
\newblock Beta polytopes and {P}oisson polyhedra: {$f$}-vectors and angles.
\newblock \emph{Adv. Math.}, 374:\penalty0 107333, 63, 2020.
\newblock \doi{10.1016/j.aim.2020.107333}.
\newblock URL \url{https://doi.org/10.1016/j.aim.2020.107333}.

\bibitem[Kabluchko et~al.(2025)Kabluchko, Steigenberger, and
  Th\"{a}le]{kabluchko_steigenberger_thaele}
Z.~Kabluchko, D.~Steigenberger, and C.~Th\"{a}le.
\newblock Volumes of random beta-type simplices.
\newblock 2025.
\newblock Manuscript of a forthcoming book.

\bibitem[McMullen(1975)]{mcmullen}
P.~McMullen.
\newblock Non-linear angle-sum relations for polyhedral cones and polytopes.
\newblock \emph{Math. Proc. Cambridge Philos. Soc.}, 78\penalty0 (2):\penalty0
  247--261, 1975.
\newblock \doi{10.1017/S030500,4100051665}.
\newblock URL \url{https://doi.org/10.1017/S0305004100051665}.

\bibitem[{Miles}(1971)]{miles}
R.~E. {Miles}.
\newblock {Isotropic random simplices.}
\newblock \emph{{Adv. Appl. Probab.}}, 3:\penalty0 353--382, 1971.
\newblock \doi{10.2307/1426176}.

\bibitem[Moseeva(2024)]{moseeva2024mixedrandombetapolytopes}
T.~Moseeva.
\newblock Mixed random beta-polytopes, 2024.
\newblock URL \url{https://arxiv.org/abs/2407.10772}.

\bibitem[Perles and Shepard(1967)]{Perles_Shepard_1967}
M.~A. Perles and G.~C. Shepard.
\newblock Angle sums of convex polytopes.
\newblock \emph{Math. Scand.}, 21:\penalty0 199–218, Dec. 1967.
\newblock \doi{10.7146/math.scand.a-10860}.
\newblock URL \url{https://www.mscand.dk/article/view/10860}.

\bibitem[Reitzner(2005)]{reitzner_clt_for_random_polys}
M.~Reitzner.
\newblock Central limit theorems for random polytopes.
\newblock \emph{Probab. Theory Related Fields}, 133\penalty0 (4):\penalty0
  483--507, 2005.
\newblock \doi{10.1007/s00440-005-0441-8}.
\newblock URL \url{https://doi.org/10.1007/s00440-005-0441-8}.

\bibitem[Ruben(1979)]{ruben_parallelotopes}
H.~Ruben.
\newblock The volume of an isotropic random parallelotope.
\newblock \emph{J. Appl. Probab.}, 16\penalty0 (1):\penalty0 84--94, 1979.
\newblock \doi{10.1017/s0021900200046222}.

\bibitem[{Ruben} and {Miles}(1980)]{ruben_miles}
H.~{Ruben} and R.~E. {Miles}.
\newblock {A canonical decomposition of the probability measure of sets of
  isotropic random points in $\mathbb R^n$.}
\newblock \emph{{J. Multivariate Anal.}}, 10:\penalty0 1--18, 1980.
\newblock \doi{10.1016/0047-259X(80)90077-9}.

\bibitem[Schneider(2014)]{SchneiderBook}
R.~Schneider.
\newblock \emph{Convex {B}odies: the {B}runn-{M}inkowski {T}heory}, volume 151
  of \emph{Encyclopedia of Mathematics and its Applications}.
\newblock Cambridge University Press, Cambridge, expanded edition, 2014.

\bibitem[Schneider(2022)]{schneider_book_convex_cones_probab_geom}
R.~Schneider.
\newblock \emph{Convex cones---geometry and probability}, volume 2319 of
  \emph{Lecture Notes in Mathematics}.
\newblock Springer, Cham, 2022.
\newblock \doi{10.1007/978-3-031-15127-9}.
\newblock URL \url{https://doi.org/10.1007/978-3-031-15127-9}.

\bibitem[Schneider and Weil(2008)]{schneider_weil_book}
R.~Schneider and W.~Weil.
\newblock \emph{Stochastic and integral geometry}.
\newblock Probability and its Applications (New York). Springer-Verlag, Berlin,
  2008.
\newblock \doi{10.1007/978-3-540-78859-1}.
\newblock URL \url{https://doi.org/10.1007/978-3-540-78859-1}.

\bibitem[Shabat(1992)]{shabat_book_2}
B.~V. Shabat.
\newblock \emph{Introduction to complex analysis. {P}art {II}}, volume 110 of
  \emph{Translations of Mathematical Monographs}.
\newblock American Mathematical Society, Providence, RI, 1992.
\newblock \doi{10.1090/mmono/110}.
\newblock URL \url{https://doi.org/10.1090/mmono/110}.
\newblock Functions of several variables, Translated from the third (1985)
  Russian edition by J. S. Joel.

\bibitem[Shephard(1967)]{Shephard_1967}
G.~C. Shephard.
\newblock An elementary proof of {G}ram’s theorem for convex polytopes.
\newblock \emph{Canadian Journal of Mathematics}, 19:\penalty0 1214–1217,
  1967.
\newblock \doi{10.4153/CJM-1967-110-7}.

\bibitem[Wendel(1962)]{Wendel}
J.~G. Wendel.
\newblock A problem in geometric probability.
\newblock \emph{Math. Scand.}, 11:\penalty0 109--111, 1962.

\bibitem[Wieacker(1978)]{wieacker}
J.~A. Wieacker.
\newblock Einige {P}robleme der polyedrischen {A}pproximation.
\newblock Diploma Thesis, University of Freiburg, 1978.

\end{thebibliography}
\bibliographystyle{plainnat}

\vspace{1cm}

\footnotesize

\textsc{Zakhar Kabluchko:} Institut f\"ur Mathematische Stochastik, Universit\"at M\"unster\\
\textit{E-mail}: \texttt{zakhar.kabluchko@uni-muenster.de}

\bigskip

\textsc{David Albert Steigenberger:} Institut f\"ur Mathematische Stochastik, Universit\"at M\"unster\\
\textit{E-mail}: \texttt{davidsteigenberger@uni-muenster.de}

\end{document}